\pdfoutput=1
\documentclass{amsart}
\title
[Weak Greenberg's generalized conjecture]
{Weak Greenberg's generalized conjecture for imaginary quadratic fields
}
\author{Kazuaki Murakami}
\address{Department~of~Mathematical~Sciences,~Graduate~School~of~Science~and~Engineering,~Keio~University,
Hiyoshi,~Kohoku\textrm{-}ku, Yokohama, Kanagawa~223-8522,~Japan}
\date{}
\usepackage{amsmath,amscd,amssymb,amsthm,euscript,ascmac}
\usepackage{amsfonts}
\usepackage{latexsym}
\usepackage{amscd}
\usepackage{amsthm}
\usepackage[dvips]{graphicx}
\usepackage{layout}
\usepackage{array}
\usepackage{tabularx}
\usepackage{multicol}
\usepackage{bm}
\newcommand{\plim}[1][]{\mathop{\varprojlim}\limits_{#1}}

\newtheorem{thm}{Theorem}[section]
\newtheorem{prop}[thm]{Proposition}
\newtheorem{cor}[thm]{Corollary}
\newtheorem{lem}[thm]{Lemma}
\newtheorem{Rem}[thm]{\rm{Remark}}
\newtheorem{Ex}[thm]{Example}
\newtheorem{GGC}{Conjecture}

\newtheorem{WGGC}{Conjecture}

\newtheorem{Defn}{Definition}
\begin{document}
%\subjclass[2010]{Primary 11R23; Secondary 11R11}% Subject code(s)
%\keywords{Iwasawa invariants, $\mathbb{Z}_p$-extensions, $\mathbb{Z}_p^2$-extensions, imaginary quadratic fields}% Key word(s)
\begin{abstract}
Let $p$ be an odd prime number and $k$ an imaginary quadratic field
in which $p$ splits.
%is not  $p$-split $p$-rational.
%We assume that $p$ splits in $k$ into $\mathfrak{p}$ and $\mathfrak{p}^{*}$.
In this paper, we consider a weak form of Greenberg's generalized conjecture for $p$ and $k$,
which states that
%the the maximal unramified pro-$p$ abelian extension field over $\widetilde{k}$.
the non-trivial Iwasawa module of the maximal multiple $\mathbb{Z}_p$-extension field over $k$
has a non-trivial pseudo-null submodule.
%In this paper,
We prove
%a weak form of Greenberg's generalized conjecture for $p$ and $k$
this conjecture for $p$ and $k$
under the assumption that the Iwasawa $\lambda$-invariant for a certain $\mathbb{Z}_p$-extension
over a finite abelian extension of $k$
%in the maximal multiple $\mathbb{Z}_p$-extension of $k$
%intermediate number field in 
%finite extension of 
vanishes and that the characteristic ideal of the Iwasawa module
associated to the cyclotomic $\mathbb{Z}_p$-extension over $k$ has a square-free generator.
%of the vanishing of the $\lambda$-invariant
%for the $\mathbb{Z}_p$-extension of $k$
%in which $\mathfrak{p}^{*}$ do not ramify.
%and of asymmetric prime decompositions of $p$ in the maximal multiple $\mathbb{Z}_p$-extension of $k$.
\end{abstract}
\maketitle
%%%%%%%%%%%%%%%%%%%%%%%%%%%%%%%%%%%%%%%%%%%%%%%%%%%%%%%%%%%%%%%%%%%%%%%%%%%%%%%%%%%%%%%%%%%%%%%%%%%%%%%%%%%%%%%%%%%%%%%%%%%%%%%%%%%%%%%%
%%%%%%%%%%%%%%%%%%%%%%%%%%%%%%%%%%%%%%%%%%%%%%%%%%%%%%%%%%%%%%%%%%%%%%%%%%%%%%%%%%%%%%%%%%%%%%%%%%%%%%%%%%%%%%%%%%%%%%%%%%%%%%%%%%%%%%%%
\section{Introduction}\label{In}
%%%%%%%%%%%%%%%%%%%%%%%%%%%%%%%%%%%%%%%%%%%%%%%%%%%%%%%%%%%%%%%%%%%%%%%%%%%%%%%%%%%%%%%%%%%%%%%%%%%%%%%%%%%%%%%%%%%%%%%%%%%%%%
%%%%%%%%%%%%%%%%%%%%%%%%%%%%%%%%%%%%%%%%%%%%%%%%%%%%%%%%%%%%%%%%%%%%%%%%%%%%%%%%%%%%%%%%%%%%%%%%%%%%%%%%%%%%%%%%%%%%%%%%%%%%%%%
Let $p$ be a prime number and $k$ a number field.
In Iwasawa theory, one studies the arithmetic of the multiple $\mathbb{Z}_p$-extensions over $k$,
which are galois groups being topologically
isomorphic to direct products of copies of the additive group $\mathbb{Z}_p$ of $p$-adic integers.
In this paper, we consider the maximal multiple $\mathbb{Z}_p$-extension field $\widetilde{k}$ of $k$.
%, which is denoted by $\widetilde{k}/k$.
%Let $\widetilde{k}/k$ be the maximal multiple $\mathbb{Z}_p$-extension of $k$.
By class field theory, there exists a positive integer $r$ with $r\geq r_2(k)+1$ such that
$\mathrm{Gal}(\widetilde{k}/k)$ is isomorphic to $\mathbb{Z}_p^{\oplus r}$, where
$r_2(k)$ is the number of complex places of $k$.
%%%%%%%%%%%%%%%%%%%%%%%%%%%%%%%%%%%%%%%%%%%%%%%%%%%%%%%%%%%%%%%%%%%%%%%%%%%%%%$$$$$$$$$$$$$$$$$$$$$$$$$$$$$$$$$$$$$$$$$$$$$$$$$$
Let $L_{\widetilde{k}}$
%$X_{\widetilde{k}}$%
be the maximal unramified pro-$p$ abelian extension field of $\widetilde{k}$.
We denote the galois group $\mathrm{Gal}(L_{\widetilde{k}}/\widetilde{k})$
by $X_{\widetilde{k}}$ and
call it the Iwasawa module of $\widetilde{k}/k$.
%put $X_{\widetilde{k}}=\mathrm{Gal}(L_{\widetilde{k}}/\widetilde{k}) $.
Iwasawa \cite{Iw} and Greenberg \cite{Gre} proved that Iwasawa modules are finitely generated torsion
$\mathbb{Z}_p[[\mathrm{Gal}(\widetilde{k}/k)]]$-modules.
%(\cite{Gre}, \cite{Iw}).
%we call $X_{\widetilde{k}}$ Iwasawa module.
This module has a very important role in Iwasawa theory.
%which has a relation with analytic side.
The usual form of Iwasawa's main conjecture states that
Iwasawa modules, regarded as algebraic objects,
have a relation with analytic objects;
%is expected to have a relation with analytic side.
roughly speaking, the characteristic ideal of a Iwasawa module
should coincide with the ideal generated by a certain $p$-adic $L$-function.
Under several assumptions,
Mazur-Wiles \cite{M-W} and Wiles \cite{Wiles} proved this conjecture
for the cyclotomic $\mathbb{Z}_p$-extensions of real abelian fields and
of totally real number fields, respectively.
\par
%Roughly speaking, that the characteristic ideal of a Iwasawa module
%coincide with the ideal generated by a $p$-adic $L$-function.\par
Concerning the structure of Iwasawa modules,
Greenberg formulated a conjecture which is called Greenberg's generalized conjecture (or GGC, for short):
%%%%%%%%%%%%%%%%%%%%%%%%%%%%%%%%%%%%%%%%%%%%%%%%%%%%%%%%%%%%%%%%%%%%%%%%%%%%%%%%%%%%%%%%%%%%%%%%%%%%%%%%%%%%%%%%%%%%%%%
%%%%%%%%%%%%%%%%%%%%%%%%%%%%%%%%%%%%%%%%%%%%%%%%%%%%%%%%%%%%%%%%%%%%%%%%%%%%%%%%%%%%%%%%%%%%%%%%%%%%%%%%%%%%%%%%%%%%%%%%%%
\begin{GGC}[\cite{Gre1}, Greenberg's generalized conjecture (GGC)]
\begin{rm}
For each prime number $p$ and number field $k$,
the $\mathbb{Z}_p[[\mathrm{Gal}(\widetilde{k}/k)]]$-module $X_{\widetilde{k}}$ is pseudo-null,
in other words, the height of the annihilator ideal
$\mathrm{Ann}_{\mathbb{Z}_p[[\mathrm{Gal}(\widetilde{k}/k)]]}(X_{\widetilde{k}})$ is greater than one.
\end{rm}
\end{GGC}
Here a $\mathbb{Z}_p[[\mathrm{Gal}(\widetilde{k}/k)]]$-module is said to be pseudo-null
if there exist two relatively prime annihilators.
%%%%%%%%%%%%%%%%%%%%%%%%%%%%%%%%%%%%%%%%%%%%%%%%%%%%%%%%%%%%%%%%%%%%%%%%%%%%%%%%%%%%%%%%%%%%%%%%%%%%%%%%%%%%%%%%%%
%%%%%%%%%%%%%%%%%%%%%%%%%%%%%%%%%%%%%%%%%%%%%%%%%%%%%%%%%%%%%%%%%%%%%%%%%%%%%%%%%%%%%%%%%%%%%%%%%%%%%%%%%%%%%%%%%%%%%%%%
If $k$ is a totally real number field and if Leopoldt's conjecture holds for $p$ and $k$,
then this conjecture is equivalent to the usual form of Greenberg's conjecture,
which states that the Iwasawa module is finite (\cite{Gre2}).
Furthermore, by the structure theorem (see Section \ref{Preli}),
the usual form of Greenberg's conjecture holds
if and only if
the Iwasawa $\lambda$- and $\mu$-invariants of the
cyclotomic $\mathbb{Z}_p$-extension over $k$ vanish.
%the Iwasawa $\lambda$-invariant of $N_{\infty}/k$ is zero if and only if
%under the assumption of
%This usual form of Greenberg conjecture has been proposed in \cite{Gre2}.
\par
%%%%%%%%%%%%%%%%%%%%%%%%%%%%%%%%%%%%%%%%%%%%%%%%%%%%%%%%%%%%%%%%%%%%%%%%%%%%%%%%%%%%%%%%%%%%%%%%%%%%%%%%%%%%555
%%%%%%%%%%%%%%%%%%%%%%%%%%%%%%%%%%%%%%%%%%%%%%%%%%%%%%%%%%%%%%%%%%%%%%%%%%%%%%%%%%%%%%%%%%%%%%%%%%%%%%%%%%%%%%%%
%%%%%%%%%%%%%%%%%%%%%%%%%%%%%%%%%%%%%%%%%%%%%%%%%%%%%%%%%%%%%%%%%%%%%%%%%%%%%%%%%%%%%%%%%%%%%%%%%%%%%%%%%%%%555
%%%%%%%%%%%%%%%%%%%%%%%%%%%%%%%%%%%%%%%%%%%%%%%%%%%%%%%%%%%%%%%%%%%%%%%%%%%%%%%%%%%%%%%%%%%%%%%%%%%%%%%%%%%%%%%%
%%%%%%%%%%%%%%%%%%%%%%%%%%%%%%%%%%%%%%%%%%%%%%%%%%%%%%%%%%%%%%%%%%%%%%%%%%%%%%%%%%%%%%%%%%%%%%%%%%%%%%%%%%%%%%%%%%
%%%%%%%%%%%%%%%%%%%%%%%%%%%%%%%%%%%%%%%%%%%%%%%%%%%%%%%%%%%%%%%%%%%%%%%%%%%%%%%%%%%%%%%%%%%%%%%%%%%%%%%%%%%%%%%%%
In this paper, we consider a weaker form of this conjecture.
We call the following conjecture weak Greenberg's generalized conjecture (or weak GGC, for short):
%We call the following conjecture weak GGC, for short:
%%%%%%%%%%%%%%%%%%%%%%%%%%%%%%%%%%%%%%%%%%%%%%%%%%%%%%%%%%%%%%%%%%%%%%%%%%%%%%%%%%%%%%%%%%%%%%%%%%%%%%%%%%%%%%%
%%%%%%%%%%%%%%%%%%%%%%%%%%%%%%%%%%%%%%%%%%%%%%%%%%%%%%%%%%%%%%%%%%%%%%%%%%%%%%%%%%%%%%%%%%%%%%%%%%%%%%%%%%%%%%%%%
\begin{WGGC}[Weak Greenberg's generalized conjecture (weak GGC)]
\begin{rm}
Let $p$ be a prime number and $k$ a number field.
Assume that $X_{\widetilde{k}}$ is not trivial.
Then $X_{\widetilde{k}}$ has a non-trivial pseudo-null $\mathbb{Z}_p[[\mathrm{Gal}(\widetilde{k}/k)]]$-submodule.
\end{rm}
\end{WGGC}
%%%%%%%%%%%%%%%%%%%%%%%%%%%%%%%%%%%%%%%%%%%%%%%%%%%%%%%%%%%%%%%%%%%%%%%%%%%%%%%%%%%%%%%%%%%%%%%%%%%%%%%%%%%%%%%%
%%%%%%%%%%%%%%%%%%%%%%%%%%%%%%%%%%%%%%%%%%%%%%%%%%%%%%%%%%%%%%%%%%%%%%%%%%%%%%%%%%%%%%%%%%%%%%%%%%%%%%%%%%%%%%%%%%
%%%%%%%%%%%%%%%%%%%%%%%%%%%%%%%%%%%%%%%%%%%%%%%%%%%%%%%%%%%%%%%%%%%%%%%%%%%%%%%%%%%%%%%%%%%%%%%%%%%%%%%%%%%%%%%%%
We note that weak GGC is proposed by Nguyen Quang Do \cite{Quang1, Quang2} for totally real number fields
and by Wingberg  \cite{Wing} for arbitrary number fields.
%Recently,
In a very recent paper,
Kataoka \cite{Kata20} proved that
$X_{\widetilde{k}}$ has no non-trivial finite $\mathbb{Z}_p[[\mathrm{Gal}(\widetilde{k}/k)]]$-submodule
for imaginary quadratic fields in which $p$ splits.
%he gave a necessary and sufficient condition that
%the Iwasawa module $X_{\widetilde{k}}$ has a non-trivial finite $\mathbb{Z}_p[[\mathrm{Gal}(\widetilde{k}/k)]]$-submodule.
%Furthermore, he
%(\cite{Kata20}).
\par
%%%%%%%%%%%%%%%%%%%%%%%%%%%%%%%%%%%%%%%%%%%%%%%%%%%%%%%%%%%%%%%%%%%%%%%%%%%%%%%%%%%%%%%%%%%%%%%%%%%%%%%%%%%%%%%%%%%%%%
%%%%%%%%%%%%%%%%%%%%%%%%%%%%%%%%%%%%%%%%%%%%%%%%%%%%%%%%%%%%%%%%%%%%%%%%%%%%%%%%%%%%%%%%%%%%%%%%%%%%%%%%%%%%%%%%%%%%%%
Throughout this paper, we suppose that $p$ is an odd prime number and that $k$ is an imaginary quadratic field.
We assume that $p$ splits in $k$ into $\mathfrak{p}$ and $\mathfrak{p}^{\ast}$.
Under this assumption, there exists a uniquely defined $\mathbb{Z}_p$-extension $N_{\infty}/k$ (respectively, $N_{\infty}^{\ast}/k$) 
such that the prime ideal $\mathfrak{p}^{\ast}$ (respectively, $\mathfrak{p}$) does not ramify.
We also denote by $k_{\infty}^c$ the cyclotomic $\mathbb{Z}_p$-extension field of $k$.\par
%%%%%%%%%%%%%%%%%%%%%%%%%%%%%%%%%%%%%%%%%%%%%%%%%%%%%%%%%%%%%%%%%%%%%%%%%%%%%%%%%%%%%%%%%%%%%%%%%%%%%%%%%%%%%%%%%%%%%%
%%%%%%%%%%%%%%%%%%%%%%%%%%%%%%%%%%%%%%%%%%%%%%%%%%%%%%%%%%%%%%%%%%%%%%%%%%%%%%%%%%%%%%%%%%%%%%%%%%%%%%%%%%%%%%%%%%%%%%
To state our main theorem,
we introduce the notion of $p$-split $p$-rational fields.
%%%%%%%%%%%%%%%%%%%%%%%%%%%%%%%%%%%%%%%%%%%%%%%%%%%%%%%%%%%%%%%%%%%%%%%%%%%%%%%%%%%%%%%%%%%%%%%%%%%%%%%%%%%%%%%%%%%%%
%%%%%%%%%%%%%%%%%%%%%%%%%%%%%%%%%%%%%%%%%%%%%%%%%%%%%%%%%%%%%%%%%%%%%%%%%%%%%%%%%%%%%%%%%%%%%%%%%%%%%%%%%%%%%%%%%%%%%
\begin{Defn}\label{rational}
For an imaginary quadratic field $k$ in which $p$ splits into $\mathfrak{p}$ and $\mathfrak{p}^{\ast}$,
$k$ is said to be $p$-split $p$-rational
if $k$ satisfies that $L_k \subset \widetilde{k}$, $\widetilde{k}^{\mathfrak{D_p}} \neq k$,
and that $\widetilde{k}^{\mathfrak{D_p}} \subset L_k$,
%and that
%the Iwasawa $\lambda$-invariant of $k_{\infty}^c/k$ is greater than one,
%that the fixed field of $\widetilde{k}$ by $\mathfrak{D_p}$ coincides with $L_k \cap \widetilde{k}$,
%$p$ splits completely in $L_{k}$.
%In other words,  coincides with $L_{k}$,
where $\mathfrak{D_p}$ is the decomposition group of the prime $\mathfrak{p}$
in $\mathrm{Gal}(\widetilde{k}/k)$,
$\widetilde{k}^{\mathfrak{D_p}}$ is the fixed field of $\widetilde{k}$ by $\mathfrak{D_p}$,
and $L_k$ is the $p$-Hilbert class field of $k$.
\end{Defn}
%%%%%%%%%%%%%%%%%%%%%%%%%%%%%%%%%%%%%%%%%%%%%%%%%%%%%%%%%%%%%%%%%%%%%%%%%%%%%%%%%%%%%%%%%%%%%%%%%%%%%%%%%%%%%%%%%%%%%%
%%%%%%%%%%%%%%%%%%%%%%%%%%%%%%%%%%%%%%%%%%%%%%%%%%%%%%%%%%%%%%%%%%%%%%%%%%%%%%%%%%%%%%%%%%%%%%%%%%%%%%%%%%%%%%%%%%%%%%5
If we suppose that
$\widetilde{k}^{\mathfrak{D_p}} \subset L_k$, then
$\mathfrak{D}_{\mathfrak{p}}$ is a normal subgroup of $\mathrm{Gal}(\widetilde{k}/\mathbb{Q})$  (Lemma \ref{cond normal}).
Hence we have ${\mathfrak{D_p}}={\mathfrak{D_{p^{\ast}}}}$.
Therefore, the definition of a 
$p$-split $p$-rational field does not depend on the choice of $\mathfrak{p}$.
%%%%%%%%%%%%%%%%%%%%%%%%%%%%%%%%%%%%%%%%%%%%%%%%%%%%%%%%%%%%%%%%%%%%%%%%%%%%%%%%%%%%%%%%%%%%%
If $k$ is $p$-split $p$-rational, the $p$-Sylow subgroup of the ideal class group of $k$ is cyclic
from the assumption that $L_k \subset \widetilde{k}$.
%Also $\widetilde{k}^{\mathfrak{D_p}} \subset L_k$ implies that
%$\mathfrak{D}_{\mathfrak{p}}$ is a normal subgroup of $\mathrm{Gal}(\widetilde{k}/\mathbb{Q})$.
%Furthermore, we see that $p$ splits in $\widetilde{k}^{\mathfrak{D_p}}$ and 
%all prime ideals of $\widetilde{k}^{\mathfrak{D_p}}$ above $p$ do not split in $\widetilde{k}$.
In Section \ref{B}, we will give a necessary and sufficient condition for $k$ to be $p$-split $p$-rational (Proposition \ref{eqrational}).\par
In this paper, we prove the following
%%%%%%%%%%%%%%%%%%%%%%%%%%%%%%%%%%%%%%%%%%%%%%%%%%%%%%%%%%%%%%%%%%%%%%%%%%%%%%%%%%%%%%%%%%%%%%%%%%%%%%%%%%%%%%%%%%%%
%%%%%%%%%%%%%%%%%%%%%%%%%%%%%%%%%%%%%%%%%%%%%%%%%%%%%%%%%%%%%%%%%%%%%%%%%%%%%%%%%%%%%%%%%%%%%%%%%%%%%%%%%%%%%%%%%%%%
\begin{thm}\label{main thm}
Let $p$ be an odd prime number and $k$ an imaginary quadratic field
which is not $p$-split $p$-rational.
%in which $p$ splits in to $\mathfrak{p}$ and $\mathfrak{p}^{*}$.
%Let be the primes of $k$ above $p$.
Assume the following conditions:\par
%$(\mathrm{i})$ $\mathfrak{D_p}$ is not a normal subgroup of $\mathrm{Gal}(\widetilde{k}/\mathbb{Q})$,
%where $\mathfrak{D_p}$ is the decomposition group of the prime $\mathfrak{p}$
%in $\mathrm{Gal}(\widetilde{k}/k)$.\par
$(\mathrm{i})$ The Iwasawa $\lambda$-invariant of $N_{\infty}/k$ is zero.\par
$(\mathrm{ii})$ The characteristic ideal of $X_{k_{\infty}^c}$
%$\mathrm{char}_{{\mathbb{Z}_p}[[\mathrm{Gal}(k_{\infty}^c/k)]]}(X_{k_{\infty}^c})$
has a generator which is square-free.\\
%$\overline{\mathbb{Q}_p}$.
Then weak {\rm{GGC}} holds for $p$ and $k$.
\end{thm}
%%%%%%%%%%%%%%%%%%%%%%%%%%%%%%%%%%%%%%%%%%%%%%%%%%%%%%%%%%%%%%%%%%%%%%%%%%%%%%%%%%%%%%%%%%%%%%%%%%%%%%%%%%%%%%%%%%%
%%%%%%%%%%%%%%%%%%%%%%%%%%%%%%%%%%%%%%%%%%%%%%%%%%%%%%%%%%%%%%%%%%%%%%%%%%%%%%%%%%%%%%%%%%%%%%%%%%%%%%%%%%%%%%%%%%%
%no counter examples have been found yet.
On the $\mathbb{Z}_p$-extension $N_{\infty}/k$, it is known that the Iwasawa $\mu$-invariant of $N_{\infty}/k$ is zero.
This was proved by Gillard \cite{Gi} and Schneps \cite{sch} for $p \geq 5$ and
Oukhaba-Vigui\'{e} \cite{OV} for $p=2, 3$.
%(\cite{Gi,OV, sch}).
By \cite[Proposition 1.C]{Mi},
%the assumption (i) in Theorem \ref{main thm} hold
the Iwasawa $\lambda$-invariant of $N_{\infty}/k$ is zero if and only if
every ideal class of the $p$-part of the ideal class group of $k$
becomes principal in $N_{\infty}$. 
%%%%%%%%%%%%%%%%%%%%%%%%%%%%%%%%%%%%%%%%%%%%%%%%%%%%%%
We note that vanishing of the Iwasawa $\lambda$-invariant of $N_{\infty}/k$ here is similar to that in the usual form of Greenberg's conjecture
%We can say that the vanishing of the Iwasawa $\lambda$-invariant of $N_{\infty}/k$ is an analogue of the usual form of Greenberg's conjecture
in the case where $\mathfrak{p}$ is totally ramified in $N_{\infty}/k$.
No counter examples of the assumption (i) have been found yet.
Fukuda-Komatsu \cite{F-K1, F-K2} checked that the Iwasawa $\lambda$-invariants of $N_{\infty}/k$ vanish
%They checked $\lambda(N_{\infty}/k)=0$
for some imaginary quadratic fields with $p=3$ 
under the assumption that $\mathfrak{p}$ is totally ramified in $N_{\infty}/k$.
%%%%%%%%%%%%%%%%%%%%%%%%%%%%%%%%%%%%%%%%%%%
%We note that Theorem \ref{main thm} includes the case where all prime ideals lying above $p$ of $k$
%split in $\widetilde{k}/k$.
%%%%%%%%%%%%%%%%%%%%%%%%%%%%%%%%%%%%%%%%%5
%which has not been considered enough in the study of GGC (see below).\par
%(\cite{F-K1, F-K2}).
%%%%%%%%%%%%%%%%%%%%%%%%%%%%%%%%%%%%%%%%%%%%%%%%%%%%%%%%%%%%%%%%%%%%%%%%%%%%%%%%%%%%%%%%%%%%%%%%%%%%%%%
On the assumption (ii) in Theorem \ref{main thm},
we note that we do not need this condition in the case where the Iwasawa module
$X_{\widetilde{k}}$ is cyclic as a $\mathbb{Z}_p[[\mathrm{Gal}(\widetilde{k}/k)]]$-module (Remark \ref{frem}).
\par
%%%%%%%%%%%%%%%%%%%%%%%%%%%%%%%%%%%%%%%%%%%%%%%%%%%%%%%%%%%%%%%%%%%%%%%%%%%%%%%%%%%%%%%%%%%%%%%%%%%%%%%%%%%%%%%%%%%%
There are two key new ideas of the proof of this theorem.
%are to
One is to find an annihilator of $X_{\widetilde{k}}$,
which is introduced in \cite[Lemma $3.3$]{Mu}.
We fix this annihilator and denote it by $f(S,T)$
(Lemma \ref{f(S,T)}).
Using this power series, we can 
%The other is to
prove that the Iwasawa module $X_{\widetilde{k}}$ is pseudo-isomorphic to
a certain $\mathbb{Z}_p[[\mathrm{Gal}(\widetilde{k}/k)]]$-cyclic module,
%under the assumptionof the failure of weak GGC for $p$ and $k$
provided that weak GGC does not hold for $p$ and $k$
(Corollary \ref{psisom}).
The other is  to consider a certain subfield of $L_{\widetilde{k}}$, which is denoted by $M_{\mathfrak{p}^{\ast}}(N_{\infty})$.
From the $\mathbb{Z}_p$-rank of $\mathrm{Gal}(M_{\mathfrak{p}^{\ast}}(N_{\infty})/\widetilde{k})$,
we consider two inequalities (A) and (B), which are stated in the end of Section \ref{decomp}.
These inequalities contradict each other, which imply Theorem \ref{main thm}.
We will prove them,
%under the assumption
assuming that weak GGC does not hold for $p$ and $k$.
%of the failure of weak GGC for $p$ and $k$,
%%%%%%%%%%%%%%%%%%%%%%%%%%%%%%%%%%%%%%%%%%%%%%%%%%%%%%%%%%%%%%%%%%%%%%%%%%%%%%%%%%%%%%%%%%%%%%%%
\par
%%%%%%%%%%%%%%%%%%%%%%%%%%%%%%%%%%%%%%%%%%%%%%%%%%%%%%%%%%%%%%%%%%%%%%%%%%%%%%%%%%%%%%%%%%%%%%%%%%%%%%%%%%%%%%%%%
%%%%%%%%%%%%%%%%%%%%%%%%%%%%%%%%%%%%%%%%%%%%%%%%%%%%%%%%%%%%%%%%%%%%%%%%%%%%%%%%%%%%%%%%%%%%%%%%%%%%%%%%%%%%%%%%%
In the case where $k$ is a $p$-split $p$-rational field, we know that the Iwasawa $\lambda$-invariant of $N_{\infty}/k$ is zero
by genus formula (see for example \cite[Remark $3.2$]{Mu}).
Furthermore, we see that $X_{\widetilde{k}}$ is cyclic as a $\mathbb{Z}_p[[\mathrm{Gal}(\widetilde{k}/k)]]$-module
(\cite[Proposition 3.8]{Mu}).
To treat this case, we consider the $\mathbb{Z}_p$-extension $N_{\infty} N_{s+1}^{*}/ N_{s+1}^{*}$,
where $s$ is the positive integer satisfying $p^s = [L_k :k]$ and $N_{s+1}^{*}$ is the $(s+1)$-st layer of $N_{\infty}^{*}/k$.
We put $H= N_{s+1}^{*}$ and $H_{\infty}=N_{\infty} N_{s+1}^{*}$.
Then $H_{\infty}/H$ is a $\mathbb{Z}_p$-extension unramified outside all prime ideals of $H$ lying above $\mathfrak{p}$.
%For the $p$-split $p$-rational fields,
We prove the following theorem
under the assumption that the Iwasawa $\lambda$-invariant for $H_{\infty}/ H$ vanishes, which is similar to that
%an analogue of the vanishing of the Iwasawa $\lambda$-invariant for $N_{\infty}/k$.
in the usual form of Greenberg's conjecture mentioned above.
%%%%%%%%%%%%%%%%%%%%%%%%%%%%%%%%%%%%%%%%%%%%%%%%%%%%%%%%%%%%%%%%%%%%%%%%%%%%%%%%%%%%%%%%%%%%%%%%%%%%%%%%%%%%%%%%%%
%%%%%%%%%%%%%%%%%%%%%%%%%%%%%%%%%%%%%%%%%%%%%%%%%%%%%%%%%%%%%%%%%%%%%%%%%%%%%%%%%%%%%%%%%%%%%%%%%%%%%%%%%%%%%%%%%%%%
\begin{thm}\label{main thm 2}
Let $p$ be an odd prime number and $k$ an imaginary quadratic field
which is $p$-split $p$-rational.
Suppose that the Iwasawa $\lambda$-invariant of $H_{\infty}/ H$ is zero.
%where $s$ is the positive integer satisfying $p^s = [L_k :k]$ and
%$N_{s+1}^{*}$ is the $(s+1)$-th layer of $N_{\infty}^{*}/k$.
%in which $p$ splits in to $\mathfrak{p}$ and $\mathfrak{p}^{*}$.
%Let be the primes of $k$ above $p$.
%Assume the following conditions:\par
%$(\mathrm{i})$ $\mathfrak{D_p}$ is not a normal subgroup of $\mathrm{Gal}(\widetilde{k}/\mathbb{Q})$,
%where $\mathfrak{D_p}$ is the decomposition group of the prime $\mathfrak{p}$
%in $\mathrm{Gal}(\widetilde{k}/k)$.\par
%$(\mathrm{i})$ The $\lambda$-invariant of $N_{\infty} \cdot N_{s+1}^{*}/ N_{s+1}^{*}$ is zero.\par
%$(\mathrm{ii})$ The characteristic ideal of $X_{k_{\infty}^c}$
%$\mathrm{char}_{{\mathbb{Z}_p}[[\mathrm{Gal}(k_{\infty}^c/k)]]}(X_{k_{\infty}^c})$
%has a generator which is square-free.\\
%$\overline{\mathbb{Q}_p}$.
Then weak {\rm{GGC}} holds for $p$ and $k$.
\end{thm}
%%%%%%%%%%%%%%%%%%%%%%%%%%%%%%%%%%%%%%%%%%%%%%%%%%%%%%%%%%%%%%%%%%%%%%%%%%%%%%%%%%%%%%%%%%%%%%%%%%%%%%%%%%%%%%%%%%%
%%%%%%%%%%%%%%%%%%%%%%%%%%%%%%%%%%%%%%%%%%%%%%%%%%%%%%%%%%%%%%%%%%%%%%%%%%%%%%%%%%%%%%%%%%%%%%%%%%%%%%%%%%%%%%%%%%%
%%%%%%%%%%%%%%%%%%%%%%%%%%%%%%%%%%%%%%%%%%%%%%%%%%%%%%%%%%%%%%%%%%%%%%%%%%%%%%%%%%%%%%%%%%%%%%%%%%%%%%%%%%%%%%%%%
%%%%%%%%%%%%%%%%%%%%%%%%%%%%%%%%%%%%%%%%%%%%%%%%%%%%%%%%%%%%%%%%%%%%%%%%%%%%%%%%%%%%%%%%%%%%%%%%%%%%%%%%%%%%%%%%%
%%%%%%%%%%%%%%%%%%%%%%%%%%%%%%%%%%%%%%%%%%%%%%%%%%%%%%%%%%%%%%%%%%%%%%%%%%%%%%%%%%%%%%%%%%%%%%%%%%%%%%%%%%%%%%%%%
There is a relation between weak GGC and GGC.
%%%%%%%%%%%%%%%%%%%%%%%%%%%%%%%%%%%%%%%%%%%%%%%%%%%%%%%%%%%%%%%%%%%%%%%%%%%%%%%%%%%%%%%%%%%%%%%%%%%%%%%%%%%%%%%%%
%%%%%%%%%%%%%%%%%%%%%%%%%%%%%%%%%%%%%%%%%%%%%%%%%%%%%%%%%%%%%%%%%%%%%%%%%%%%%%%%%%%%%%%%%%%%%%%%%%%%%%%%%%%%%%%%%
%\begin{prop}[{\cite[Proposition $3.1$]{Fuj}}]\label{Fu}
%Suppose that weak {\rm{GGC}} holds for $p$ and $k$
%and that
%$\mathrm{char}_{\mathbb{Z}_p[[\mathrm{Gal}(k_{\infty}^c/k)]]}((X_{\widetilde{k}})_{\mathrm{Gal}(\widetilde{k}/k_{\infty}^c)})$
%is a prime ideal.
%Then {\rm{GGC}} holds for $p$ and $k$.
%\end{prop}
Fujii verified that weak {\rm{GGC}} for $p$ and $k$ implies {\rm{GGC}} for $p$ and $k$
%%%%%%%%%%%%%%%%%%%%%%%%%%%%%%%%%%%%%%%%%%%%%%%%%%%%%%%%%%%%%%%%%%%%%%%%%%%%%%%%%%%%%%%%%%%%%%%%%%%%%%%%%%%%%%%%
if the characteristic ideal of
the $\mathrm{Gal}(\widetilde{k}/k_{\infty}^c)$-coinvariant module of $X_{\widetilde{k}}$
as a $\mathbb{Z}_p[[\mathrm{Gal}(k_{\infty}^c/k)]]$-module
%$X_{k_{\infty}^c}$
%$\mathrm{char}_{\mathbb{Z}_p[[\mathrm{Gal}(k_{\infty}^c/k)]]}((X_{\widetilde{k}})_{\mathrm{Gal}(\widetilde{k}/k_{\infty}^c)})$
is a prime ideal (\cite[Proposition $3.1$]{Fuj10}).
%%%%%%%%%%%%%%%%%%%%%%%%%%%%%%%%%%%%%%%%%%%%%%%%%%%%%%%%%%%%%%%%%%%%%%%%%%%%%%%%%%%%%%%%
%%%%%%%%%%%%%%%%%%%%%%%%%%%%%%%%%%%%%%%%%%%%%%%%%%%%%%%%%%%%%%%%%%%%%%%%%%%%%%%%%%%%%%%%%
Hence we obtain the following corollary from Theorems \ref{main thm} and \ref{main thm 2}.
%%%%%%%%%%%%%%%%%%%%%%%%%%%%%%%%%%%%%%%%%%%%%%%%%%%%%%%%%%%%%%%%%%%%%%%%%%%%%%%%%%%%%%%%%%%%%%%%%%%%%
%%%%%%%%%%%%%%%%%%%%%%%%%%%%%%%%%%%%%%%%%%%%%%%%%%%%%%%%%%%%%%%%%%%%%%%%%%%%%%%%%%%%%%%%%%%%%%%%%%%%%%%%%%
\begin{cor} \label{cor GGC}
%Assume the same conditions in
Assume either the assumption of Theorem \ref{main thm} or that of Theorem \ref{main thm 2} holds.
%Assume also that $\mathrm{char}_{\mathbb{Z}_p[[\mathrm{Gal}(k_{\infty}^c/k)]]}((X_{\widetilde{k}})_{\mathrm{Gal}(\widetilde{k}/k_{\infty}^c)})$
%is generated by an irreducible polynomial.
Assume also that the characteristic ideal of
the $\mathrm{Gal}(\widetilde{k}/k_{\infty}^c)$-coinvariant module of $X_{\widetilde{k}}$
as a $\mathbb{Z}_p[[\mathrm{Gal}(k_{\infty}^c/k)]]$-module
is a prime ideal.
Then {\rm{GGC}} holds for $p$ and $k$.
\end{cor}
%%%%%%%%%%%%%%%%%%%%%%%%%%%%%%%%%%%%%%%%%%%%%%%%%%%%%%%%%%%%%%%%%%%%%%%%%
%%%%%%%%%%%%%%%%%%%%%%%%%%%%%%%%%%%%%%%%%%%%%%%%%%%%%%%%%%%%%%%%%%%%%%%%%%%%
Furthermore, we apply Theorem \ref{main thm} and Corollary \ref{cor GGC}
to imaginary quadratic fields
$k=\mathbb{Q}(\sqrt{-d})$.
%In the range $1< d < 10^3$
%and $d \equiv 2 ~\mathrm{mod}~3$,
Fujii and Shirakawa gave criteria of  GGC
for imaginary quadratic fields (\cite[Theorem 5.1]{Fuj10}, \cite[Theorem 5.2]{Shi}).
%In the range above,
If $1< d < 10^3$
and $d \equiv 2 ~\mathrm{mod}~3$,
they verified the conjecture
using the first and the second layers of a certain $\mathbb{Z}_p$-extension
except for $d=971$.
%%%%%%%%%%%%%%%%%%%%%%%%%%%%%%%%%%%%%%%%%%%%%%%%%%%%%%%%%%%%%%%%%%%%%%%%%%%%%%
%%%%%%%%%%%%%%%%%%%%%%%%%%%%%%%%%%%%%%%%%%%%%%%%%%%%%%%%%%%%%%%%%%%%%%%%%%%%%5
Combining Corollary \ref{cor GGC} with Fujii's criterion,
we can verify the conjecture including the case of $d=971$.
We will give more precise calculation in Section \ref{example}.
Therefore we obtain the following
%%%%%%%%%%%%%%%%%%%%%%%%%%%%%%%%%%%%%%%%%%%%%%%%%%%%%%%%%%%%%%%%%%%%
\begin{cor} \label{thm GGC}
Let $d$ be a positive integer and
$k=\mathbb{Q}(\sqrt{-d})$ an imaginary quadratic field in which $p$ splits.
If $1< d < 10^3$ and $d \equiv 2 ~\mathrm{mod}~3$, then
{\rm{GGC}} holds for $p=3$ and $k$.
%Then  Greenberg's generalized conjecture holds for $p$ and $k$.
\end{cor}
%%%%%%%%%%%%%%%%%%%%%%%%%%%%%%%%%%%%%%%%%%%%%%%%%%%%%%%%%%%%%%%%%%%%%%%%%%%%%%%%%%%%%%%%%%%%%%%%%5
%%%%%%%%%%%%%%%%%%%%%%%%%%%%%%%%%%%%%%%%%%%%%%%%%%%%%%%%%%%%%%%%%%%%%%%%%%%%%%%%%%%%%%%%%%%%%%%%%%%%%
Concerning GGC, the following is known.
Let $p$ be an odd prime number.
Minardi \cite[Proposition $3.\mathrm{B}$]{Mi} verified GGC for an imaginary quadratic field $k$ in which $p$ splits
under the assumption that
%$p$ does not divide
the class number of $k$ is not divisible by $p$.
This result is generalized by Itoh \cite{Ito} and Fujii \cite{Fuj17}.
They verified the conjecture for imaginary abelian quartic fields (\cite{Ito})
and for CM-fields of degree greater than or equal to $4$ (\cite{Fuj17})
%They proved GGC for a CM-field $k$ in which $p$ splits completely
under
%elementary assumption, in other words, 
the assumption that $p$ splits completely in the base field and that
the class number of the base field is not divisible by $p$ , and
that all Iwasawa invariants of the cyclotomic $\mathbb{Z}_p$-extension over the maximal totally real subfield of the base field
are trivial.
%({\cite{Ito}}, {\cite{Fuj17}}).
%In this case, it is known that GGC holds.
%This was proved by
%We note that
%, in these situation, at least one prime of $k$ lying above $p$ does not split in $\widetilde{k}$.
%these situation includes the condition
%because $p$ does not divide the class number of $k$.
Using a method of Minardi in \cite{Mi},
Ozaki \cite[Theorem B]{Oz}
%More precisely, he
proved GGC under the assumption that
at least one prime ideal lying above $p$ does not split in $N_{\infty}$
and that the Iwasawa $\lambda$-invariant of $N_{\infty}/k$ is zero
for an imaginary quadratic field $k$ in which $p$ splits.
GGC has some important consequences.
For example, if this conjecture holds for a totally real number field $k$, then
this conjecture gives a simple proof of
the usual form of Iwasawa's main conjecture for $p$ and $k$ (\cite{Gre2}).
Ozaki \cite{Oz} and Kataoka \cite{Kata17} studied the behavior of Iwasawa $\lambda$- and $\mu$-invariants
of infinitely many $\mathbb{Z}_p$-extensions over a fixed number field.
They determined many Iwasawa $\lambda$- and
$\mu$-invariants of $\mathbb{Z}_p$-extensions
under the assumption that GGC holds for the base field.
%(\cite{Oz}, \cite{Kata17}).
%%%%%%%%%%%%%%%%%%%%%%%%%%%%%%%%%%%%%%%%%%%%%%%%%%%%%%%%%%%%%%%%%%%%%%%%%%%%%%%%%%5
%Moreover, Sakamoto recently found a new relation between GGC and ETNC (\cite{Sa}).
\par
%%%%%%%%%%%%%%%%%%%%%%%%%%%%%%%%%%%%%%%%%%%%%%%%%%%%%%%%%%%%%%%%%%%%%%%%%%%%%%%%%%%%%%%%%%%%%%%%%%%%%%%%%%%%%%%%
%%%%%%%%%%%%%%%%%%%%%%%%%%%%%%%%%%%%%%%%%%%%%%%%%%%%%%%%%%%%%%%%%%%%%%%%%%%%%%%%%%%%%%%%%%%%%%%%%%%%%%%%%%%%%%%%%%%%%
%%%%%%%%%%%%%%%%%%%%%%%%%%%%%%%%%%%%%%%%%%%%%%%%%%%%%%%%%%%%%%%%%%%%%%%%%%%%%%%%%%%%%%%%%%%%%%%%%%%%%%%%%%%%%%%%%%
The outline of this paper is as follows.
In Section \ref{Preli}, we provide some basic definitions and notation.
In Section \ref{decomp}, we study the decomposition of the prime number $p$
in $\widetilde{k}/\mathbb{Q}$.
%There are two case according to the decomposition group $\mathfrak{D_p}$ is normal or non-normal
%in $\mathrm{Gal}(\widetilde{k}/\mathbb{Q})$.
Furthermore, in the end of this section,
we give two inequalities $(\mathrm{A})$ and $(\mathrm{B})$,
which imply Theorem \ref{main thm}.
%The inequality $\mathrm{A}$
These inequalities give an upper bound and a lower bound of the $\mathbb{Z}_p$-rank
of a certain Galois group.
%of the maximal pro-$p$ abelian extension of $N_{\infty}$
%unramified outside all prime ideals of $N_{\infty}$ lying above $\frak{p}^{\ast}$.
In Sections \ref{A} and \ref{B},
we give the proof of the inequalities (A) and (B), respectively,
%In Section \ref{B}, we give the proof of the inequality (B).
assuming that weak GGC does not hold for $p$ and $k$.
In Section \ref{rational case}, we consider the case where $k$ is a $p$-split $p$-rational and prove
Theorem \ref{main thm 2}.
In Section \ref{example}, we introduce some numerical examples including $p=3$ and $k=\mathbb{Q}(\sqrt{-971})$. 
%%%%%%%%%%%%%%%%%%%%%%%%%%%%%%%%%%%%%%%%%%%%%%%%%%%%%%%%%%%%%%%%%%%%%%%%%%%%%%%%%%%%%%%%%%%%%%%%%%%%%%%%%%%%%%%%%%%%
%%%%%%%%%%%%%%%%%%%%%%%%%%%%%%%%%%%%%%%%%%%%%%%%%%%%%%%%%%%%%%%%%%%%%%%%%%%%%%%%%%%%%%%%%%%%%%%%%%%%%%%%%%%%%%%%%%
%%%%%%%%%%%%%%%%%%%%%%%%%%%%%%%%%%%%%%%%%%%%%%%%%%%%%%%%%%%%%%%%%%%%%%%%%%%%%%%%%%%%%%%%%%%%%
%%%%%%%%%%%%%%%%%%%%%%%%%%%%%%%%%%%%%%%%%%%%%%%%%%%%%%%%%%%%%%%%%%%%%%%%%%%%%%%%%%%%%%%%%%%%%%%%%%%%%%%%%%%%%%%%%%
\section{Preliminary}\label{Preli}
\subsection~
%In this section, we define notation.
Let $p$ be an odd prime number and $k$ an imaginary quadratic field in which $p$ splits
into $\mathfrak{p}$ and $\mathfrak{p}^{\ast}$.
Let $K$ be a $\mathbb{Z}_p$-extension or the $\mathbb{Z}_p^{\oplus 2}$-extension of $k$.
We denote by $L_{K}/K$ the maximal unramified pro-$p$ abelian extension
and put $X_K=\mathrm{Gal}(L_{K}/K) $.
%Then $X_{K}$ is a $\mathbb{Z}_p$-module,
Since the Galois group $\mathrm{Gal}(K/k)$ acts naturally on $X_{K}$,
it becomes a $\mathbb{Z}_p[[\mathrm{Gal}(K/k)]]$-module.
%and $\mathrm{Gal}(K/k)$ acts on $X_{k_{\infty}}$ by conjugation.
%More precisely, the action is given by
It is known that $X_{K}$ is a finitely generated torsion $\mathbb{Z}_p[[\mathrm{Gal}(K/k) ]]$-module
(\cite{Gre}, \cite{Iw}).\par
Since we have $\mathrm{Gal}(\widetilde{k}/k) \cong \mathbb{Z}_p^{\oplus 2}$,
$k$ has two independent $\mathbb{Z}_p$-extensions.
For example, the cyclotomic $\mathbb{Z}_p$-extension $k_{\infty}^c$ and the anti-cyclotomic $\mathbb{Z}_p$-extension $k_{\infty}^a$
are disjoint over $k$ and satisfy $\widetilde{k}=k_{\infty}^ck_{\infty}^a$.
Furthermore, since we suppose that $p$ splits in $k$,
there exist two $\mathbb{Z}_p$-extensions in which one of the prime ideals of $k$ lying above $p$ does not ramify.
Let $N_{\infty}/k$ (respectively, $N_{\infty}^{\ast}/k$) be the $\mathbb{Z}_p$-extension of $k$
in which $\mathfrak{p}^{\ast}$ (respectively, $\mathfrak{p}$) does not ramify.
We note that $N_{\infty}$ (respectively, $N_{\infty}^{\ast}$) coincides with the fixed field of $\widetilde{k}$ by
the inertia subgroup of $\mathrm{Gal}(\widetilde{k}/k)$ for the prime ideal $\mathfrak{p}^{\ast}$ (respectively, $\mathfrak{p}$) of $k$
(\cite[Lemma 3.2]{Mi}).\par
%%%%%%%%%%%%%%%%%%%%%%%%%%%%%%%%%%%%%%%%%%%%%%%%%%%%%%%%%%%%%%%%%%%%%%%%%%%%%%%%%%%%%%%%
Let $\sigma$ and $\tau $ be topological generators of $\mathrm{Gal}(\widetilde{k}/k_{\infty}^c)$ and $\mathrm{Gal}(\widetilde{k}/k_{\infty}^a)$, respectively.
By the isomorphism
$$
\mathrm{Gal}(\widetilde{k}/k) \cong \mathrm{Gal}(\widetilde{k}/k_{\infty}^c) \times \mathrm{Gal}(\widetilde{k}/k_{\infty}^a),
$$
we fix an isomorphism 
\begin{eqnarray}
\mathbb{Z}_p[[ \mathrm{Gal}(\widetilde{k}/k)]]  \cong \mathbb{Z}_p[[S,T]] \qquad  (\sigma \leftrightarrow 1+S, ~\tau \leftrightarrow 1+T). \label{isom ring}
\end{eqnarray}
We put $\Lambda=\mathbb{Z}_p[[S,T]]$. By this isomorphism, we regard $X_{\widetilde{k}}$ as a $\Lambda$-module.
%Put $\Lambda = \mathbb{Z}_p[[ \mathrm{Gal}(\widetilde{k}/k)]]$.
We note that $\Lambda$ is a unique factorization domain and 
a noetherian local integral domain with the maximal ideal $(S,T,p)$.\par
%%%%%%%%%%%%%%%%%%%%%%%%%%%%%%%%%%%%%%%%%%%%%%%%%%%%%%%%%%%%%%%%%%%%%%%%%%%%%%%%%%%%%%%%%%%%%%%%%%%%%%%%%%%%%%%%%%%5
%The completed group ring $\Lambda$ has
%subrings $\mathbb{Z}_p[[S]]$ and $\mathbb{Z}_p[[T]]$.
%%%%%%%%%%%%%%%%%%%%%%%%%%%%%%%%%%%%%%%%%%%%%%%%%%%%%%%%%%%%%%%%%%%%%%%%%%%%%%%%%%%%%%%%%%%%
\subsection~
%%%%%%%%%%%%%%%%%%%%%%%%%%%%%%%%%%%%%%%%%%%%%%%%%%%%%%%%%%%%%%%%%%%%%%%%%%%%%%%%%%%%%%%%%%%%%
For a ring $R$, we denote by $R^{\times}$ the unit group of $R$.
We suppose that $R$ is the formal power series ring in one variable
over a discrete valuation ring or suppose that
$R$ is the formal power series ring $\Lambda$.
%$R=\mathbb{Z}_p[[S]]$ or $R=\mathbb{Z}_p[[T]]$.
For a finitely generated torsion $R$-module $M$,
we define the characteristic ideal of $M$ as follows.
By the structure theorem of $R$-modules (\cite[Proposition 5.1.7]{NSW}), there exists an $R$-homomorphism 
%\[
%M \rightarrow \left( \bigoplus _i R/({p}^{m_i}) \right) \oplus
%\left( \bigoplus_j  R / (f_j^{n_j}) \right)
%\]
%with finite kernel and finite cokernel, where
%$m_i$'s, $n_j$'s %\in \mathbb{Z}_{\geq 0}$ 
%are non-negative integers and $f_j \in R $
%is a distinguished irreducible polynomial. We define the characteristic ideal of $M$ by
%\[
%\mathrm{char}_{R}(M)=\left( \prod_i {p}^{m_i}\prod_j{f_j^{n_j}} \right),
%\]
%which is an ideal in $R$.\par
%Next we suppose that $R=\Lambda$.
%Let $M$ be a finitely generated torsion $\Lambda$-module.
%By the  structure theorem, there exists a $R$-homomorphism
\[
%0 \rightarrow P \rightarrow
M \rightarrow \bigoplus_{i=1}^{l} R/{\mathfrak{p}_i}^{m_i}
%\rightarrow Q \rightarrow 0,
\]
with pseudo-null kernel and pseudo-null cokernel,
where $\mathfrak{p}_i$'s are prime ideals of height one and
$l$ is a non-negative integer, and $m_i$'s are positive integers.
Here, for an $R$-module $P$,
$P$ is said to be pseudo-null if
%Here we call $P$ pseudo-null for a $\Lambda$-module $P$
there are two relatively prime annihilators of $P$.
We note that a pseudo-null $R$-module is an $R$-module of finite length
if $R$ is the formal power series ring in one variable
over a discrete valuation ring.
Then, we define the characteristic ideal of $M$ by
\[
\mathrm{char}_{R}(M)=\left( \prod_{i=1}^{l} {\mathfrak{p}_i}^{m_i} \right),
\]
which is an ideal in $R$.\par
%%%%%%%%%%%%%%%%%%%%%%%%%%%%%%%%%%%%%%%%%%%%%%%%%%%%%%%%%%%%%%%%%%%%%%%%%%%%%%%%%%%%%%%%%%%%%%%%%%%%%%%%%%%%%%%%%%%%
%%%%%%%%%%%%%%%%%%%%%%%%%%%%%%%%%%%%%%%%%%%%%%%%%%%%%%%%%%%%%%%%%%%%%%%%%%%%%%%%%%%%%%%%%%%%%%%%%%%%%%%%%%%%%%%%%%%
Let $G$ be a profinite group. For any topological $G$-module $M$, we denote by $M^{G}$
the subset of elements of $M$ invariant under the action of $G$.
We also denote by $M_{G}$ the largest quotient module of $M$ on which $G$ acts trivially, namely,
$$
M_{G}= M/M', \quad M'= \overline{\langle (g-1)m~|~g \in G, m \in M \rangle },
$$
where $\overline{\langle (g-1)m~|~g \in G, m \in M \rangle }$ is the topological closure of $\langle (g-1)m~|~g \in G, m \in M \rangle $ in $M$.
%%%%%%%%%%%%%%%%%%%%%%%%%%%%%%%%%%%%%%%%%%%%%%%%%%%%%%%%%%%%%%%%%%%
%%%%%%%%%%%%%%%%%%%%%%%%%%%%%%%%%%%%%%%%%%%%%%%%%%%%%%%%%%%%%%%%%%
%%%%%%%%%%%%%%%%%%%%%%%%%%%%%%%%%%%%%%%%%%%%%%%%%%%%%%%%%%%%%%%%%%%
\section{The decomposition of the prime $p$}\label{decomp}
In this section, we roughly describe the decomposition of the prime $p$ in $\mathrm{Gal}(\widetilde{k}/\mathbb{Q})$.
%%%%%%%%%%%%%%%%%%%%%%%%%%%%%%%%%%%%%%%%%%%%%%%%%%%%%%%%%%%%%%%%%%%%%%%%%%%
%%%%%%%%%%%%%%%%%%%%%%%%%%%%%%%%%%%%%%%%%%%%%%%%%%%%%%%%%%%%%%%%%%%%%%%%%
%%%%%%%%%%%%%%%%%%%%%%%%%%%%%%%%%%%%%%%%%%%%%%%%%%%%%%%%%%%%%%%%%%%%%%%%%%%%%%%%%%%%%%%%%%%%%%%%%%%%%%%%%%%%%%%%%%%%%%%%%%5
We first have the following
%%%%%%%%%%%%%%%%%%%%%%%%%%%%%%%%%%%%%%%%%%%%%%%%%%%%%%%%%%%%%%%%%%%%%%%%%%%%%%%%%%%%%%%%%%%%%%%%%%%%%%%%%%%%%%%%%%%%%%%%%%%
\begin{lem}[{\rm{See for example} \cite[Lemma 1]{Oz}}] \label{Ozaki}
Let $k$ be an imaginary quadratic field in which $p$ splits and
$k_{\infty}$ a $\mathbb{Z}_p$-extension different from $N_{\infty}$ and $N_{\infty}^{\ast}$.
%Assume that $k_{\infty}$ is totally ramified at the prime lying above $p$ if $p$ does not split in $k$.
Then there exists an exact sequence
of $\mathbb{Z}_p[[\mathrm{Gal}(k_{\infty}/k)]]$-modules:
$$
0 \rightarrow (X_{\widetilde{k}})_{\mathrm{Gal}(\widetilde{k}/k_{\infty})}
\rightarrow X_{k_{\infty}} \rightarrow
\mathrm{Gal}(\widetilde{k}/k_{\infty}) \rightarrow 0.
$$
\end{lem}
%%%%%%%%%%%%%%%%%%%%%%%%%%%%%%%%%%%%%%%%%%%%%%%%%%%%%%%%%%%%%%%%%%%%%%%%%%%%%%%%%%%%%%%%%%%%%%%%%%%%%%%%%%%%%%%%5
%%%%%%%%%%%%%%%%%%%%%%%%%%%%%%%%%%%%%%%%%%%%%%%%%%%%%%%%%%%%%%%%%%%%%%%%%%%%%%%%%%%%%%%%%%%%%%%%%%%%%%%%%%%%%%%%%
For a $\mathbb{Z}_p$-extension $k_{\infty}/k$, we denote by $\lambda(k_{\infty}/k)$, $\mu(k_{\infty}/k)$
the Iwasawa $\lambda$-invariant, $\mu$-invariant of $k_{\infty}/k$, respectively.
Since we suppose that $p$ splits in $k$,
we have $\lambda(k_{\infty}^c/k) \geq 1$.
% by classical Iwasawa theory.
%%%%%%%%%%%%%%%%%%%%%%%%%%%%%%%%%%%%%%%%%%%%%%%%%%%%%%%%%%%%%%%%%%%%%%%%%%%%%%%%%%%%%%%%%%%%%%%%%%%%55
%%%%%%%%%%%%%%%%%%%%%%%%%%%%%%%%%%%%%%%%%%%%%%%%%%%%%%%%%%%%%%%%%%%%%%%%%%%%%%%%%%%%%%%%%%%%%%%%%%%%%%%%%
\begin{lem}\label{trivial}
The Iwasawa module $X_{\widetilde{k}}$ is trivial if and only if $\lambda(k_{\infty}^c/k)=1$.
\end{lem}
%%%%%%%%%%%%%%%%%%%%%%%%%%%%%%%%%%%%%%%%%%%%%%%%%%%%%%%%%%%%%%%%%%%%%%%%%%%%%%%%%%%%%%%%%%%%%%%%%%%%%%%%%%
%%%%%%%%%%%%%%%%%%%%%%%%%%%%%%%%%%%%%%%%%%%%%%%%%%%%%%%%%%%%%%%%%%%%%%%%%%%%%%%%%%%%%%%%%%%%%%%%%%%%%%%%%%%%
\begin{proof}
In the case of $\lambda(k_{\infty}^c/k)=1$,
%in other words, $\lambda^{*}=0$,
%we note that GGC holds for $p$ and $k$.
%Indeed,
we have $(X_{\widetilde{k}})_{\mathrm{Gal}(\widetilde{k}/k_{\infty}^c)}=0$ using Lemma \ref{Ozaki} and using the fact that
$X_{k_{\infty}^c}$ has no non-trivial finite $\mathbb{Z}_p[[\mathrm{Gal}(k_{\infty}^c/k)]]$-submodule (\cite[Corollary 13.29]{Wa}).
By Nakayama's lemma, we obtain $X_{\widetilde{k}}=0$.
If we suppose that $X_{\widetilde{k}}=0$, then we have $X_{k_{\infty}^c} \cong \mathrm{Gal}(\widetilde{k}/k_{\infty}^c) \cong \mathbb{Z}_p$.
This implies that $\lambda(k_{\infty}^c/k)=1$. Thus we get the conclusion.
\end{proof}
%Throughout this paper, we assume that $\lambda^{*} \geq 1$.\par
%%%%%%%%%%%%%%%%%%%%%%%%%%%%%%%%%%%%%%%%%%%%%%%%%%%%%%%%%%%%%%%%%%%%%%%%%%%%%%%%%%%%%%%%%%%%%%%%%%%%%%%%%
We put $\lambda^{*}=\lambda(k_{\infty}^c/k) -1$.
We obtain an annihilator of $X_{\widetilde{k}}$ from the following 
%Using Lemma \ref{Ozaki}, we have the following
%%%%%%%%%%%%%%%%%%%%%%%%%%%%%%%%%%%%%%%%%%%%%%%%%%%%%%%%%%%%%%%%%%%%%%%%%%%%%%%%%%%%%%%%%%%%%%%%%%%%%%%%%%%%%%%
\begin{lem}[{\cite[Lemma 3.3]{Mu}}]\label{f(S,T)}
Suppose that $\lambda^{\ast}\geq 1$, where $\lambda^{\ast}$ is the integer defined above.
Then there exist power series $f(S,T) \in \mathrm{Ann}_{\Lambda}(X_{\widetilde{k}})$ and
$g_i(S) \in \mathbb{Z}_p[[S]]$ $(i=0, \dots , \lambda^{\ast}-1)$
such that
$$
f(S,T) = T^{\lambda^{\ast}} + g_{\lambda^{\ast}-1}(S) T^{\lambda^{\ast}-1} + \dots + g_{1}(S)T +g_{0}(S).
$$
%where $\lambda^{c}=\lambda(k_{\infty}^c / k)$.
\end{lem}
%%%%%%%%%%%%%%%%%%%%%%%%%%%%%%%%%%%%%%%%%%%%%%%%%%%%%%%%%%%%%%%%%%%%%%%%%%%%%%%%%%%%%%%%%%%%%%%%
%%%%%%%%%%%%%%%%%%%%%%%%%%%%%%%%%%%%%%%%%%%%%%%%%%%%%%%%%%%%%%%%%%%%%%%%%%%%%%%%%%%%%%%%%%%%%%%%%%
\begin{proof}
We note that
$X_{k_{\infty}^c}$ has no non-trivial finite $\mathbb{Z}_p[[\mathrm{Gal}(k_{\infty}^c/k)]]$-submodule.
%(\cite[Corollary 13.29]{Wa}).
Applying Lemma  \ref{Ozaki} to the case of $k_{\infty}=k_{\infty}^c$,
we have
%the following exact sequence
%$$
%0 \rightarrow (X_{\widetilde{k}})_{\mathrm{Gal}(\widetilde{k}/k_{\infty}^c)} \rightarrow X_{k_{\infty}^c}
%\rightarrow \mathrm{Gal}(\widetilde{k} \cap L_{k_{\infty}^c}/k_{\infty}^c) \rightarrow 0
%$$
%as $\mathbb{Z}_p[[\mathrm{Gal}(k_{\infty}^c/k)]]$-modules.
%Since $k$ is an imaginary quadratic field, $X_{k_{\infty}^c}$ is a free $\mathbb{Z}_p$-module.
%We note that $\mathrm{rank}_{\mathbb{Z}_p} 
%\left( (X_{\widetilde{k}})_{\mathrm{Gal}(\widetilde{k}/k_{\infty}^c)} \right)=\lambda^{\ast}$ by (\ref{n}).
%Since the element $\sigma$ is a generator of $\mathrm{Gal}(\widetilde{k}/k_{\infty}^c)$,
%we have
$$
  X_{\widetilde{k}}/SX_{\widetilde{k}} \cong (X_{\widetilde{k}})_{\mathrm{Gal}(\widetilde{k}/k_{\infty}^c)} \cong \mathbb{Z}_p^{\oplus \lambda^{\ast} }
$$
by the isomorphism (\ref{isom ring}).
Using Nakayama's lemma, there exist $x_i \in X_{\widetilde{k}} ~(i=1, \dots, \lambda^{\ast})$ such that $X_{\widetilde{k}}= \langle x_1, \dots, x_{\lambda^{\ast}}
\rangle_{\mathbb{Z}_p[[S]]}$.
Then there exist $f_{ij}(S) \in \mathbb{Z}_p[[S]]~(i,j=1,\dots , \lambda^{\ast})$ such that
\begin{eqnarray*}
%Tx_{1}&=&f_{11}(S) x_1 + f_{1 2}(S)x_2 + \dots +f_{1\lambda^{\ast}}(S)x_{\lambda^{\ast}},\\
%Tx_{2}&=&f_{21}(S) x_1 + f_{2 2}(S)x_2 + \dots +f_{2\lambda^{\ast}}(S)x_{\lambda^{\ast}},\\
%&~&\qquad \qquad \qquad \dots \\
%Tx_{\lambda^{\ast}}&=&f_{\lambda^{\ast} 1}(S) x_1 + f_{\lambda^{\ast} 2}(S)x_2 + \dots +f_{\lambda^{\ast} \lambda^{\ast}}(S)x_{\lambda^{\ast}}.
Tx_{i}= \sum_{j=1}^{\lambda^{\ast}} f_{ij}(S) x_j , \quad (i=1,\dots , \lambda^{\ast}).
\end{eqnarray*}
By these relations, we have the following matrix
\begin{eqnarray*}
A=
%\begin{cases}
~\left(
\begin{array}{ccccc}
T-f_{11}(S) &  -f_{12}(S) & \dots & -f_{1 \lambda^{\ast}}(S)\\
-f_{21}(S) &T-f_{22}(S) &  \dots & -f_{2 \lambda^{\ast}}(S) \\
  \dots &  \dots  &   \dots      &  \dots \\
-f_{\lambda^{\ast} 1}(S)  &  -f_{\lambda^{\ast} 2}(S) & \dots & T-f_{\lambda^{\ast} \lambda^{\ast}}(S)
\end{array}
\right).
%\quad &\mathrm{~if~} \lambda^{*} \geq 2,\\~\\
%~(T-f_{11}(S)) \quad &\mathrm{~if~} \lambda^{*}=1.
%\end{cases}
\end{eqnarray*}
We denote by $\mathrm{det}(A)$ the determinant of the matrix $A$.
We put $f(S,T)=\mathrm{det}(A)$. Then we obtain
$$
f(S,T) = T^{\lambda^{\ast}} + g_{\lambda^{\ast}-1}(S) T^{\lambda^{\ast}-1} + \dots + g_{1}(S)T +g_{0}(S)
$$
for some $g_i(S) \in \mathbb{Z}_p[[S]]$ $(i=0, \dots , \lambda^{\ast}-1)$.
It is easy to see that $f(S,T)X_{\widetilde{k}}=0$.
Thus we get the conclusion.
\end{proof}
%%%%%%%%%%%%%%%%%%%%%%%%%%%%%%%%%%%%%%%%%%%%%%%%%%%%%%%%%%%%%%%%%%%%%%%%%%%%%%%%%%%%%%%%%%%%%%%%%%%%%%%%%%%%%%%%%%%%%%
%%%%%%%%%%%%%%%%%%%%%%%%%%%%%%%%%%%%%%%%%%%%%%%%%%%%%%%%%%%%%%%%%%%%%%%%%%%%%%%%%%%%%%%%%%%%%%%%%%%%%%%%%
\begin{Rem}\label{uni of f}
\begin{rm}
The uniqueness of the power series $f(S,T)$ is not known.
In this paper, we fix this power series in the case where $\lambda(k_{\infty}^c/k) \geq 2$.
\end{rm}
\end{Rem}
%%%%%%%%%%%%%%%%%%%%%%%%%%%%%%%%%%%%%%%%%%%%%%%%%%%%%%%%%%%%%%%%%%%%%%%%%%%%%%%%%%%%%%%%%%%%%%%%%%%%%%%%%%%%%%%%%%%%%%
%%%%%%%%%%%%%%%%%%%%%%%%%%%%%%%%%%%%%%%%%%%%%%%%%%%%%%%%%%%%%%%%%%%%%%%%%%%%%%%%%%%%%%%%%%%%%%%%%%%%%%%%%
%%%%%%%%%%%%%%%%%%%%%%%%%%%%%%%%%%%%%%%%%%%%%%%%%%%%%%%%%%%%%%%%%%%%%%%%%5
%%%%%%%%%%%%%%%%%%%%%%%%%%%%%%%%%%%%%%%%%%%%%%%%%%%%%%%%%%%%%%%%%%%%%%%%%%%
We fix an isomorphism
\begin{eqnarray}\label{isom cyc}
\mathbb{Z}_p[[ \mathrm{Gal}(k_{\infty}^c/k)]]  \cong \mathbb{Z}_p[[T]] \qquad
(\tau \mathrm{Gal}(\widetilde{k}/k_{\infty}^c)  \leftrightarrow 1+T).
\end{eqnarray}
%%%%%%%%%%%%%%%%%%%%%
%%%%%%%%%%%%%%%%%%%%%%%%%%%%%%%%%%%%%%%%%%%%%%%%%%%%%%%%%%%%%%%%%%%%%%%%%%%%%%%%%%%%%%%%%%%%%%%%%%%%%%%%
%%%%%%%%%%%%%%%%%%%%%%%%%%%%%%%%%%%%%%%%%%%%%%%%%%%%%%%%%%%%%%%%%%%%%%%%%%%%%%%%%%%%%%%%%%%%%%%%%%%%%%%%
By this isomorphism, we identify these rings.\par
There exists a relation between the power series $f(S,T)$
and the characteristic ideal of $X_{k_{\infty}^c}$ as follows.
%$\mathrm{char}_{{\mathbb{Z}_p}[[T]]}(X_{k_{\infty}^c})$.
%%%%%%%%%%%%%%%%%%%%%%%%%%%%%%%%%%%%%%%%%%%%%%%%%%%%%%%%%%%%%%%%%%%%%%%%%%%%%%%%%%%%%%%%%%%%%%%%%%%%%%%%%%%%%%%%%%%%%%%
%%%%%%%%%%%%%%%%%%%%%%%%%%%%%%%%%%%%%%%%%%%%%%%%%%%%%%%%%%%%%%%%%%%%%%%%%%%%%%%%%%%%%%%%%%%%%%%%%%%%%%%%%
%%%%%%%%%%%%%%%%%%%%%%%%%%%%%%%%%%%%%%%%%%%%%%%%%%%%%%%%%%%%%%%%%%%%%%%%%%%%%%%%%%%%%%%%%%%%%%%%%%%%%%%%%%%%%%%%%%%
\begin{prop}\label{char cyc}
Assume that the characteristic ideal 
%$\mathrm{char}_{{\mathbb{Z}_p}[[T]]}(X_{k_{\infty}^c})$
of $X_{k_{\infty}^c}$
has a generator which is square-free
or $X_{\widetilde{k}}$ is cyclic as a $\Lambda$-module.
%has no double roots in an algebraic closure of $\mathbb{Q}_p$.
Suppose that $\lambda(k_{\infty}^c/k) \geq 2$.
Then we have
\[
\mathrm{char}_{\mathbb{Z}_p[[T]]} (X_{k_{\infty}^c})
= (Tf(0,T)),
\]
where $f(S,T)$ is the same power series defined in Lemma \ref{f(S,T)}.
\end{prop}
\begin{proof}
%By Lemma \ref{f(S,T)},
Since $X_{\widetilde{k}}$ is a finitely generated $\Lambda$-module,  
there exists a positive integer $r$ such that
$\left( \Lambda /f(S,T) \Lambda \right)^{\oplus r} \rightarrow  X_{\widetilde{k}}$
is surjective.
We can take $r=1$ if $X_{\widetilde{k}}$ is cyclic as a $\Lambda$-module.
Hence we get a surjective homomorphism
$\left( \mathbb{Z}_p[[T]] /f(0,T) \mathbb{Z}_p[[T]] \right)^{\oplus r}
\rightarrow  X_{\widetilde{k}}/S X_{\widetilde{k}}$.
This implies that
\[
(Tf(0,T)^r) \subset  \mathrm{char}_{{\mathbb{Z}_p}[[T]]}(X_{k_{\infty}^c}).
\]
By Lemma \ref{f(S,T)}, we note that
$f(0,T)$ is a polynomial with $\mathrm{deg}(f(0,T))=\lambda^{*}$,
where  $\mathrm{deg}(f(0,T))$ is the degree
of the polynomial $f(0,T)$.
Since  $\mathrm{char}_{{\mathbb{Z}_p}[[T]]}(X_{k_{\infty}^c})$
has a generator which is square-free in the case of $r\geq 2$,
%has no double roots in an algebraic closure of $\mathbb{Q}_p$,
%we have
%$(f(0,T)) \subset \mathrm{char}_{{\mathbb{Z}_p}[[T]]}(X_{k_{\infty}^c}).$ By Lemmas \ref{f(S,T)}, \ref{Ozaki},
we obtain
\[
(Tf(0,T)) \subset  \mathrm{char}_{{\mathbb{Z}_p}[[T]]}(X_{k_{\infty}^c}).
\]
By the definition of $\lambda^{*}$, we have
%$\mathrm{deg}(f(0,T))=\lambda^{*}$
$(Tf(0,T)) =  \mathrm{char}_{{\mathbb{Z}_p}[[T]]}(X_{k_{\infty}^c})$.
\end{proof}
%%%%%%%%%%%%%%%%%%%%%%%%%%%%%%%%%%%%%%%%%%%%%%%%%%%%%%%%%%%%%%%%%%%%%%%%%%%%%%%%%%%%%%%%%%%%%%%%%%%%%%
Using the structure theorem, we can calculate the order of $\left( X_{\widetilde{k}}\right)_{\mathrm{Gal}(\widetilde{k}/k)}$.
%%%%%%%%%%%%%%%%%%%%%%%%%%%%%%%%%%%%%%%%%%%%%%%%%%%%%%%%%%%%%%%%%%%%%%%%%%%%%%%%%%%%%%%%%%%%%%%%%%%%%%%%%
\begin{lem}\label{Momo}
Suppose that $\lambda(k_{\infty}^c/k) \geq 2$.
With the same notation as above, we have the following:
\[
\# \left( X_{\widetilde{k}}\right)_{\mathrm{Gal}(\widetilde{k}/k)} = \# (\mathbb{Z}_p/g_0(0)\mathbb{Z}_p),
\]
where $g_0(S)$ is the same power series defined in Lemma \ref{f(S,T)}.
\end{lem}
%%%%%%%%%%%%%%%%%%%%%%%%%%%%%%%%%%%%%%%%%%%%%%%%%%%%%%%%%%%%%%%%%%%%%%%%%%%%%%%%%%%%%%
\begin{proof}
Using the structure theorem, we have an injective pseudo-isomorphism
\[
\Phi : (X_{\widetilde{k}})_{\mathrm{Gal}(\widetilde{k}/k_{\infty}^c)}
\rightarrow 
\bigoplus_{j=1}^{l} \mathbb{Z}_p[[T]]/(f_{j}(T)^{n_j}),
\]
where $l$ and $n_j$'s are positive integers and $f_j(T)$'s are irreducible elements of $\mathbb{Z}_p[[T]]$.
By Lemma \ref{Ozaki} and Proposition \ref{char cyc},
the characteristic ideal of $(X_{\widetilde{k}})_{\mathrm{Gal}(\widetilde{k}/k_{\infty}^c)}$
is generated by $f(0,T)$.
Then we have a prime factorization $f(0,T)=\displaystyle{\prod_{j=1}^{l}} f_j(T)^{n_j}$.
By \cite[Proposition $6$]{J-S}, $f(0,T)$ is not divisible by $T$.
Hence we have a commutative diagram:
%$$
%0 \rightarrow (X_{\widetilde{k}})_{\mathrm{Gal}(\widetilde{k}/k_{\infty})}
%\rightarrow X_{k_{\infty}} \rightarrow \mathrm{Gal}(\widetilde{k}/k_{\infty}) \rightarrow 0.
%$$
\begin{eqnarray*}\label{diag}
\begin{CD}
0 @>>> (X_{\widetilde{k}})_{\mathrm{Gal}(\widetilde{k}/k_{\infty}^c)} @>\Phi>> \displaystyle{\bigoplus_{j=1}^{l} \mathbb{Z}_p[[T]]/(f_{j}(T)^{n_j})}
 @>>> C @>>> 0
\\
& & @V  ~T\times VV  @V ~T\times VV   @V  ~T\times VV & 
\\
0 @>>> (X_{\widetilde{k}})_{\mathrm{Gal}(\widetilde{k}/k_{\infty}^c)} @>\Phi>>\displaystyle{\bigoplus_{j=1}^{l} \mathbb{Z}_p[[T]]/(f_{j}(T)^{n_j})}
@>>> C @>>> 0,
\end{CD}
\end{eqnarray*}
where the vertical maps are multiplication by $T$ and $C$ is the cokernel of $\Phi$.
%Let $X_{k_{\infty}^c}[T]$ be the $T$-torsion subgroup of $X_{k_{\infty}^c}$.
%Then we can show that
%\[
%\# X_{k_{\infty}^c}/ (X_{k_{\infty}^c}[T] + TX_{k_{\infty}^c}) = \# (X_{\widetilde{k}})_{\mathrm{Gal}(\widetilde{k}/k)}
%\]
%because $T \mathrm{Gal}(\widetilde{k}/k_{\infty}^c) =0$ and $(X_{k_{\infty}^c}[T] + TX_{k_{\infty}^c})/T X_{k_{\infty}^c}$ is torsion-free.
%Here we define $X_{k_{\infty}}[\times T] = \mathrm{ker}().$
Since $C$ is finite, we obtain
%Furthermore, by the structure theorem of $\mathbb{Z}_p[[T]]$-modules and Proposition \ref{char cyc}, we obtain
\[
\# (X_{\widetilde{k}})_{\mathrm{Gal}(\widetilde{k}/k)} = \# (\mathbb{Z}_p/f(0,0) \mathbb{Z}_p) = \# (\mathbb{Z}_p/g_0(0) \mathbb{Z}_p).
\]
\end{proof}
%%%%%%%%%%%%%%%%%%%%%%%%%%%%%%%%%%%%%%%%%%%%%%%%%%%%%%%%%%%%%%%%%%%%%%%%%%%%%%%%%%%%%%%%%%%%%%%%%%%%%%%%%
%%%%%%%%%%%%%%%%%%%%%%%%%%%%%%%%%%%%%%%%%%%%%%%%%%%%%%%%%%%%%%%%%%%%%%%%%%%%%%%%%%%%%%%%%%%%%%%%%%%%%%%%%
In the following lemmas and propositions,
% \ref{index}, \ref{split},
%we can roughly describe the decomposition of the prime $p$ in $\mathrm{Gal}(\widetilde{k}/\mathbb{Q})$.
we consider the decomposition of the prime $p$ in $\widetilde{k}/\mathbb{Q}$.
We first note that the prime number $p$ is only finitely decomposed in $\widetilde{k}/\mathbb{Q}$.
This fact was verified by Minardi in his thesis using class field theory:
%%%%%%%%%%%%%%%%%%%%%%%%%%%%%%%%%%%%%%%%%%%%%%%%%%%%%%%%%%%%%%%%%%%%%%%%%%%%%%%%%%%%%%%%%%%%%%%%%%%%%%%%%%%%%%%%%%%%%%%%%%%
%%%%%%%%%%%%%%%%%%%%%%%%%%%%%%%%%%%%%%%%%%%%%%%%%%%%%%%%%%%%%%%%%%%%%%%%%%%%%%%%%%%%%%%%%%%%%%%%%%%%%%%%%%%%%%%%%%%%%%%%%%%
\begin{lem}\cite[Lemma 3.1]{Mi}\label{Minardi}
Let $\mathfrak{D_{p}}$ be the decomposition group of $\mathfrak{p}$ in $\mathrm{Gal}(\widetilde{k}/k)$.
Then $\mathfrak{D_p}$ has finite index in $\mathrm{Gal}(\widetilde{k}/k)$.
\end{lem}
%%%%%%%%%%%%%%%%%%%%%%%%%%%%%%%%%%%%%%%%%%%%%%%%%%%%%%%%%%%%%%%%%%%%%%%%%%%%%%%%%%%%%%%%%%%%%%%%%%%%%%%%%%%%%%%%%%%%%%%%%
%%%%%%%%%%%%%%%%%%%%%%%%%%%%%%%%%%%%%%%%%%%%%%%%%%%%%%%%%%%%%%%%%%%%%%%%%%%%%%%%%%%%%%%%%%%%%%%%%%%%%%%%%%%%%%%%%%%%%%%%%
%%%%%%%%%%%%%%%%%%%%%%%%%%%%%%%%%%%%%%%%%%%%%%%%%%%%%%%%%%%%%%%%%%%%%%%%%%%%%%%%%%%%%%%%%%%%%%%%%%%%%%%%%%%%%%%%%%%%%%%%%
%%%%%%%%%%%%%%%%%%%%%%%%%%%%%%%%%%%%%%%%%%%%%%%%%%%%%%%%%%%%%%%%%%%%%%%%%%%%%%%%%%%%%%%%%%%%%%%%%%%%%%%%%%%%%%%%%%%%%%%%%%%
We give a necessary and sufficient condition
for the decomposition group $\mathfrak{D_{p}}$ to be normal in $\mathrm{Gal}(\widetilde{k}/\mathbb{Q})$.
%%%%%%%%%%%%%%%%%%%%%%%%%%%%%%%%%%%%%%%%%%%%%%%%%%%%%%%%%%%%%%%%%%%%%%%%%%%%%%%%%%%%%%%%%%%%%%%%%%%%%%%%%%%%%%%%%%%%%%%%%
%%%%%%%%%%%%%%%%%%%%%%%%%%%%%%%%%%%%%%%%%%%%%%%%%%%%%%%%%%%%%%%%%%%%%%%%%%%%%%%%%%%%%%%%%%%%%%%%%%%%%%%%%%%%%%%%%%%%%%%%%%%
\begin{lem}\label{cond normal}
%Let $\mathfrak{D_{p}}$ be the decomposition group of $\mathfrak{p}$ in $\mathrm{Gal}(\widetilde{k}/k)$.
%Then
The decomposition group $\mathfrak{D_p}$ is a normal subgroup of $\mathrm{Gal}(\widetilde{k}/\mathbb{Q})$
if and only if
$[\mathrm{Gal}(\widetilde{k}/k): \mathfrak{D_{p}}] \leq  [L_{k}\cap \widetilde{k}:k].$
\end{lem}
%%%%%%%%%%%%%%%%%%%%%%%%%%%%%%%%%%%%%%%%%%%%%%%%%%%%%%%%%%%%%%%%%%%%%%%%%%%%%%%%%%%%%%%%%%%%%%%%%%%%%%%%%%%%%%%%%%%%%%%%%
%%%%%%%%%%%%%%%%%%%%%%%%%%%%%%%%%%%%%%%%%%%%%%%%%%%%%%%%%%%%%%%%%%%%%%%%%%%%%%%%%%%%%%%%%%%%%%%%%%%%%%%%%%%%%%%%%%%%%%%%%
\begin{proof}
First, we note that $k_{\infty}^a /\mathbb{Q}$ is a galois extension. Furthermore, the complex conjugation, namely,
the generator of $\mathrm{Gal}(k/\mathbb{Q})$ acts on $\mathrm{Gal}(k_{\infty}^a /k)$ as inverse. 
Then we have $L_{k} \cap \widetilde{k} \subset k_{\infty}^a$.
By the definition of $N_{\infty}$,
$k_{\infty}^a \cap N_{\infty}$ coincides with $ L_{k} \cap \widetilde{k}$.
Since $\mathfrak{p}^{\ast}$ splits completely in $\widetilde{k}^{\mathfrak{D}_{\mathfrak{p}^{\ast}}}$,
we have $\widetilde{k}^{\mathfrak{D}_{\mathfrak{p}^{\ast}}} \subset N_{\infty}$.\par
We suppose that
$ \mathfrak{D_{p}}$ is a normal subgroup of $\mathrm{Gal}(\widetilde{k}/\mathbb{Q})$.
Then we have $\widetilde{k}^{\mathfrak{D_{p}}}= \widetilde{k}^{\mathfrak{D}_{\mathfrak{p}^{\ast}}}$.
Hence $p$ splits completely in $\widetilde{k}^{\mathfrak{D_{p}}}$
and $\widetilde{k}^{\mathfrak{D_{p}}}$ is a subfield of $k_{\infty}^a$.
This implies that $[\mathrm{Gal}(\widetilde{k}/k): \mathfrak{D_{p}}] \leq  [L_{k}\cap \widetilde{k}:k]$.
Conversely, we suppose that $[\mathrm{Gal}(\widetilde{k}/k): \mathfrak{D_{p}}] \leq  [L_{k}\cap \widetilde{k}:k].$
%%%%%%%%%%%%%%%%%%%%%%%%%%%%%%%%%%%%%%%%%%%%%%%%%%%%%%%%%%%%%%%%%%%%%%%%%%%%%%%%%%%%%%%%%%%%%%%%%%%%%%%%%%%%%%%%%%%%%%%%%%%%%
%%%%%%%%%%%%%%%%%%%%%%%%%%%%%%%%%%%%%%%%%%%%%%%%%%%%%%%%%%%%%%%%%%%%%%%%%%%%%%%%%%%%%%%%%%%%%%%%%%%%%%%%%%%%%%%%%%%%%%%%%%%%%
%$\mathfrak{D_{p}}$ is a normal subgroup of $\mathrm{Gal}(\widetilde{k}/\mathbb{Q})$ if and only if
%$\widetilde{k}^{\mathfrak{D_{p}}} / \mathbb{Q}$ is a galois extension.
%Conversely, we suppose that $ \mathfrak{D_{p}}$ is not a normal subgroup of $\mathrm{Gal}(\widetilde{k}/\mathbb{Q})$.
%Then we have $[\mathrm{Gal}(\widetilde{k}/k): \mathfrak{D_{p}}] >  [L_{k}\cap \widetilde{k}:k]$.
%Then $\widetilde{k}^{\mathfrak{D}_{\mathfrak{p}^{\ast}}} \neq \widetilde{k}^{\mathfrak{D}_{\mathfrak{p}}}$.
Since $L_{k}\cap \widetilde{k}$ and $\widetilde{k}^{\mathfrak{D}_{\mathfrak{p}}}$ are subfields of the $\mathbb{Z}_p$-extension
field $N_{\infty}^{\ast}$,
we have $\widetilde{k}^{\mathfrak{D}_{\mathfrak{p}}} \subset L_{k}\cap \widetilde{k}$.
Then $\widetilde{k}^{\mathfrak{D}_{\mathfrak{p}}}/\mathbb{Q}$ is a galois extension.
This implies that $\mathfrak{D_p}$ is normal.
\end{proof}
%%%%%%%%%%%%%%%%%%%%%%%%%%%%%%%%%%%%%%%%%%%%%%%%%%%%%%%%%%%%%%%%%%%%%%%%%%%%%%%%%%%%%%%%%%%%%%%%%%%%%%%%%%%%%%%%%%%%%%%%%
From the following proposition,
we see that the $p$-adic valuation of the constant term of $f(S,T)$ gives an upper bound of the number of prime ideals of $\widetilde{k}$
lying above $p$. 
%%%%%%%%%%%%%%%%%%%%%%%%%%%%%%%%%%%%%%%%%%%%%%%%%%%%%%%%%%%%%%%%%%%%%%%%%%%%%%%%%%%%%%%%%%%%%%%%%%%%%%%%%%%%%%%%%%%%%%%%
\begin{prop}\label{index}
Suppose that $\lambda(k_{\infty}^c/k) \geq 2$.
We have the following:
\begin{equation*}
[\mathrm{Gal}(\widetilde{k}/k): \mathfrak{D_{p}}] 
\leq 
\mathrm{min}
\{
[L_{k}\cap \widetilde{k}:k], \# \left( \mathbb{Z}_p/g_0(0)\mathbb{Z}_p \right)
\}
\end{equation*}
\hspace{3.5cm} if $\mathfrak{D_{p}}$ is a normal subgroup of $\mathrm{Gal}(\widetilde{k}/\mathbb{Q})$,
\begin{equation*}
 [L_{k}\cap \widetilde{k}:k] < [\mathrm{Gal}(\widetilde{k}/k): \mathfrak{D_{p}}] \leq \# \left( \mathbb{Z}_p/g_0(0)\mathbb{Z}_p \right)
\end{equation*}
\hspace{3.5cm} if $ \mathfrak{D_{p}}$ is not a normal subgroup of $\mathrm{Gal}(\widetilde{k}/\mathbb{Q}).$
\end{prop}
\begin{proof}
%%%%%%%%%%%%%%%%%%%%%%%%%%%%%%%%%%%%%%%%%%%%%%%%%%%%%%%%%%%%%%%%%%%%%%%%%%%%%
%We suppose that
%$ \mathfrak{D_{p}}$ is a normal subgroup of $\mathrm{Gal}(\widetilde{k}/\mathbb{Q})$.
%From Lemma \ref{cond normal}, we have
%$[\mathrm{Gal}(\widetilde{k}/k): \mathfrak{D_{p}}] \leq  [L_{k}\cap \widetilde{k}:k].$
%%%%%%%%%%%%%%%%%%%%%%%%%%%%%%%%%%%%%%%%%%%%%%%%%%%%%%%%%%%%%
%We will prove that $ [L_{k}\cap \widetilde{k}:k] \leq \# \left( \mathbb{Z}_p/g_0(0)\mathbb{Z}_p \right)$.
%By Proposition \ref{char cyc},
%$Tf(0,T)$ is a generator of the characteristic ideal of $X_{k_{\infty}^c}$.
%Using (\cite{F-G}), we obtain
%\[
%\mathrm{ord}_p(g_0(0)) = \mathrm{ord}_p \left( \frac{h_k \cdot \mathrm{log}_p(\alpha^{p-1})}{p} \right), \quad \alpha^{p-1} \equiv 1 \mathrm{~mod~} \frak{p^{\ast}},
%\]
%where $\alpha$ is an element of $\mathcal{O}_k$ satisfying
%$\frak{p}^t=(\alpha)$ for some positive integer $t$ not divisible by $p$,
%$h_k$ is the class number of $k$, and $\mathrm{log}_p$ is the $p$-adic logarithm.
%Hence we have $[L_k \cap \widetilde{k}:k] \leq \mathrm{ord}_p(h_k) \leq \mathrm{ord}_p(g_0(0))$.\par
%%%%%%%%%%%%%%%%%%%%%%%%%%%%%%%%%%%%%%%%%%%%%%%%%%%%%%%%%%%%%%%%%%%%%%%%%%%%%%%%%%%%%%%%%%%%%%%%%%%%%%%%%%%%%%%%
%Next we suppose that
%$ \mathfrak{D_{p}}$ is not a normal subgroup of $\mathrm{Gal}(\widetilde{k}/\mathbb{Q}).$
We will prove that $[\mathrm{Gal}(\widetilde{k}/k): \mathfrak{D_{p}}] \leq \# \left( \mathbb{Z}_p/g_0(0)\mathbb{Z}_p \right)$.
By Lemma \ref{Ozaki}, we have an exact sequence
\begin{eqnarray*}
0 \rightarrow (X_{\widetilde{k}})_{\mathrm{Gal}(\widetilde{k}/k_{\infty}^c)}
\rightarrow X_{k_{\infty}^c} \rightarrow \mathrm{Gal}(\widetilde{k}/k_{\infty}^c) \rightarrow 0 \label{ex cyc}
\end{eqnarray*}
as $\mathbb{Z}_p[[\mathrm{Gal}(k_{\infty}^c/k)]]$-modules. 
The snake lemma gives the exact sequence
\begin{eqnarray*}
0 &\rightarrow & \left( (X_{\widetilde{k}})_{\mathrm{Gal}(\widetilde{k}/k_{\infty}^c)} \right)^{\mathrm{Gal}(k_{\infty}^c/k)}
\rightarrow \left( X_{k_{\infty}^c} \right)^{\mathrm{Gal}(k_{\infty}^c/k)}
\rightarrow \left( \mathrm{Gal}(\widetilde{k}/k_{\infty}^c) \right)^{\mathrm{Gal}(k_{\infty}^c/k)} \\ 
& \rightarrow &
\left( (X_{\widetilde{k}})_{\mathrm{Gal}(\widetilde{k}/k_{\infty}^c)} \right)_{\mathrm{Gal}(k_{\infty}^c/k)}
\rightarrow  \left( X_{k_{\infty}^c} \right)_{\mathrm{Gal}(k_{\infty}^c/k)} \rightarrow
\left( \mathrm{Gal}(\widetilde{k}/k_{\infty}^c) \right)_{\mathrm{Gal}(k_{\infty}^c/k)} \rightarrow 0. \nonumber
\end{eqnarray*}
%%%%%%%%%%%%%%%%%%%%%%%%%%%%%%%%%%%%%%%%%%%%%%%%%%%%%%%%%%%%%%%%%%%%%%%%%%%%%%%%%%%%%%%%%%%%%%%%%%%%%%%%%%%
%This implies that
Then we have
$\left( (X_{\widetilde{k}})_{\mathrm{Gal}(\widetilde{k}/k_{\infty}^c)} \right)^{\mathrm{Gal}(k_{\infty}^c/k)}=0$.
Indeed,
$T$ does not divide a generator of
$\mathrm{char}_{\mathbb{Z}_p[[T]]}\left((X_{\widetilde{k}})_{\mathrm{Gal}(\widetilde{k}/k_{\infty}^c)} \right)$.
%since
%$\mathrm{char}_{\mathbb{Z}_p[[T]]}({\mathrm{Gal}(\widetilde{k}/k_{\infty}^c)} )$
%$=(T)$.
%is generated by $T$.
%Hence we obtain
Furthermore, we have $\left( X_{k_{\infty}^c} \right)^{\mathrm{Gal}(k_{\infty}^c/k)}=D_{k_{\infty}^c}$ by \cite[Lemma 4.1]{Okano},
where $D_{k_{\infty}^c} $ is the decomposition group of a prime ideal of $k_{\infty}^c$ lying above $p$ in
$X_{k_{\infty}^c} =\mathrm{Gal}(L_{k_{\infty}^c}/k_{\infty}^c)$.
Therefore we obtain an exact sequence
\begin{eqnarray}\label{ex index}
0 &\rightarrow &  D_{k_{\infty}^c}
\rightarrow  \mathrm{Gal}(\widetilde{k}/k_{\infty}^c) \\  \label{snake ex} 
& \rightarrow &  (X_{\widetilde{k}})_{\mathrm{Gal}(\widetilde{k}/k)} 
\rightarrow  \left( X_{k_{\infty}^c} \right)_{\mathrm{Gal}(k_{\infty}^c/k)}  \rightarrow
\mathrm{Gal}(\widetilde{k}/k_{\infty}^c)  \rightarrow 0. \nonumber
\end{eqnarray}
%where $L$ is the maximal unramified abelian pro-$p$ extension of $k$ in $L_{k_{\infty}^c}$.
%%%%%%%%%%%%%%%%%%%%%%%%%%%%%%%%%%%%%%%%%%%%%%%%%%%%%%%%%%%%%%%%%%%%%%%%%%%%%%%%%%%%%%%%%%%
This implies that
\begin{eqnarray*}
\mathrm{Coker} (D_{k_{\infty}^c} \rightarrow \mathrm{Gal}(\widetilde{k}/k_{\infty}^c))
&\cong &
\mathrm{Gal}(\widetilde{k}/k_{\infty}^c) / \mathrm{ker}(\mathrm{Gal}({\widetilde{k}/k_{\infty}^c})
\rightarrow \left( X_{\widetilde{k}}\right)_{\mathrm{Gal}(\widetilde{k}/k)})\\
&\cong &
\mathrm{Image}(\mathrm{Gal}({\widetilde{k}/k_{\infty}^c})
\rightarrow \left( X_{\widetilde{k}}\right)_{\mathrm{Gal}(\widetilde{k}/k)}).
% \\
%&\subset &
% (X_{\widetilde{k}})_{\mathrm{Gal}(\widetilde{k}/k)}.
\end{eqnarray*}
%%%%%%%%%%%%%%%%%%%%%%%%%%%%%%%%%%%%%%%%%%%%%%%%%%%%%%%%%%%%%%%%%%%%%%%%%%%%%%%%%%%%%%%%%%%
We note that
\[
\mathrm{Image} (D_{k_{\infty}^c}  \rightarrow \mathrm{Gal}(\widetilde{k}/k_{\infty}^c) )
=
\mathrm{Gal}(\widetilde{k}/k_{\infty}^c) \cap \mathfrak{D_{p}}.
\]
%By Lemma \ref{Momo},
From the exact sequence (\ref{ex index}), we obtain
%%%%%%%%%%%%%%%%%%%%%%%%%%%%%%%%%%%%%%%%%%%%%%%%%%%%%%%%%%%%%%%%%%%%%%%%%%%%%%%%%%%%%%%%%%%
\begin{eqnarray*}
 [\mathrm{Gal}(\widetilde{k}/k) : \mathfrak{D_{p}}] 
&=&
[\mathrm{Gal}(\widetilde{k}/k_{\infty}^c) \mathfrak{D_{p}} : \mathfrak{D_{p}}] \\
&= &
[\mathrm{Gal}(\widetilde{k}/k_{\infty}^c)  : \mathrm{Gal}(\widetilde{k}/k_{\infty}^c)  \cap \mathfrak{D_{p}}] \\
&=&
\# \mathrm{Coker} (D_{k_{\infty}^c} \rightarrow \mathrm{Gal}(\widetilde{k}/k_{\infty}^c))\\
%\rightarrow \left( X_{\widetilde{k}}\right)_{\mathrm{Gal}(\widetilde{k}/k)} \\
&\leq &
\#  (X_{\widetilde{k}})_{\mathrm{Gal}(\widetilde{k}/k)}.
%\\
%&= &
%\# \left( \mathbb{Z}_p/g_0(0) \mathbb{Z}_p \right).
\end{eqnarray*}
%We note that we have the last equation using the structure theorem
%(see for example \cite{MOMO}).
By Lemma \ref{Momo}, we have completed the proof.
\end{proof}
%%%%%%%%%%%%%%%%%%%%%%%%%%%%%%%%%%%%%%%%%%%%%%%%%%%%%%%%%%%%%%%%%%%%%%%%%%%%%%%%%%%%%%%%%%%%%%%%%%%%%%%%%%%%%%%%%%%%%%%%%%%
%%%%%%%%%%%%%%%%%%%%%%%%%%%%%%%%%%%%%%%%%%%%%%%%%%%%%%%%%%%%%%%%%%%%%%%%%%%%%%%%%%%%%%%%%%%%%%%%%%%%%%%%%%%%%%%%%%%%%%%%%%%%%%
%As a corollary of Proposition \ref{split}, we obtain
Let $D_{\mathfrak{p}}$ (respectively, $D_{\mathfrak{p}^\ast}$)
be the decomposition group of $\mathfrak{p}$
(respectively, $\mathfrak{p}^\ast$) in $\mathrm{Gal}(N_{\infty}/k)$ (respectively, $\mathrm{Gal}(N_{\infty}^{\ast}/k))$.
%%%%%%%%%%%%%%%%%%%%%%%%%%%%%%%%%%%%%%%%%%%%%%%%%%%%%%%%%%%%%%%%%%%%%%%%%%%%%%%%%%%%%%%%%%%%%%%%%%%%%%%%%%%
%%%%%%%%%%%%%%%%%%%%%%%%%%%%%%%%%%%%%%%%%%%%%%%%%%%%%%%%%%%%%%%%%%%%%%%%%%%%%%%%%%%%%%%%%%%%%%%%%%%%%%%%%%5
\begin{lem}\label{indexD}
With the same notation as above, we have the following:
\begin{eqnarray*}\label{D}
[\mathrm{Gal}(\widetilde{k}/k):\mathfrak{D_{p}}] &=& [\mathrm{Gal}(N_{\infty}^{\ast}/k):D_{\mathfrak{p}}],\\
~[\mathrm{Gal}(\widetilde{k}/k):\mathfrak{D_{p^{\ast}}}] &=& [\mathrm{Gal}(N_{\infty}/k):D_{\mathfrak{p}^{\ast}}].
%[\mathrm{Gal}(\widetilde{k}/k):\mathfrak{D_{p^\ast}}] = [\mathrm{Gal}(N_{\infty}/k):D_{\mathfrak{p^\ast}}].
%[\mathrm{Gal}(\widetilde{k}/k):\mathfrak{D_{p}}] = [\mathrm{Gal}(N_{\infty}^{\ast}/k):D_{\mathfrak{p}}],
%[\mathrm{Gal}]
%[\mathrm{Gal}(\widetilde{k}/k):\mathfrak{D_{p}}] = [\mathrm{Gal}(N_{\infty}^{\ast}/k):D_{\mathfrak{p}}]
%[\mathrm{Gal}(\widetilde{k}/k):\mathfrak{D_{p}}] = [\mathrm{Gal}(N_{\infty}^{\ast}/k):D_{\mathfrak{p}}].
\end{eqnarray*}
\end{lem}
%%%%%%%%%%%%%%%%%%%%%%%%%%%%%%%%%%%%%%%%%%%%
\begin{proof}
From the natural exact sequence
%%%%%%%%%%%%%%%%%%%%%%%%%%%%%%%%%%%%
\begin{eqnarray*}
0 &\rightarrow &
\mathrm{Gal}(\widetilde{k}/N_{\infty})
%\left( (X_{\widetilde{k}})_{\mathrm{Gal}(\widetilde{k}/k_{\infty}^c)} \right)^{\mathrm{Gal}(k_{\infty}^c/k)}
\rightarrow \mathrm{Gal}(\widetilde{k}/k)
%\left( X_{k_{\infty}^c} \right)^{\mathrm{Gal}(k_{\infty}^c/k)}
\rightarrow \mathrm{Gal}(N_{\infty}/k)
\rightarrow 0,
%\left( \mathrm{Gal}(\widetilde{k}/k_{\infty}^c) \right)^{\mathrm{Gal}(k_{\infty}^c/k)} \label{snake ex} \\ 
%& \rightarrow &
%\left( (X_{\widetilde{k}})_{\mathrm{Gal}(\widetilde{k}/k_{\infty}^c)} \right)_{\mathrm{Gal}(k_{\infty}^c/k)}
%\rightarrow  \left( X_{k_{\infty}^c} \right)_{\mathrm{Gal}(k_{\infty}^c/k)} \rightarrow
%\left( \mathrm{Gal}(\widetilde{k}/k_{\infty}^c) \right)_{\mathrm{Gal}(k_{\infty}^c/k)} \rightarrow 0. \nonumber
\end{eqnarray*}
we have
\begin{eqnarray*}
0 &\rightarrow & \mathrm{Gal}(\widetilde{k}/N_{\infty}) \cap \mathfrak{D}_{\mathfrak{p}^{\ast}}
%\left( (X_{\widetilde{k}})_{\mathrm{Gal}(\widetilde{k}/k_{\infty}^c)} \right)^{\mathrm{Gal}(k_{\infty}^c/k)}
\rightarrow  \mathfrak{D}_{\mathfrak{p}^{\ast}}
%\left( X_{k_{\infty}^c} \right)^{\mathrm{Gal}(k_{\infty}^c/k)}
\rightarrow D_{\mathfrak{p}^{\ast}}
\rightarrow 0.
\end{eqnarray*}
Since all prime ideals of $N_{\infty}$ lying above $\mathfrak{p}^{\ast}$ ramify in
$\widetilde{k}/N_{\infty}$, we have $\mathrm{Gal}(\widetilde{k}/N_{\infty}) \cap \mathfrak{D}_{\mathfrak{p}^{\ast}}=\mathrm{Gal}(\widetilde{k}/N_{\infty})$.
Thus we obtain
$\mathrm{Gal}(\widetilde{k}/k)/\mathfrak{D}_{\mathfrak{p}^{\ast}} \cong \mathrm{Gal}(N_{\infty}/k)/D_{\mathfrak{p}^{\ast}}$.
Hence we have proved the former part.
We can prove that
$[\mathrm{Gal}(\widetilde{k}/k):\mathfrak{D_{p}}] = [\mathrm{Gal}(N_{\infty}^{\ast}/k):D_{\mathfrak{p}}]$
in the same way as above.
Thus we get the conclusion.
\end{proof}
%%%%%%%%%%%%%%%%%%%%%%%%%%%%%%%%%%%%%%%%%%%%%%%%%%%%%%%%%%%%%%%%%%%%%%%%%%%%%%%%%%%%%%%%%%%%%%%%%%%%%%%%%%%%%%%%%%%5
%%%%%%%%%%%%%%%%%%%%%%%%%%%%%%%%%%%%%%%%%%%%%%%%%%%%%%%%%%%%%%%%%%%%%%%%%%%%%%%%%%%%%%%%%%%%%%%%%%%%
%\begin{Rem}
%We have the same result for $\mathfrak{p}$. Hence we have
%\begin{eqnarray*}
%[\mathrm{Gal}(\widetilde{k}/k):\mathfrak{D_{p}}]=[\mathrm{Gal}(N_{\infty}^{\ast}/k):D_{\mathfrak{p}}].
%\end{eqnarray*}
%where $D_{\mathfrak{p}}$ be the decomposition group of $\mathfrak{p}$ in $N_{\infty}^{\ast}/k$.
%\end{Rem}
%%%%%%%%%%%%%%%%%%%%%%%%%%%%%%%%%%%%%%%%%%%%%%%%%%%%%%%%%%%%%%%%%%%%%%%%%%%%%%%%%%%%%%%%%%%%%%%%%%%%%%%%%%%%%%%%
%%%%%%%%%%%%%%%%%%%%%%%%%%%%%%%%%%%%%%%%%%%%%%%%%%%%%%%%%%%%%%%%%%%%%%%%%%%%%%%%%%%%%%%%%%%%%%%%%%%%%%%%%%%%%%%%%%%5
\begin{prop}\label{split}
%$\mathrm{char}_{\mathbb{Z}_p[[T]]} (X_{k_{\infty}^c}) = (Tf(0,T))$.
Suppose that
$ \mathfrak{D_{p^{\ast}}}$ is not a normal subgroup of $\mathrm{Gal}(\widetilde{k}/\mathbb{Q}).$
Then,
%all primes of $k$ lying above 
$p$ splits completely in $L_k \cap \widetilde{k}$.
\end{prop}
%%%%%%%%%%%%%%%%%%%%%%%%%%%%%%%%%%%%
\begin{proof}
%\begin{eqnarray}
%[\mathrm{Gal}(\widetilde{k}/k):\mathfrak{D_{p^\ast}}]=[\mathrm{Gal}(N_{\infty}/k):D_{\mathfrak{p^\ast}}].
%\end{eqnarray}
We have
%$\mathrm{Gal}(\widetilde{k}/k)/\mathfrak{D_{p^\ast}} \cong \mathrm{Gal}(N_{\infty}/k)/D_{\mathfrak{p^\ast}}$
$[\mathrm{Gal}(\widetilde{k}/k):\mathfrak{D_{p^{\ast}}}] = [\mathrm{Gal}(N_{\infty}/k):D_{\mathfrak{p}^{\ast}}]$
by Lemma \ref{indexD}.
Thus we have $N_{\infty}^{D_{\mathfrak{p^\ast}}} = \widetilde{k}^{\mathfrak{D_{p^\ast}}}$.
%, where $N^{D_{\mathfrak{p^\ast}}}$ is the fixed field by ${D_{\mathfrak{p^\ast}}}$ and $\widetilde{k}^{\mathfrak{D_{p^\ast}}}$, .
Since ${D_{\mathfrak{p^\ast}}}$ is not a normal subgroup of $\mathrm{Gal}(N_{\infty}/\mathbb{Q})$, we have
\[
N_{\infty}^{D_{\mathfrak{p^\ast}}} \supset L_{k} \cap \widetilde{k},
\quad N_{\infty}^{D_{\mathfrak{p^\ast}}} \neq L_{k} \cap \widetilde{k}.
\]
Indeed, if we suppose that $N_{\infty}^{D_{\mathfrak{p^\ast}}} \subset L_{k} \cap \widetilde{k}$,
then $N_{\infty}^{D_{\mathfrak{p^\ast}}}$ is a subfield of $k_{\infty}^a$.
Hence $N_{\infty}^{D_{\mathfrak{p^\ast}}}$ is galois over $\mathbb{Q}$.
Therefore ${{\mathfrak{D}}_{\mathfrak{p^\ast}}}$ is a normal subgroup of $\mathrm{Gal}(\widetilde{k}/\mathbb{Q})$.
This is a contradiction.
By the same reason, the inclusion
$N_{\infty}^{D_{\mathfrak{p}}}  \supset L_{k} \cap \widetilde{k}$ 
is proper.
%\[
% \quad N^{D_{\mathfrak{p}}} \neq L_{k} \cap \widetilde{k}.
%\]
These imply that $\mathfrak{p}$ and $\mathfrak{p^{\ast}}$ split completely in $L_k \cap \widetilde{k}$.
Thus we get the conclusion.
\end{proof}
%%%%%%%%%%%%%%%%%%%%%%%%%%%%%%%%%%%%%%%%%%%%%%%%%%%%%%%%%%%%%%%%%%%%%%%%%%%%%%%%%%%%%%%%%%%%%%%%%%%%%%%%
%%%%%%%%%%%%%%%%%%%%%%%%%%%%%%%%%%%%%%%%%%%%%%%%%%%%%%%%%%%%%%%%%%%%%%%%%%%%%%%%%%%55%%%%%%%%%%%%%%%%%%$$%%%%%%%%%%%%%
For an algebraic extension $K/k$,
we denote by $M_{\mathfrak{p}}(K)$ (respectively, $M_{\mathfrak{p}^{\ast}}(K)$) the maximal pro-$p$ abelian extension of $K$
unramified outside all prime ideals of $K$ lying above
$\mathfrak{p}$ (respectively, $\mathfrak{p}^{\ast}$).
To prove Theorem \ref{main thm}, we will show two inequalities below
under the assumption that weak GGC does not hold for $p$ and $k$:
\begin{eqnarray*}\label{AB}
%(\mathrm{Gal}(M_{\mathfrak{p^{\ast}}}(\widetilde{k})/\widetilde{k}}))
%\lambda^{\ast}:=
%\mathrm{rank}_{\mathbb{Z}_p}\left( (X_{\widetilde{k}})_{\mathrm{Gal}(\widetilde{k}/k_{\infty}^c)} \right) \label{n}
%=
\begin{cases}
(\mathrm{A}) ~\mathrm{rank}_{\mathbb{Z}_p} (\mathrm{Gal} (M_{\mathfrak{p^{\ast}}}(N_{\infty})/ \widetilde{k}))
\leq \mathrm{min}\{ [L_k \cap \widetilde{k}:k], [\mathrm{Gal}(\widetilde{k}/k):\mathfrak{D_{p^\ast}}] \} -1,\\
(\mathrm{B})~\mathrm{rank}_{\mathbb{Z}_p} (\mathrm{Gal} (M_{\mathfrak{p^{\ast}}}(N_{\infty})/ \widetilde{k}))
\geq 
%[L_k \cap \widetilde{k}:k].
\mathrm{min}\{ [L_k \cap \widetilde{k}:k], [\mathrm{Gal}(\widetilde{k}/k):\mathfrak{D_{p^\ast}}] \},
\end{cases}
\end{eqnarray*}
where we denote by $\mathrm{rank}_{\mathbb{Z}_p}(\ast)$ the $\mathbb{Z}_p$-rank of $\ast$. 
%where $M_{\mathfrak{p^{\ast}}}(N_{\infty})$ is the maximal pro-$p$ abelian extension of $N_{\infty}$
%unramified outside all prime ideals of $N_{\infty}$ lying above $\frak{p}^{\ast}$.
If both (A) and (B) hold, then this is a contradiction.
Thus weak GGC holds for $p$ and $k$.
In Sections \ref{A} and \ref{B}, we will prove these inequalities.\par
Finally, we introduce the following proposition needed later. For convenience we include a proof.
%%%%%%%%%%%%%%%%%%%%%%%%%%%%%%%%%%%%%%%%%%%%%%%%%%%%%%%%%%%%%%%%%%%%%%%%%%%%%%%%%%%%%%%%%%%%%%%%%%%%%%%%%%%%%%%%%%%%%%%%%%
\begin{prop}{\rm{(\cite[Theorem 2]{Fuj})}}\label{Fuj}
The following two conditions are equivalent:\\
$\mathrm{(i)}$
%Weak {\rm{GGC}} holds for $p$ and $k$.\\
The Iwasawa module $X_{\widetilde{k}}$ has a non-trivial pseudo-null $\Lambda$-submodule.\\
$\mathrm{(ii)}$ We have $M_{\mathfrak{p^\ast}}(\widetilde{k}) \neq L_{\widetilde{k}}$.
%$\mathrm{Gal}(L_{\widetilde{k}}/\widetilde{k})$ 
\end{prop}
%%%%%%%%%%%%%%%%%%%%%%%%%%%%%%%%%%%%%%%%%%%%%%%%%%%%%%%%%%%%%%%%%%%%%%%%%%%%%%%%%%%%%%%%%%%%%%%%%%%%%%%%%%%%%%%%%%%%%%%5
%%%%%%%%%%%%%%%%%%%%%%%%%%%%%%%%%%%%%%%%%%%%%%%%%%%%%%%%%%%%%%%%%%%%%%%%%%%%%%%%%%%%%%%%%%%%%%%%%%%%%%%%%%%%%55
\begin{proof}
We first suppose that $M_{\mathfrak{p^\ast}}(\widetilde{k}) = L_{\widetilde{k}}$.
Then we have $X_{\frak{p}^{\ast}}(\widetilde{k})=X_{\widetilde{k}}$,
where $X_{\mathfrak{p}^{\ast}}(\widetilde{k})$ is the galois group of the extension
$M_{\mathfrak{p^\ast}}(\widetilde{k}) / \widetilde{k}$.
By (\cite{PR}), $X_{\mathfrak{p}^{\ast}}(\widetilde{k})$ has no non-trivial pseudo-null $\Lambda$-submodule.
Hence $\mathrm{(i)}$ implies that $\mathrm{(ii)}$ holds.\par
Next we suppose that $\mathrm{(ii)}$ holds.
%Using local class field theory and
Using Lemma $2$ in \cite{Fuj},
we obtain
%\begin{eqnarray}
$M_{\mathfrak{p^{\ast}}}(k_{\infty}^c) =L_{k_{\infty}^c}$.
%\end{eqnarray}
%by Lemma $2$ in \cite{Fuj}.
%where $M_{\mathfrak{p^{\ast}}}(k_{\infty}^c) $ is the maximal pro-$p$ abelian extension of $K$
%unramified outside all prime ideals of $k_{\infty}^c$ lying above $\mathfrak{p}^{\ast}$.
Hence we have a commutative diagram of $\mathbb{Z}_p[[T]]$-modules:
%\begin{eqnarray*}
%0\rightarrow X_{\widetilde{k}}/S X_{\widetilde{k}} \rightarrow X_{k_{\infty}^c} \rightarrow \mathrm{Gal}(\widetilde{k}/k_{\infty}^c) \rightarrow 0,\\
%0\rightarrow X_{\mathfrak{p^{\ast}}}(\widetilde{k})/S  X_{\mathfrak{p^{\ast}}}(\widetilde{k}) \rightarrow X_{k_{\infty}^c}
%\rightarrow \mathrm{Gal}(\widetilde{k}/k_{\infty}^c) \rightarrow 0.
%\end{eqnarray*}
%%%%%%%%%%%%%%%%%%%%%%%%%%%%%%%%%%%%%%%%%%%%%%%%%%%%%%%%%%%%%%%%%%%%%%%%%%%%%%%%%5
\begin{eqnarray*}\label{diag}
\begin{CD}
0 @>>> X_{\mathfrak{p^{\ast}}}(\widetilde{k})/S  X_{\mathfrak{p^{\ast}}}(\widetilde{k})  @>>> X_{k_{\infty}^c} @>>> \mathrm{Gal}(\widetilde{k}/k_{\infty}^c) @>>> 0
\\
& & @V   VV @| @\vert    & 
\\
0 @>>>X_{\widetilde{k}}/S X_{\widetilde{k}}@>>> X_{k_{\infty}^c} @>>> \mathrm{Gal}(\widetilde{k}/k_{\infty}^c) @>>> 0.
\end{CD}
\end{eqnarray*}
%%%%%%%%%%%%%%%%%%%%%%%%%%%%%%%%%%%%%%%%%%%%%%%%%%%%%%%%%%%%%%%%%%%%%%%%%
This implies that
\begin{eqnarray}\label{LCT}
X_{\mathfrak{p^{\ast}}}(\widetilde{k})/S  X_{\mathfrak{p^{\ast}}}(\widetilde{k}) 
\cong X_{\widetilde{k}}/S X_{\widetilde{k}} \cong \mathbb{Z}_p^{\oplus \lambda^{\ast}}.
\end{eqnarray}
%%%%%%%%%%%%%%%%%%%%%%%%%%%%%%%%%%%%%%%%%%%%%%%%%%%%%%%%%%%%%%%%%%%%%%%%
In the case of $\lambda(k_{\infty}^c /k) =1$, we have proved that
$X_{\widetilde{k}}=0$ by Lemma \ref{trivial}.
Then we have $X_{\mathfrak{p^{\ast}}}(\widetilde{k}) =0$ by (\ref{LCT}).
This is a contradiction.
Hence we have $\lambda(k_{\infty}^c /k) \geq 2$.
In this case, we see that $X_{\widetilde{k}}$ is not trivial.
Hence, without loss of generality, we may assume that $f(S,T)$ annihilates $X_{\mathfrak{p^{\ast}}}(\widetilde{k})$,
where $f(S,T)$ is the same power series defined in Lemma \ref{f(S,T)}.
We put $Y=\mathrm{Gal}(M_{\mathfrak{p^\ast}}(\widetilde{k}) / L_{\widetilde{k}})$.
Then we have a commutative diagram of $\Lambda$-modules:
%$$
%0 \rightarrow (X_{\widetilde{k}})_{\mathrm{Gal}(\widetilde{k}/k_{\infty})}
%\rightarrow X_{k_{\infty}} \rightarrow \mathrm{Gal}(\widetilde{k}/k_{\infty}) \rightarrow 0.
%$$
\begin{eqnarray*}\label{diag}
\begin{CD}
0 @>>> Y @>>>X_{\mathfrak{p^{\ast}}}(\widetilde{k}) @>>>X_{\widetilde{k}} @>>> 0
\\
& & @V ~S\times VV  @V ~S\times VV   @V  ~S\times VV & 
\\
0 @>>> Y @>>> X_{\mathfrak{p^{\ast}}}(\widetilde{k}) @>>> X_{\widetilde{k}} @>>> 0,
\end{CD}
\end{eqnarray*}
where the vertical maps are multiplication by $S$.
Using (\ref{LCT}), we have a surjective homomorphism
\[
X_{\widetilde{k}} [S] \rightarrow Y/SY,
%X_{\mathfrak{p^{\ast}}}(\widetilde{k})/SX_{\mathfrak{p^{\ast}}}(\widetilde{k}) \rightarrow X_{\widetilde{k}}/SX_{\widetilde{k}} \rightarrow 0.
\]
where $X_{\widetilde{k}} [S]$ is the $S$-torsion subgroup of $X_{\widetilde{k}}$.
By Nakayama's lemma, $X_{\widetilde{k}} [S] $ is not trivial because of the assumption that $Y \neq 0$.
Furthermore, $X_{\widetilde{k}} [S]$ is pseudo-null since $S$ is coprime to $f(S,T)$.
Therefore $X_{\widetilde{k}}$ has a non-trivial pseudo-null submodule.
Thus we get the conclusion.
\end{proof}
%%%%%%%%%%%%%%%%%%%%%%%%%%%%%%%%%%%%%%%%%%%%%%%%%%%%%%%%%%%%%%%%%%%%%%%%%%%%%%%%%%%%%%%%%%%%%%%%%%%%%%%%%%%%%%%%%%%%%
%%%%%%%%%%%%%%%%%%%%%%%%%%%%%%%%%%%%%%%%%%%%%%%%%%%%%%%%%%%%%%%%%%%%%%%%%%%%%%%%%%%%%%%%%%%%%%%%%%%%%%%%%%%%%%%%%%
%%%%%%%%%%%%%%%%%%%%%%%%%%%%%%%%%%%%%%%%%%%%%%%%%%%%%%%%%%%%%%%%%%%%%%%%%%%%%%%%%%%%%%%%%%%%%%%%%%%%%%%%%%%%%%%%%%
\section{Proof of the inequality (A)} \label{A}
In this section, we will prove the inequality (A):
\begin{eqnarray*}
%(\mathrm{Gal}(M_{\mathfrak{p^{\ast}}}(\widetilde{k})/\widetilde{k}}))
%\lambda^{\ast}:=
%\mathrm{rank}_{\mathbb{Z}_p}\left( (X_{\widetilde{k}})_{\mathrm{Gal}(\widetilde{k}/k_{\infty}^c)} \right) \label{n}
%=
\mathrm{rank}_{\mathbb{Z}_p} (\mathrm{Gal} (M_{\mathfrak{p^{\ast}}}(N_{\infty})/ \widetilde{k}))
 \leq \mathrm{min}\{ [L_k \cap \widetilde{k}:k], [\mathrm{Gal}(\widetilde{k}/k):\mathfrak{D_{p^\ast}}] \} -1
\end{eqnarray*}
under the assumption
of Theorem \ref{main thm} and assuming that
weak GGC does not hold for $p$ and $k$.\par
%by the method of \cite{Mi}.
%We define some notation.
By Lemma \ref{Minardi}, the prime ideal $\frak{p}^{\ast}$ finitely decomposes in $\widetilde{k}/k$.
We put $p^{n_0} = [\mathrm{Gal}(\widetilde{k}/k) : \mathfrak{D_{p^{\ast}}}] $.
Using Lemma \ref{indexD}, we have
$$
p^{n_0}
= [\mathrm{Gal}(\widetilde{k}/k) : \mathfrak{D_{p^{\ast}}}] 
= [\mathrm{Gal}(N_{\infty}/k) : D_{\mathfrak{p^{\ast}}}].
$$
Let $k_{\infty}/k$ be a $\mathbb{Z}_p$-extension.
For each $n\geq 0 $, we denote by $k_n$ the intermediate field of the $\mathbb{Z}_p$-extension $k_\infty$
such that $k_n$ is the unique cyclic extension over $k$ of degree $p^n$.
%We define some notation.
Let $K/k$ be an algebraic extension.
For a prime ideal $\mathfrak{q}$ of $K$,
let $K_{\mathfrak{q}}$ be the completion of $K$ at $\mathfrak{q}$.
We denote by $U_{\mathfrak{q}}^{(1)}$ the principal local unit group with respect to $\mathfrak{q}$
in $K_{\mathfrak{q}}$.\\
\par
%%%%%%%%%%%%%%%%%%%%%%%%%%%%%%%%%%%%%%%%%%%%%%%%%%%%%%%%%%%%%%%%%%%%%%%%%%%%%%%%%%%%%%%%%%%%%%%%%%%%%%%%%%%%%%%%%%%%%%%%%%
\subsection{Normal case}~
%%%%%%%%%%%%%%%%%%%%%%%%%%%%%%%%%%%%%%%%%%%%%%%%%%%%%%%%%%%%%%%%%%%%%%%%%%%%%%%%%%%%%%%%%%%%%%%%%%%%%%%%%%%%%%%%%%%%%%%
%%%%%%%%%%%%%%%%%%%%%%%%%%%%%%%%%%%%%%%%%%%%%%%%%%%%%%%%%%%%%%%%%%%%%%%%%%%%%%%%%%%%%%%%%%%%%%%%%%%%%%%%%%%%%%%%%%%%%%%%%%%
In this subsection, we will prove the inequality (A)
in the case where $\mathfrak{D_{p}}$ is a normal subgroup of $\mathrm{Gal}(\widetilde{k}/\mathbb{Q})$.\par
%%%%%%%%%%%%%%%%%%%%%%%%%%%%%%%%%%%%%%%%%%%%%%%%%%%%%%%%%%%%%%%%%%%%%%%%%%%%%%%%%%%%%%%%%%%%%%%%%%%%%%%%%%%%%%%%%%%%%%%%
\begin{prop}\label{rank uper}
Assume the same condition $\mathrm{(i)}$ as in Theorem \ref{main thm}. 
Assume also that weak {\rm{GGC}} does not hold for $p$ and $k$
and that
$\mathfrak{D_{p}}$ is a normal subgroup of $\mathrm{Gal}(\widetilde{k}/\mathbb{Q})$.
Then the inequality $\mathrm{(A)}$ holds.
% have $\mathrm{rank}_{\mathbb{Z}_p} (\mathrm{Gal} (M_{\mathfrak{p^{\ast}}}(N_{\infty})/ \widetilde{k})) \leq p^{n_0}-1$.
%$\mathrm{char}_{\mathbb{Z}_p[[T]]} (X_{k_{\infty}^c}) = (Tf(0,T))$.
%All the primes of $k$ above $p$ split completely in $L_k \cap \widetilde{k}$.
%$\mathrm{rank}_{\mathbb{Z}_p} (\mathrm{Gal} (M_{{\mathfrak{p^{\ast}}}}(N_{n_0})/ N_{n_0})) =1$.
\end{prop}
%%%%%%%%%%%%%%%%%%%%%%%%%%%%%%%%%%%%%%%%%%%%%%%%%%%%%%%%%%%%%%%%%%%%%%%%%%%%%%%%%%%%%%%%%%%%%%%%%%%%%%%%%%%%%%%%%%%%%%%
%%%%%%%%%%%%%%%%%%%%%%%%%%%%%%%%%%%%%%%%%%%%%%%%%%%%%%%%%%%%%%%%%%%%%%%%%%%%%%%%%%%%%%%%%%%%%%%%%%%%%%%%%%%%%%%%%%%%%%%%
\begin{proof}
We assume that weak GGC does not hold for $p$ and $k$.
By Proposition \ref{Fuj}, we have $M_{\mathfrak{p^\ast}}(\widetilde{k}) = L_{\widetilde{k}}$.
By Lemma \ref{cond normal},
$p$ splits completely in $\widetilde{k}^{\mathfrak{D}_{{\mathfrak{p}}}}$.
Since $\mathfrak{D_{p}}$ is a normal subgroup of $\mathrm{Gal}(\widetilde{k}/\mathbb{Q})$,
we have 
$\widetilde{k}^{\mathfrak{D}_{{\mathfrak{p}}}}=\widetilde{k}^{\mathfrak{D}_{{\mathfrak{p}}^{\ast}}}$.
%of $N_{\infty}$ lying above $\fark{p}~{\ast}$ are
%Let $\mathfrak{p}_{\infty,i}^{\ast}$ be a prime of $N_{\infty}$ lying above $\mathfrak{p}^\ast$
%for $1 \leq  i \leq p^{n_0}$.
Let
\[
\{ \mathfrak{p}_{\infty,i}^{\ast} ~|~ 1 \leq  i \leq p^{n_0}\}
\]
be the set of prime ideals of $N_{\infty}$ lying above $\mathfrak{p}^\ast$.
We denote by $ I_{\mathfrak{p}_{\infty,i}^{\ast}}$ the inertia subgroup of $\mathrm{Gal}(M_{\mathfrak{p^{\ast}}}(N_{\infty})/N_{\infty})$
for the prime ideal $\mathfrak{p}_{\infty,i}^{\ast}$.
Then we have an exact sequence
\[
0 \rightarrow  \displaystyle{\sum_{i=1}^{p^{n_0}} I_{\mathfrak{p}_{\infty,i}^{\ast}}} \rightarrow \mathrm{Gal}(M_{\mathfrak{p^{\ast}}}(N_{\infty})/ N_{\infty})
\rightarrow \mathrm{Gal}(L_{N_{\infty}}/ N_{\infty}) \rightarrow 0.
\]
Since $M_{\mathfrak{p^{\ast}}}(N_{\infty}) / \widetilde{k}$ is an unramified extension,
we have $\mathrm{Gal}(M_{\mathfrak{p^\ast}}(N_{\infty})/\widetilde{k}) \cap I_{\mathfrak{p}_{\infty,i}^{\ast}} =1$.
Hence $I_{\mathfrak{p}_{\infty,i}^{\ast}}$ is isomorphic to $\mathbb{Z}_p$ for each $i$.
From \cite{Gi, sch, OV},
we obtain $\mu(N_{\infty}/k)=0.$
Furthermore,
since we suppose that $\lambda({N_{\infty}}/k)=0$, $\mathrm{Gal}(L_{N_{\infty}}/N_{\infty})$ is finite.
This implies that
$\mathrm{Gal}(M_{\mathfrak{p^{\ast}}}(N_{\infty})/ N_{\infty})$ is a finitely generated $\mathbb{Z}_p$-module.
Therefore we obtain
$\mathrm{rank}_{\mathbb{Z}_p}(\mathrm{Gal}(M_{\mathfrak{p^{\ast}}}(N_{\infty})/ \widetilde{k})) \leq p^{n_0}-1$.
%By the definition of $n_{0}$,
By Lemma \ref{cond normal},
we have $p^{n_0} \leq [L_k \cap \widetilde{k} :k]$.
Thus we get the conclusion.
\end{proof}~\\
%%%%%%%%%%%%%%%%%%%%%%%%%%%%%%%%%%%%%%%%%%%%%%%%%%%%%%%%%%%%%%%%%%%%%%%%%%%%%%%%%%%%%%%%%%%%%%%%%%%%%%%%%%%%%%%%%%%%%%%
%%%%%%%%%%%%%%%%%%%%%%%%%%%%%%%%%%%%%%%%%%%%%%%%%%%%%%%%%%%%%%%%%%%%%%%%%%%%%%%%%%%%%%%%%%%%%%%%%%%%%%%%%%%%%%%%%%%%%%%
\subsection{Non-normal case}~
%%%%%%%%%%%%%%%%%%%%%%%%%%%%%%%%%%%%%%%%%%%%%%%%%%%%%%%%%%%%%%%%%%%%%%%%%%%%%%%%%%%%%%%%%%%%%%%%%%%%%%%%%%%%%%%%%%%%%%%%
%%%%%%%%%%%%%%%%%%%%%%%%%%%%%%%%%%%%%%%%%%%%%%%%%%%%%%%%%%%%%%%%%%%%%%%%%%%%%%%%%%%%%%%%%%%%%%%%%%%%%%%%%%%%%%%%%%%%%%%%%%%
In this subsection, we will prove the inequality (A) in the case where $\mathfrak{D_{p}}$ is not a normal subgroup of $\mathrm{Gal}(\widetilde{k}/\mathbb{Q})$.
%%%%%%%%%%%%%%%%%%%%%%%%%%%%%%%%%%%%%%%%%%%%%%%%%%%%%%%%%%%%%%%%%%%%%%%%%%%%%%%%%%%%%%%%%%%%%%%%%%%%%%%%%%%%%%%%%%%%%%%%
\begin{prop}\label{rank n0}
%Assume the same condition $\mathrm{(i)}$ as in Theorem \ref{main thm}.
%Assume also that weak {\rm{GGC}} does not hold for $p$ and $k$
%and that
%$\mathfrak{D_{p}}$ is not a normal subgroup of $\mathrm{Gal}(\widetilde{k}/\mathbb{Q})$.
With the same notation as above, we have
%$\mathrm{char}_{\mathbb{Z}_p[[T]]} (X_{k_{\infty}^c}) = (Tf(0,T))$.
%All the primes of $k$ above $p$ split completely in $L_k \cap \widetilde{k}$.
\[
\mathrm{rank}_{\mathbb{Z}_p} (\mathrm{Gal} (M_{{\mathfrak{p^{\ast}}}}(N_{n_0})/ N_{n_0})) =1.
\]
\end{prop}
%%%%%%%%%%%%%%%%%%%%%%%%%%%%%%%%%%%%%%%%%%%%%%%%%%%%%%%%%%%%%%%%%%%%%%%%%%%%%%%%%%%%%%%%%%%%%%%%%%%%%%%%%%%%%%%%%%%5
\begin{proof}
Let $\mathcal{O}_{N_{n_0}}$ be the ring of integers in $N_{n_0}$ and
\[
\{ \mathfrak{p}_{n_0,i}^{\ast} ~|~ 1 \leq  i \leq p^{n_0}\}
\]
be the set of prime ideals of $N_{n_0}$ lying above $\mathfrak{p}^\ast$.
We put $E_{N_{n_0}}=\{ u \in \mathcal{O}_{N_{n_0}}^{\times}~|~u \equiv 1 ~\mathrm{mod}~\frak{p}_{n_0,i}^{\ast}~(1 \leq i \leq p^{n_0}) \}$.
Since $N_{n_0}/k$ is an abelian extension,
Leopoldt's conjecture holds for $p$ and $N_{n_0}$ (\cite{Bu}).
Using class field theory, we obtain
%\[
%\mathrm{rank}_{\mathbb{Z}_p}(\mathrm{Gal}(M_{\mathfrak{p^{\ast}}}(N_{n_0})/N_{n_0}))=1.
%\] 
%In fact, we have
\begin{eqnarray*}
\mathrm{rank}_{\mathbb{Z}_p}(\mathrm{Gal}(M_{\mathfrak{p^{\ast}}}(N_{n_0})/N_{n_0}))
&=&\mathrm{rank}_{\mathbb{Z}_p}(\mathrm{Gal}(M_{\mathfrak{p^{\ast}}}(N_{n_0})/L_{N_{n_0}}))\\
&=&\mathrm{rank}_{\mathbb{Z}_p}\left(
\prod_{i=1}^{p^{n_0}} U_{\mathfrak{p}_{n_0,i}^{\ast}}^{(1)}/\overline{\varphi (E_{N_{n_0}})} \right)\\
%\overline{E_{N_{n_0}}}} )\\
%&=&\mathrm{rank}_{\mathbb{Z}_p[[\mathrm{Gal}(N_{\infty}/k)]]}(\displaystyle{\sum_{i=1^{p^{n_0}}} I_{\mathfrak{p}_{\infty,i}^{\ast}}})\\
&=&\mathrm{rank}_{\mathbb{Z}_p} \left(
\sum_{i=1}^{p^{n_0}}[(N_{n_0})_{\mathfrak{p}_{n_0,i}^{\ast}} : \mathbb{Q}_p] \right)
- \mathrm{rank}_{\mathbb{Z}_p}(\overline{\varphi (E_{N_{n_0}})})\\
%&=&p^{n_0} -(p^{n_0}-1)\\
&=&1,
\end{eqnarray*}
where
%$E_{N_{n_0}}^{(1)}$ is the group of principal units of $N_{n_0}$ and
$\varphi : E_{N_{n_0}} \rightarrow \displaystyle{\prod_{i=1}^{p^{n_0}}U_{\mathfrak{p}_{n_0,i}^{\ast}}^{(1)}}$
is the diagonal map,
and $\overline{\varphi (E_{N_{n_0}})}$ is the topological closure of $\varphi (E_{N_{n_0}})$
in $\displaystyle{\prod_{i=1}^{p^{n_0}} U_{\mathfrak{p}_{n_0,i}^{\ast}}^{(1)}}$.
Thus we get the conclusion.
\end{proof}
We obtain the inequality (A) by the following
%%%%%%%%%%%%%%%%%%%%%%%%%%%%%%%%%%%%%%%%%%%%%%%%%%%%%%%%%%%%%%%%%%%%%%%%%%%%%%%%%%%%%%%%%%%%%%%%%%%%%%%%%%%%%%%%%%%
%%%%%%%%%%%%%%%%%%%%%%%%%%%%%%%%%%%%%%%%%%%%%%%%%%%%%%%%%%%%%%%%%%%%%%%%%%%%%%%%%%%%%%%%%%%%%%%%%%%%%%%%%%%%%%%%%%%
\begin{prop}\label{(A)}
%Assume that weak Greenberg's generalized conjecture does not hold for $p$ and $k$.
%Assume also that
%$\mathfrak{D_{p}}$ is not a normal subgroup of $\mathrm{Gal}(\widetilde{k}/\mathbb{Q})$.
%Under the same assumption as in Proposition \ref{rank n0},
Assume the same condition $\mathrm{(i)}$ as in Theorem \ref{main thm}.
Assume also that weak {\rm{GGC}} does not hold for $p$ and $k$
and that
$\mathfrak{D_{p}}$ is not a normal subgroup of $\mathrm{Gal}(\widetilde{k}/\mathbb{Q})$.
Then we have the inequality $\mathrm{(A)}$.
%$\mathrm{char}_{\mathbb{Z}_p[[T]]} (X_{k_{\infty}^c}) = (Tf(0,T))$.
%All the primes of $k$ above $p$ split completely in $L_k \cap \widetilde{k}$.
%$\mathrm{rank}_{\mathbb{Z}_p} (\mathrm{Gal} (M_{{\mathfrak{p^{\ast}}}}(N_{\infty})/ \widetilde{k}))
%\leq [L_{\widetilde{k}} \cap \widetilde{k}:k]-1$.
\end{prop}
%%%%%%%%%%%%%%%%%%%%%%%%%%%%%%%%%%%%%%%%%%%%%%%%%%%%%%%%%%%%%%%%%%%%
\begin{proof}
%We assume that weak Greenberg's generalized conjecture does not hold for $p$ and $k$.
%By Proposition \ref{Fuj}, we have $M_{\mathfrak{p^\ast}}(\widetilde{k}) = L_{\widetilde{k}}$.
We note that $\mathfrak{p^{\ast}}$ splits completely in $N_{\infty}^{ D_{\mathfrak{p^{\ast}}}} = N_{n_0}$
and that all prime ideals of $ N_{n_0}$
lying above $\mathfrak{p}^{\ast}$ do not split in $N_{\infty}$.
%For each $i$ with $1 \leq  i \leq p^{n_0}$,
%let $\mathfrak{p}_{\infty,i}^{\ast}$ be a prime of $N_{\infty}$ lying above $\mathfrak{p}^\ast$
Let
\[
\{ \mathfrak{p}_{\infty,i}^{\ast} ~|~ 1 \leq  i \leq p^{n_0}\}
\]
be the set of prime ideals of $N_{\infty}$ lying above $\mathfrak{p}^\ast$
and
$I_{\mathfrak{p}_{\infty,i}^{\ast}}$ the inertia subgroup of
$\mathrm{Gal}(M_{\mathfrak{p^{\ast}}}(N_{\infty})/ N_{\infty})$
for the prime ideal
$\mathfrak{p}_{\infty,i}^{\ast}$.
By the same reason as in the proof of Proposition \ref{rank uper},
%Since $M_{\mathfrak{p^{\ast}}}(N_{\infty}) / \widetilde{k}$ is an unramified extension,
%we have
%$\mathrm{Gal}(M_{\mathfrak{p^\ast}}(N_{\infty})/\widetilde{k}) \cap I_{\mathfrak{p}_{\infty,i}^{\ast}} =1$.
%Then we have an injective homomorphism
%\begin{eqnarray}\label{inj}
% I_{\mathfrak{p}_{\infty,i}^{\ast}}  \rightarrow \mathrm{Gal}(\widetilde{k}/N_{\infty})
%\end{eqnarray}
%as $\mathbb{Z}_p[[\mathrm{Gal}(N_{\infty}/N_{n_0})]]$-modules.
%%%%%%%%%%%%%%%%%%%%%%%%%%%%%%%%%%%%%%%%%%%%%%%%%%%%%%%%%%%%%%%%%%%%%%%%%%%%%%%%%%%%%%%%%%%%%%%%%%%%%%%%%%%%%%%%%%%%%%
%This implies that
we have
$I_{\mathfrak{p}_{\infty,i}^{\ast}} \cong \mathbb{Z}_p$ and $\mu(N_{\infty}/k)=0$.
%Thus $\mathrm{Gal}(M_{\mathfrak{p^{\ast}}}(N_{\infty})/ N_{\infty})$ is finitely generated $\mathbb{Z}_p$-module.
%Indeed,
Furthermore, the fixed field of
$M_{\mathfrak{p^{\ast}}}(N_{\infty})$ by $\displaystyle{\sum_{i=1}^{p^{n_0}} I_{\mathfrak{p}_{\infty,i}^{\ast}}}$
coincides with $L_{N_{\infty}}$.
Thus $\mathrm{Gal}(M_{\mathfrak{p^{\ast}}}(N_{\infty})/ N_{\infty})$ is a finitely generated $\mathbb{Z}_p$-module,
since we suppose that $\lambda(N_{\infty}/k)=0$.
%%%%%%%%%%%%%%%%%%%%%%%%%%%%%%%%%%%%%%%%%%%%%%%%%%%%%%%%%%%%%%%%%
%Let $M'$ be the fixed field of $M_{\mathfrak{p^{\ast}}}(N_{\infty})$ by the torsion part as $\mathbb{Z}_p$-modules.
Let $M'$ be the submodule of $\mathbb{Z}_p$-torsion in $\mathrm{Gal}(M_{\mathfrak{p^{\ast}}}(N_{\infty})/N_{\infty})$.
%by the torsion part as $\mathbb{Z}_p$-modules.
Hence we have an exact sequence
\[
0\rightarrow M' \rightarrow \mathrm{Gal}(M_{\mathfrak{p^{\ast}}}(N_{\infty})/N_{\infty}) \rightarrow
\mathrm{Gal}(M_{\mathfrak{p^{\ast}}}(N_{\infty})/N_{\infty})/M' \rightarrow 0 
\]
as $\mathbb{Z}_p[[\mathrm{Gal}(N_{\infty}/k)]]$-modules.
Since $M'$ is a finite $\mathbb{Z}_p$-module, we have
\begin{eqnarray}\label{char}
&~&
\mathrm{char}_{\mathbb{Z}_p[[\mathrm{Gal}(N_{\infty}/k)]]}(\mathrm{Gal}(M_{\mathfrak{p^{\ast}}}(N_{\infty})/N_{\infty})) \\
%&=&\mathrm{char}_{\mathbb{Z}_p[[\mathrm{Gal}(N_{\infty}/k)]]}(\mathrm{Gal}(M_{\mathfrak{p^{\ast}}}(N_{\infty})/N_{\infty})/M') \\
&=&\mathrm{char}_{\mathbb{Z}_p[[\mathrm{Gal}(N_{\infty}/k)]]}
\left(\displaystyle{\sum_{i=1}^{p^{n_0}} I_{\mathfrak{p}_{\infty,i}^{\ast}}}\right).\nonumber
\end{eqnarray}
For each $i$ with $1 \leq i \leq p^{n_0}$, the decomposition group $D_{\mathfrak{p}^{\ast}}$ acts on
$I_{\mathfrak{p}_{\infty,i}^{\ast}} $
because  all prime ideals of $ N_{n_0}$
lying above $\mathfrak{p}^{\ast}$ do not split in $N_{\infty}$.
Since we have $\mathrm{Gal}(M_{\mathfrak{p^\ast}}(N_{\infty})/\widetilde{k}) \cap I_{\mathfrak{p}_{\infty,i}^{\ast}} =1$,
there exists an injective homomorphism
\begin{eqnarray}\label{inj}
 I_{\mathfrak{p}_{\infty,i}^{\ast}}  \rightarrow \mathrm{Gal}(\widetilde{k}/N_{\infty})
\end{eqnarray}
as $\mathbb{Z}_p[[\mathrm{Gal}(N_{\infty}/N_{n_0})]]$-modules.
%%%%%%%%%%%%%%%%%%%%%%%%%%%%%%%
Then $D_{\mathfrak{p}^{\ast}}$ acts on $I_{\mathfrak{p}_{\infty,i}^{\ast}}$ trivially
because  $\widetilde{k}/N_{n_0}$ is an abelian extension.
%and all the prime ideals of $N_{n_0}$ lying above $\mathfrak{p^{\ast}}$
%do not split in $N_{\infty}$.
Hence $D_{\mathfrak{p}^{\ast}}$ acts on
$\displaystyle{\sum_{i=1}^{p^{n_0}} I_{\mathfrak{p}_{\infty,i}^{\ast}}}$ trivially.
Let $F$ be the fixed field of $M_{\mathfrak{p^{\ast}}}(N_{\infty})$
by $M'$.
By (\ref{char}) and (\ref{inj}),
%Therefore
$F/N_{n_0}$ is an abelian extension.\par
%%%%%%%%%%%%%%%%%%%%%%%%%%%%%%%%%%%%%%%%%%%%%%%%%%%%%%%%%%%%%
Let $s$ be a non-negative integer such that $p^s= [L_{k} \cap \widetilde{k}:k]$.
By Proposition \ref{split}, $p$ splits completely in $L_{k} \cap \widetilde{k}$.
We note that all prime ideals of $L_k \cap \widetilde{k}$ lying above $\mathfrak{p}$
do not split in $N_{\infty}$.
%For each $i$ with $1 \leq  i \leq p^{s}$,
%let $\mathfrak{p}_{n_0,i}$ be a prime of $N_{n_0}$ lying above $\mathfrak{p}$
Let
\[
\{ \mathfrak{p}_{n_0,i}~|~1 \leq  i \leq p^{s}\}
\]
be the set of prime ideals of $N_{n_0}$ lying above $\mathfrak{p}$
and
$I_{\mathfrak{p}_{n_0,i}}$ the inertia subgroup of $\mathrm{Gal}(F/ N_{n_0})$
for the prime ideal $\mathfrak{p}_{n_0,i}$.
%in the abelian extension.
%where $F$ is the same fixed field defined in the proof of Proposition \ref{rank n0}.
Furthermore, by the same method as in the proof of Proposition \ref{rank uper},
we have $I_{\mathfrak{p}_{n_{0},i}} \cong \mathbb{Z}_p$ for each $i$.
Since the fixed field in $F$
by $\displaystyle{\sum_{i=1}^{p^{s}} I_{\mathfrak{p}_{n_0,i}}}$
is contained in $M_{\mathfrak{p}^{\ast}}(N_{n_0})$,
we obtain
\[
\mathrm{rank}_{\mathbb{Z}_p}(\mathrm{Gal}(M_{\mathfrak{p}^{\ast}}(N_{\infty})/N_{n_0})) =
\mathrm{rank}_{\mathbb{Z}_p}(\mathrm{Gal}(F/N_{n_0})) \leq p^s +1
\]
by Proposition \ref{rank n0}.
Using Proposition \ref{index}, we have $p^s < p^{n_0}$.
Thus we get the conclusion.
\end{proof}
%%%%%%%%%%%%%%%%%%%%%%%%%%%%%%%%%%%%%%%%%%%%%%
%%%%%%%%%%%%%%%%%%%%%%%%%%%%%%%%%%%%%%%%%%%%%%%%%%%%%%%%%%%%%%%%%%%%%%%%%%%%%%%%%%%%%%%%%%%%%%%%%%%%%%%%%%%%%%%%%%
\section{Proof of the inequality (B)} \label{B}
In this section, we will prove the inequality (B):
\begin{eqnarray*}
\mathrm{rank}_{\mathbb{Z}_p} (\mathrm{Gal} (M_{\mathfrak{p^{\ast}}}(N_{\infty})/ \widetilde{k})) \geq
%[L_k \cap \widetilde{k}:k].
\mathrm{min}\{ [L_k \cap \widetilde{k}:k], [\mathrm{Gal}(\widetilde{k}/k):\mathfrak{D_{p^\ast}}] \}
\end{eqnarray*}
%%%%%%%%%%%%%%%%%%%%%%%%%%%%%%%%%%%%%%%%%%%%%%%%%%%%%%%%%%%%%%%%%%%%%%%%%%%%%%%%%%%%%%%%%%%%%%%%%%%%%%%%%%%%%%%%%%%
%%%%%%%%%%%%%%%%%%%%%%%%%%%%%%%%%%%%%%%%%%%%%%%%%%%%%%%%%%%%%%%%%%%%%%%%%%%%%%%%%%%%%%%%%%%%%%%%%%%%%%%%%%%%%%%%%%%
under the assumption
of Theorem \ref{main thm} and assuming that
weak GGC does not hold for $p$ and $k$.~\\
%%%%%%%%%%%%%%%%%%%%%%%%%%%%%%%%%%%%%%%%%%%%%%%%%%%%%%%%%%%%%%%%%%%%%%%%%%%%%%%%%%%%%%%%%%%%%%%%%%%%%%%%%%%%%%%%
%%%%%%%%%%%%%%%%%%%%%%%%%%%%%%%%%%%%%%%%%%%%%%%%%%%%%%%%%%%%%%%%%%%%%%%%%%%%%%%%%%%%%%%%%%%%%%%%%%%%%%%%%%%%%%%%%
\subsection{Module theory}~
In this subsection, we will prove some module theoretical properties needed later
%under the assumption
%of Theorem \ref{main thm} and
%of the failure of
provided that
weak GGC does not hold for $p$ and $k$.
%We note that,
%we may assume that $X_{\widetilde{k}}\neq 0$ as we stated after Lemma \ref{Ozaki}.\par
From Lemma \ref{trivial},
we have $\lambda(k_{\infty}^c/k) \geq 2$ if weak GGC does not hold for $p$ and $k$.
%%%%%%%%%%%%%%%%%%%%%%%%%%%%%%%%%%%%%%%%%%%%%%%%%%%%%%%%%%%%%%%%%%%%%%%%%%%%%%%%%%%%%%%%%%%%%
Let $\mathrm{Ass}_{\Lambda}(X_{\widetilde{k}})$ be the set of associated prime ideals of $X_{\widetilde{k}}$.
In other words, we put
\[
\mathrm{Ass}_{\Lambda}(X_{\widetilde{k}})
= \{\mathfrak{p} : \mathrm{~prime~ideal~}|~\mathfrak{p}=\mathrm{Ann}_{\Lambda}(x) ~\mathrm{for~some~element~} x \mathrm{~of}~X_{\widetilde{k}} \},
\]
where we write $\mathrm{Ann}_{\Lambda}(x)=\{ a \in \Lambda ~|~ ax=0\}$.
%%%%%%%%%%%%%%%%%%%%%%%%%%%%%%%%%%%%%%%%%%%%%%%%%%%%%%%%%%%%%%%%%%%%%%%%%%%%%%%%%%%%%%%%%%%%%
\begin{lem}\label{principal}
Assume that weak {\rm{GGC}} does not hold for $p$ and $k$.
Let $Y$ be a $\Lambda$-submodule of $X_{\widetilde{k}}$.
If $Y$ is cyclic as a $\Lambda$-module, in other words,
%there exists an ideal $I$ such that
$Y$ is isomorphic to $\Lambda/\mathrm{Ann}_{\Lambda}(Y)$, then $\mathrm{Ann}_{\Lambda}(Y)$ is a principal ideal.
In particular, an associated prime ideal of $X_{\widetilde{k}}$ is principal.
\end{lem}
%%%%%%%%%%%%%%%%%%%%%%%%%%%%%%%%%%%%%%%%%%%%%%%%%%%%%%%%%%%%%%%%%%%%%%%%%%%%%%%%%%%%%%%%%
%%%%%%%%%%%%%%%%%%%%%%%%%%%%%%%%%%%%%%%%%%%%%%%%%%%%%%%%%%%%%%%%%%%%%%%%%%%%%%%%%%%%%%%%%
\begin{proof}
We suppose that $Y$ is isomorphic to $\Lambda/\mathrm{Ann}_{\Lambda}(Y)$.
Let $g,h$ be elements of $\mathrm{Ann}_{\Lambda}(Y)$.
We denote by $G$ the greatest common divisor for $g$ and $h$.
Then there exist $g',h' \in \Lambda$ such that $g=g'G, h=h'G$, and $g'$ is coprime to $h'$.
Then $GY$ is a pseudo-null $\Lambda$-submodule of $X_{\widetilde{k}}$.
Since $X_{\widetilde{k}}$ has no non-trivial pseudo-null $\Lambda$-submodule, we get $GY=0$.
Hence we have $G \in \mathrm{Ann}_{\Lambda}(Y)$.
This implies that $(g,h) \subset (G) \subset \mathrm{Ann}_{\Lambda}(Y)$ as an ideal of $\Lambda$.
Since $\mathrm{Ann}_{\Lambda}(Y)$ is a finitely generated $\Lambda$-module,
we can prove that
$\mathrm{Ann}_{\Lambda}(Y)$ is a principal ideal, inductively.
\end{proof}
%%%%%%%%%%%%%%%%%%%%%%%%%%%%%%%%%%%%%%%%%%%%%%%%%%%%%%%%%%%%%%%%%%
%%%%%%%%%%%%%%%%%%%%%%%%%%%%%%%%%%%%%%%%%%%%%%%%%%%%%%%%%%%%%%%%%%%%%%%%%%%%%%%%%%%%%%%%%
If $X_{\widetilde{k}}$ is cyclic as a $\Lambda$-module, we can determine the isomorphism class
of $X_{\widetilde{k}}$.
%%%%%%%%%%%%%%%%%%%%%%%%%%%%%%%%%%%%%%%%%%%%%%%%%%%%%%%%%%%%%%%%%%
%%%%%%%%%%%%%%%%%%%%%%%%%%%%%%%%%%%%%%%%%%%%%%%%%%%%%%%%%%%%%%%%%%%%%%%%%%%%%%%%%%%%%%%%%
\begin{prop}\label{isom str}
%Let $k$ be a $p$-split $p$-rational field.
Assume that $X_{\widetilde{k}}$ is cyclic as a $\Lambda$-module.
Assume also that weak {\rm{GGC}} does not hold for $p$ and $k$. Then we have
\[
X_{\widetilde{k}} \cong \Lambda/f(S,T) \Lambda
\]
as $\Lambda$-modules, where $f(S,T)$ is the same annihilator of $X_{\widetilde{k}}$ defined in Lemma \ref{f(S,T)}.
\end{prop}
%%%%%%%%%%%%%%%%%%%%%%%%%%%%%%%%%%%%%%%%%%%%%%
\begin{proof}
%By Proposition \ref{g00}, $X_{\widetilde{k}}$ is cyclic as a $\Lambda$-module.
Since $X_{\widetilde{k}}$ is cyclic,
we have $X_{\widetilde{k}} \cong \Lambda/ \mathrm{Ann}_{\Lambda}(X_{\widetilde{k}}).$
By Lemma \ref{principal}, $\mathrm{Ann}_{\Lambda}(X_{\widetilde{k}})$ is a principal ideal.
Let $G(S,T) \in \Lambda$ be a generator of $\mathrm{Ann}_{\Lambda}(X_{\widetilde{k}})$.
Using Lemma \ref{Ozaki}, we have $\mathrm{char}_{\mathbb{Z}_p[[T]]}(X_{k_{\infty}^c}) = (TG(0,T))$.
By Lemma \ref{f(S,T)}, there exists a power series $H(S,T) \in \Lambda$ such that
$f(S,T) =G(S,T)H(S,T).$
Using Proposition \ref{char cyc}, we have $\mathrm{char}_{\mathbb{Z}_p[[T]]}(X_{k_{\infty}^c}) = (Tf(0,T))$.
This implies that $H(S,T) \in \Lambda^{\times}$.
Therefore we have $\mathrm{Ann}_{\Lambda}(X_{\widetilde{k}})=(f(S,T))$.
Thus we get the conclusion.
\end{proof}
%%%%%%%%%%%%%%%%%%%%%%%%%%%%%%%%%%%%%%%%%%%%%%%%%%%%%%%%%%%%%%%%%%
%%%%%%%%%%%%%%%%%%%%%%%%%%%%%%%%%%%%%%%%%%%%%%%%%%%%%%%%%%%%%%%%%%%%%%%%%%%%%%%%%%%%%%%%%
%%%%%%%%%%%%%%%%%%%%%%%%%%%%%%%%%%%%%%%%%%%%%%%%%%%%%%%%%%%%%%%%%%
To consider the case where $X_{\widetilde{k}}$ is not $\Lambda$-cyclic, we prove the following
%%%%%%%%%%%%%%%%%%%%%%%%%%%%%%%%%%%%%%%%%%%%%%%%%%%%%%%%%%%%%%%%%%
%%%%%%%%%%%%%%%%%%%%%%%%%%%%%%%%%%%%%%%%%%%%%%%%%%%%%%%%%%%%%%%%%%%%%%%%%%%%%%%%%%%%%%%%%
\begin{lem}\label{an}
Suppose that $\lambda(k_{\infty}^c/k) \geq 2$
and that $X_{\widetilde{k}}$ is not cyclic as a $\Lambda$-module.
Let
\begin{eqnarray*}\label{factor}
\displaystyle{f(S,T)} = \prod_{i=1}^{l} f_{i}(S,T)^{n_i}
\end{eqnarray*}
be a prime factorization, in other words,
$l$ and $n_i$'s are positive integers and
$f_{i}(S,T)$'s are irreducible elements of $\Lambda$.
Assume that
the characteristic ideal of $X_{k_{\infty}^c}$
%$\mathrm{char}_{{\mathbb{Z}_p}[[\mathrm{Gal}(k_{\infty}^c/k)]]}(X_{k_{\infty}^c})$
has a generator which is square-free.
Then $f_{i}(S,T)$ is a zero-divisor of $X_{\widetilde{k}}$ for each $i$.
\end{lem}
%%%%%%%%%%%%%%%%%%%%%%%%%%%%%%%%%%%%%%%%%%%%%%%%
%%%%%%%%%%%%%%%%%%
%%%%%%%%%%%%%%%%%%%%%%%%%%%%%%%%%%%%%%%%%%%%%%%%%%%%%%%%%%%%%%%%%%%%%%%%%%%%%%%%%%%%%%%%%
\begin{proof}
We note that
$\mathrm{char}_{\Lambda}(X_{\widetilde{k}})$ has a square-free generator
because
the characteristic ideal
of
$X_{k_{\infty}^c}$
%$\mathrm{char}_{{\mathbb{Z}_p}[[\mathrm{Gal}(k_{\infty}^c/k)]]}(X_{k_{\infty}^c})$
also has a square-free generator.
Hence we have $n_i=1$ for each $i$.
In the case of $l =1$, in other words, $f(S,T)$ is an irreducible element,
$f(S,T)$ is a zero-divisor of $X_{\widetilde{k}}$ by Lemma \ref{f(S,T)}.
Thus we get the conclusion.
In the following, we suppose that $l \geq 2$.
We put $F_i(S,T)=\displaystyle{\prod_{j=1, j \neq i}^{l} f_{j}(S,T)}$.
We prove that
%$\displaystyle{F_i(S,T) X_{\widetilde{k}} \neq 0$.
$F_{i}(S,T)  X_{\widetilde{k}} \neq 0$.
Suppose that
%$\displaystyle{\prod_{j=1, j \neq i}^{\ell} f_{j}(S,T)} X_{\widetilde{k}} = 0$.
$F_{i}(S,T)  X_{\widetilde{k}} =0$.
Since $X_{\widetilde{k}}$ is a finitely generated $\Lambda$-module, there exists a positive integer $r$ and
a surjective homomorphism
\[
\left( \Lambda/ F_{i}(S,T) \Lambda \right)^{\oplus r}
\rightarrow 
X_{\widetilde{k}}.
\]
This homomorphism induces
\[
\left( \mathbb{Z}_p[[T]]/ F_{i}(0,T) \mathbb{Z}_p[[T]] \right)^{\oplus r}
\rightarrow 
X_{\widetilde{k}}/S X_{\widetilde{k}}.
\]
Since the characteristic ideal of $X_{k_{\infty}^c}$
%$\mathrm{char}_{{\mathbb{Z}_p}[[\mathrm{Gal}(k_{\infty}^c/k)]]}(X_{k_{\infty}^c})$
has a square-free generator, we have
\[
(T F_{i}(0,T)) \subset
\mathrm{char}_{{\mathbb{Z}_p}[[\mathrm{Gal}(k_{\infty}^c/k)]]}(X_{k_{\infty}^c}) = (Tf(0,T))
\]
by Proposition \ref{char cyc}. 
%This implies that 
Hence we have
$\displaystyle{ \mathrm{deg}(F_{i}(0,T)) \geq  \mathrm{deg}(f(0,T))}$.
This implies that $f_i(S,T)$ is a unit in $\Lambda$
because we have $f(0,T)=F_i(0,T)f_i(0,T)$.
This is a contradiction.
Thus we get the conclusion.
\end{proof}
%%%%%%%%%%%%%%%%%%%%%%%%%%%%%%%%%%%%%%%%%%%%%%%%%%%%%%%%%%%%%%%%%%%%%%%%%%%%%%%%%%%%%%%%%
%%%%%%%%%%%%%%%%%%%%%%%%%%%%%%%%%%%%%%%%%%%%%%%%%%%%%%%%%%%%%%%%%%%%%%%%%%%%%%%%%%%%%%%%%
\begin{prop}\label{cycsub}
Assume that weak {\rm{GGC}} does not hold for $p$ and $k$.
Assume also that
the characteristic ideal 
%$\mathrm{char}_{{\mathbb{Z}_p}[[\mathrm{Gal}(k_{\infty}^c/k)]]}(X_{k_{\infty}^c})$
of $X_{k_{\infty}^c}$
has a generator which is square-free
or $X_{\widetilde{k}}$ is cyclic as a $\Lambda$-module.
Then there exists a $\Lambda$-submodule $Y$ of $X_{\widetilde{k}}$ such that $Y$ is isomorphic to $\Lambda/(f(S,T))$.
%where $f(S,T)$ is the same power series defined in Lemma \ref{f(S,T)}.
\end{prop}
%%%%%%%%%%%%%%%%%%%%%%%%%%%%%%%%%%%%%%%%%%%%%%%%%%%%%%%%%%%%%%%%%%%%%%%%%%%%%%%%%%%%%%%%%%%%%%
\begin{proof}
If $X_{\widetilde{k}}$ is cyclic as a $\Lambda$-module,
then $X_{\widetilde{k}}$ is isomorphic to $\Lambda/(f(S,T))$ by Proposition \ref{isom str}.
Hence we can take $Y=X_{\widetilde{k}}$.
Thus we get the conclusion.
We suppose that $X_{\widetilde{k}}$ is not cyclic as a $\Lambda$-module.
Since $\Lambda$ is a noetherian ring, $\displaystyle{\bigcup_{\mathfrak{p}\in \mathrm{Ass}_{\Lambda}(X_{\widetilde{k}})} \mathfrak{p}}$
is the set of zero-divisors of $X_{\widetilde{k}}$ by \cite[Theorem $6.1$]{Mat}.
Let $\mathfrak{p}$ be an associated prime ideal of $X_{\widetilde{k}}$.
Let $\displaystyle{f(S,T)} = \prod_{i=1}^{l} f_{i}(S,T)$
be the same prime factorization as in Lemma \ref{an}.
Since $f_i(S,T)$ is an irreducible element of $\Lambda$, $(f_i(S,T))$ is a prime ideal.
We claim that $(f_{i}(S,T))$ is an associated prime ideal of $X_{\widetilde{k}}$ for each $i$.
%there exists a unique $i$ such that $\mathfrak{p}=(f_{i}(S,T))$.
Indeed, $f_{i}(S,T)$ is a zero-divisor of $X_{\widetilde{k}}$
by Lemma \ref{an}.
%since $X_{\widetilde{k}}$ has no non-trivial pseudo-null $\Lambda$-submodule.
Hence we have $\displaystyle{ (f_{i}(S,T)) \subset \bigcup_{\mathfrak{p}\in \mathrm{Ass}_{\Lambda}(X_{\widetilde{k}})} \mathfrak{p}}$.
By the prime avoidance theorem,
$(f_{i}(S,T)) \subset \mathfrak{p}$ for some associated prime ideal $\mathfrak{p}$.
%\in \mathrm{Ass}_{\Lambda}(X_{\widetilde{k}})$.
%Since $\mathfrak{p}$ is an associated prime,
%Then there exists an element $x$ of $X_{\widetilde{k}}$ satisfying $\mathfrak{p} = \mathrm{Ann}_{\Lambda}(x)$.
%Hence
%$\Lambda x$ is a $\Lambda$-submodule of $X_{\widetilde{k}}$ which is isomorphic to
%the $\Lambda$-module
%$\Lambda/\mathfrak{p}$.
%we have $\Lambda/\mathfrak{p} \cong \Lambda x \subset X_{\widetilde{k}}$.
By Lemma \ref{principal}, $\mathfrak{p}$ is a principal ideal.
Therefore we obtain $\mathfrak{p}=(f_{i}(S,T))$.
%because $(f_i(S,T))$ is also a prime ideal.
%%%%%%%%%%%%%%%%%%%%%%%%%%%%%%
Furthermore, we can prove that
\[
\mathrm{Ass}_{\Lambda}(X_{\widetilde{k}})=
\{
(f_{i}(S,T)) ~|~1\leq i\leq l
\}.
\]
Indeed, let $\mathfrak{p}$ be
an associated prime ideal of $X_{\widetilde{k}}$.
%$ \in \mathrm{Ass}(X_{\widetilde{k}})$.
There exists an element $x$ of $X_{\widetilde{k}}$ such that
$\mathfrak{p} = \mathrm{Ann}_{\Lambda}(x)$.
%We have $\mathfrak{p} = \mathrm{Ann}_{\Lambda}(x)$ for some $x \in X_{\widetilde{k}}$.
By Lemma \ref{principal}, $\mathfrak{p}$ is a principal ideal of $\Lambda$.
We put $\mathfrak{p}=(g)$ for some irreducible element $g$ in $\Lambda$.
If $g$ is coprime to $f(S,T)$, then
%$\langle x \rangle_{\Lambda}$
$\Lambda x$ is a non-trivial pseudo-null $\Lambda$-submodule of $X_{\widetilde{k}}$.
This is a contradiction.
Therefore there exists an integer $i$ such that
$(g)=(f_i(S,T))$.\par
By the structure theorem,
%(\cite[Proposition 5.1.7]{NSW}),
there exists a pseudo-isomorphism
\begin{eqnarray}\label{pshom}
\varphi: \bigoplus_{i=1}^{l'} \Lambda/\mathfrak{q}_i \rightarrow X_{\widetilde{k}},
\end{eqnarray}
where $l'$ is a positive integer and
$\mathfrak{q}_i$'s are prime ideals of height one.
We note that
$\mathrm{char}_{\Lambda}(X_{\widetilde{k}})$ has a square-free generator
%because
%the characteristic ideal
%$\mathrm{char}_{{\mathbb{Z}_p}[[\mathrm{Gal}(k_{\infty}^c/k)]]}(X_{k_{\infty}^c})$
%also has a square-free generator.
by the same reason as in the proof of Lemma \ref{an}.
Hence we obtain $\mathfrak{q}_i \neq \mathfrak{q}_j$ for $i \neq j$.
%%%%%%%%%%%%%%%%%%%%%%%%%%%%%%%%%%%%%%%%%%%%%%%%%%%%%%%%%%%%%%%%%%%%%%%%%%%%%%%%%%%%%%%
Since $\displaystyle{\bigoplus_{i=1}^{l'} \Lambda/\mathfrak{q}_i}$ has no non-trivial pseudo-null
$\Lambda$-submodule, this homomorphism is injective. 
In the following, we will prove that $\mathfrak{q}_i$ is an associated prime ideal for each $i$ and that
$\mathrm{Ass}_{\Lambda}(X_{\widetilde{k}})$ coincides with
the set $\{\mathfrak{q}_i ~|~1\leq i \leq l' \}$.
Indeed, for each $j$ with $1\leq j\leq l'$,
we have a homomorphism
\[
\pi_j:\Lambda/\mathfrak{q}_j \rightarrow \bigoplus_{i=1}^{l'} \Lambda/\mathfrak{q}_i \xrightarrow{\varphi } X_{\widetilde{k}},
\]
where $\displaystyle{\Lambda/\mathfrak{q}_j \rightarrow \bigoplus_{i=1}^{l'} \Lambda/\mathfrak{q}_i}$ is the
natural injective homomorphism
such that, for an element $a$ of $\Lambda/\mathfrak{q}_j$, we define $a \mapsto (a_i)_i$ satisfying that $a_i =0$ for $i \neq j$ and $a_j=a$.
%Since $\mathfrak{q}_j$ is a prime ideal of height one, 
%We note that
Hence
$\pi_j$ is an injective homomorphism.
Thus $\mathfrak{q}_j$ is an associated prime ideal.
%we obtain $\mathfrak{q}_j \in \mathrm{Ass}(X_{\widetilde{k}})$.\par
%Let $y_i \in X_{\widetilde{k}}$ be a generator of the image of $\pi_i$.
%$\mathrm{Image}()$]
%Hence we have
This implies that
$\{\mathfrak{q}_i| 1\leq i \leq l' \}$ is contained in $\mathrm{Ass}_{\Lambda}(X_{\widetilde{k}})$.
Next we prove that $\mathrm{Ass}_{\Lambda}(X_{\widetilde{k}})$ is contained in $\{\mathfrak{q}_i| 1\leq i \leq l' \}$.
%for each an associated prime $\mathfrak{p}$
%there exists $i$ such that a prime ideal $\mathfrak{p}=\mathfrak{q}_i$.
%$\mathfrak{q}_i \neq \mathfrak{q}_j$ for $i \neq j$.
We assume that $\mathrm{Ass}_{\Lambda}(X_{\widetilde{k}}) \not\subset \{\mathfrak{q}_i| 1\leq i \leq l' \}$.
Then there exists an associated prime ideal $\mathfrak{p}$
%\in \mathrm{Ass}(X_{\widetilde{k}})$
such that
$\mathfrak{p} \not\in \{\mathfrak{q}_i| 1\leq i \leq l' \}$.
We note that $\mathfrak{p}$ is a prime ideal of height one by Lemma \ref{principal}.
Using localization with respect to
%(\ref{pshom}) to
$\mathfrak{p}$, we obtain
\[
(X_{\widetilde{k}})_{\mathfrak{p}} \cong \left(  \bigoplus_{i=1}^{l'} \Lambda/\mathfrak{q}_i \right)_{\mathfrak{p}} = 0
\]
from (\ref{pshom}).
Hence $\mathrm{Ass}_{\Lambda}(X_{\widetilde{k}})$
is not contained in
%$\mathrm{supp}(X_{\widetilde{k}})$, where $\mathrm{supp}(X_{\widetilde{k}})$ is
the support of $X_{\widetilde{k}}$.
This is a contradiction. Thus we get
%$\ell=\ell'$ and
\[
\mathrm{Ass}_{\Lambda}(X_{\widetilde{k}}) =  \{\mathfrak{q}_i~|~ 1\leq i \leq l' \} = \{ (f_i(S,T))~|~ 1\leq i \leq l \}.
\]
This implies that $l = l'$.
Therefore we have a natural injective homomorphism
\[
\Lambda/(f(S,T)) \rightarrow \bigoplus_{i=1}^{l} \Lambda/(f_{i}(S,T)) = \bigoplus_{i=1}^{l} \Lambda/\mathfrak{q}_i
\rightarrow X_{\widetilde{k}}
\]
since we have a prime factorization $\displaystyle{f(S,T)} = \prod_{i=1}^{l} f_{i}(S,T)$.
Thus we obtain the conclusion.
\end{proof}
%%%%%%%%%%%%%%%%%%%%%%%%%%%%%%%%%%%%%%%%%%%%%%%%%%%%%%%%%%%%%%%%%%%%%%%%%%%%%%%%%%%%%%%%%%%%%%%%%%%%%%%%%%%%
%%%%%%%%%%%%%%%%%%%%%%%%%%%%%%%%%%%%%%%%%%%%%%%%%%%%%%%%%%%%%%%%%%%%%%%%%%%%%%%%%%%%%%%%%%%%%%%%%%%%%%%%%%%%
By the proposition above, we obtain the following
\begin{cor}\label{psisom}
Let $M$ be a finitely generated torsion $\Lambda$-module.
Assume that $M/SM$ is a free $\mathbb{Z}_p$-module of finite rank
and that
$\mathrm{char}_{\mathbb{Z}_p[[T]]}(M/SM)$
has a generator which is square-free.
Assume also that
$M$ has no non-trivial pseudo-null $\Lambda$-submodule.
Then we have a pseudo-isomorphism
\[
M \rightarrow \Lambda / \mathrm{char}_{\Lambda}(M).
\]
\end{cor}
%%%%%%%%%%%%%%%%%%%%%%%%%%%%%%%%%%%%%%%%%%%%%%%%%%%%%%%%%%%%%%%%%%%%%%%%%%%%%%%%%%%%%%%%%%%%%%%%%%%%%%%%%%%%%%
%%%%%%%%%%%%%%%%%%%%%%%%%%%%%%%%%%%%%%%%%%%%%%%%%%%%%%%%%%%%%%%%%%%%%%%%%%%%%%%%%%%%%%%%%%%%%%%%%%%%%%%%%%%%
Let $k_{\infty}/k$ be a $\mathbb{Z}_p$-extension. Then there exists a pair $(\alpha, \beta) \in \mathbb{Z}_p^{ \oplus 2} - p\mathbb{Z}_p^{\oplus 2}$
such that $k_{\infty}=\widetilde{k}^{\overline{\langle \sigma^{\alpha} \tau^{\beta} \rangle}}$,
where $\sigma$ and $\tau$ are the fixed topological generators defined in Section \ref{Preli}.
It is easy to check the following
%%%%%%%%%%%%%%%%%%%%%%%%%%%%%%%%%%%%%%%%%%%%%%%%%%%%%%%%%%%%%%%%%%%%%%%%%%%%%%%%%%%%%
\begin{lem}\label{Zp}
Let $k_{\infty}/k$ be a $\mathbb{Z}_p$-extension with $k_{\infty} \cap k_{\infty}^a = k_{m}^a$,
where $ k_{m}^a$ is the $m$-th layer of $ k_{\infty}^a$
and $m$ is a non-negative integer.
Then there exists a unit $u \in \mathbb{Z}_p^{\times}$ such that
$k_{\infty}=\widetilde{k}^{\overline{\langle \sigma^{up^{m}} \tau \rangle}}$.
\end{lem}
%%%%%%%%%%%%%%%%%%%%%%%%%%%%%%%%%%%%%%%%%%%%%%%%%%%%%%%%%%%%%%%%%%%%%%%%%%%%%%%%%%%%
We put $\Gamma=\mathrm{Gal}(\widetilde{k}/N_{\infty})$.
Recall that $p^s = [L_{k}\cap \widetilde{k}:k]$.
Hence we have $N_{\infty} \cap k_{\infty}^a = L_k \cap \widetilde{k} = k_s^a$.
Then there exists a unit $u$ such that
\[
\mathrm{Gal}(\widetilde{k}/N_{\infty}) = \overline{ \langle \sigma^{up^{s}} \tau \rangle}
\]
by the lemma above.
We put $T_{s}=(1+S)^{up^s}(1+T)-1.$\par
%%%%%%%%%%%%%%%%%%%%%%%%%%%%%%%%%%%%%%%%%%%%%%%%%%%%%%%%%%%%%%%%%%%%%%%%%%%%%%%%%%%%%%%%%
The following proposition plays an important role
to obtain a lower bound of the $\mathbb{Z}_p$-rank of
$\mathrm{Gal} (M_{\mathfrak{p^{\ast}}}(N_{\infty})/ \widetilde{k}) $.
%%%%%%%%%%%%%%%%%%%%%%%%%%%%%%%%%%%%%%%%%%%%%%%%%%%%%%%%%%%%%%%%%%%%%%%%%%%%%%%%%%%%%%%%%
\begin{prop}\label{rank}
Under the same assumption as in Proposition \ref{cycsub},
%Assume that weak Greenberg's generalized conjecture does not hold for $p$ and $k$.
we have $\mathrm{rank}_{\mathbb{Z}_p}((X_{\widetilde{k}})_{\Gamma}) \geq \mathrm{rank}_{\mathbb{Z}_p}(Y_{\Gamma})$,
where $Y$ is the same cyclic $\Lambda$-submodule of $X_{\widetilde{k}}$ as in Proposition \ref{cycsub}.
%There exist a $\Lambda$-submodule $Y$ of $X_{\widetilde{k}}$ such that $Y \cong \Lambda/f(S,T)$.
\end{prop}
%%%%%%%%%%%%%%%%%%%%%%%%%%%%%%%%%%%%%%%%%%%%%%%%%%%%
\begin{proof}
If $X_{\widetilde{k}}$ is cyclic as a $\Lambda$-module, then we have $Y=X_{\widetilde{k}}$.
Thus we get the conclusion.
We suppose that $X_{\widetilde{k}}$ is not cyclic as a $\Lambda$-module.
By Proposition \ref{cycsub}, we have an exact sequence
\[
0 \rightarrow Y \rightarrow X_{\widetilde{k}} \rightarrow \mathrm{Coker} \rightarrow 0,
\]
where we put $\mathrm{Coker} = X_{\widetilde{k}}/Y$.
%$=\mathrm{Coker}( Y \rightarrow X_{\widetilde{k}}).$
Using the snake lemma, we get
\begin{eqnarray*}
0 \rightarrow   Y^{\Gamma} \rightarrow (X_{\widetilde{k}})^{\Gamma}
\rightarrow   \mathrm{Coker}^{\Gamma} 
  \rightarrow   Y_{\Gamma} \rightarrow (X_{\widetilde{k}})_{\Gamma} \rightarrow 
\mathrm{Coker}_{\Gamma} \rightarrow 0.
\end{eqnarray*}
Since $X_{\widetilde{k}}$ has no non-trivial pseudo-null $\Lambda$-submodule,
we have $(X_{\widetilde{k}})^{\Gamma}=0$.
Indeed, we have $T_s (X_{\widetilde{k}})^{\Gamma}=f(S,T)(X_{\widetilde{k}})^{\Gamma}=0$.
Since the characteristic ideal of $X_{k_{\infty}^c}$
%$\mathrm{char}_{\mathbb{Z}_p[[T]]}(X_{k_{\infty}^c})$
is generated by a power series which is square-free,
%has no double rootsin an algebraic closure of $\mathbb{Q}_p$,
$T_s$ is coprime to $f(S,T)$.
This implies that $(X_{\widetilde{k}})^{\Gamma}$ is pseudo-null.
Hence we have $(X_{\widetilde{k}})^{\Gamma}=0$.
Then we obtain an exact sequence
\begin{eqnarray}\label{museq}
0 \rightarrow   \mathrm{Coker}^{\Gamma} 
  \rightarrow   Y_{\Gamma} \rightarrow (X_{\widetilde{k}})_{\Gamma} \rightarrow 
\mathrm{Coker}_{\Gamma} \rightarrow 0.
\end{eqnarray}
We identify the following rings below,
using isomorphisms
\begin{eqnarray}\label{is}
\Lambda/T_s \Lambda \cong \mathbb{Z}_p[[S]]
\cong \mathbb{Z}_p[[\mathrm{Gal}(N_{\infty}/k)]],
\end{eqnarray}
where we define
\begin{eqnarray*}
G(S,T) \mapsto G(S,(1+S)^{-up^s}-1) \mapsto G(\sigma \mathrm{Gal}(\widetilde{k}/N_{\infty})-1,\sigma^{-up^s} \mathrm{Gal}(\widetilde{k}/N_{\infty})-1).
\end{eqnarray*}
%we identify these rings above.
Here we note that $\mathbb{Z}_p[[S,T_s]]=\mathbb{Z}_p[[S,T]]$
because we have $T=(1+T_s)(1+S)^{-up^s}-1$.
We will prove that
%the $\mathbb{Z}_p$-rank of
$Y_{\Gamma}$ is a finitely generated $\mathbb{Z}_p$-module.
%, in other words,
%the $\mu$-invariant $\mu( Y_{\Gamma}) =0$.
Indeed,
$\mathrm{Coker}_{\Gamma}$ is a finitely generated $\mathbb{Z}_p$-module
since $(X_{\widetilde{k}})_{\Gamma}$, which is isomorphic to $\mathrm{Gal}(M_{\mathfrak{p}^{\ast}}(N_{\infty})/\widetilde{k})$,
is a finitely generated $\mathbb{Z}_p$-module.
Furthermore, we can show that
$\mathrm{Coker}^{\Gamma}$ is a pseudo-null $\Lambda$-module
by the same method as above.
From (\ref{is}), we have
\begin{eqnarray*}
%\[
&~&f(S,T) \equiv f(S,(1+S)^{-up^s}-1)~~\mathrm{~mod~}T_s, \\
%\mathrm{~and~}
&~&f(S,(1+S)^{-up^s}-1) \mathrm{Coker}^{\Gamma}=0.
%\]
\end{eqnarray*}
We note that $f(S,(1+S)^{-up^s}-1) \neq 0$ since $f(0,0) \neq 0$.
Hence $\mathrm{Coker}^{\Gamma}$ is a torsion $\mathbb{Z}_p[[S]]$-module.
By \cite[I,1.3, Lemma 4]{PR}, we have
\[
\mathrm{char}_{\mathbb{Z}_p[[S]]}(\mathrm{Coker}_{\Gamma}) =
\pi(\mathrm{char}_{\Lambda}(\mathrm{Coker})) \mathrm{char}_{\mathbb{Z}_p[[S]]}(\mathrm{Coker}^{\Gamma}),
\]
where $\pi$ is the natural projection from $\Lambda$ to $\Lambda/T_s \Lambda$.
This implies that
%$\mu( Y_{\Gamma}) =0$.
$Y_{\Gamma}$ is a finitely generated $\mathbb{Z}_p$-module.
%since $\mathrm{Coker}_{\Gamma}$ is a finitely generated $\mathbb{Z}_p$-module.
Hence all terms in (\ref{museq}) have finite $\mathbb{Z}_p$-rank.
Therefore we obtain
\begin{eqnarray*}
\mathrm{rank}_{\mathbb{Z}_p}( (X_{\widetilde{k}})_{\Gamma})
&=& \mathrm{rank}_{\mathbb{Z}_p}( Y_{\Gamma})
+ \mathrm{rank}_{\mathbb{Z}_p}( \mathrm{Coker}_{\Gamma}) -
\mathrm{rank}_{\mathbb{Z}_p}(  \mathrm{Coker}^{\Gamma})\\ \nonumber
%&=& \mathrm{rank}_{\mathbb{Z}_p}( Y_{\Gamma}) +\mathrm{rank}_{\mathbb{Z}_p[[\Gamma]]}( \mathrm{Coker})\\
&\geq & \mathrm{rank}_{\mathbb{Z}_p}( Y_{\Gamma}).
\end{eqnarray*}
Thus we get the conclusion.
\end{proof}~\\
%%%%%%%%%%%%%%%%%%%%%%%%%%%%%%%%%%%%%%%%%%%%%%%%%%%%%%%%%%%%%%%%%%%%%%%%%%%%%%%%%%%%%%%%%%%%%%%%%%%%%%%%%%%%%%%%%%%%%%%%%%%%%%%5
\subsection{Normal case}~
%%%%%%%%%%%%%%%%%%%%%%%%%%%%%%%%%%%%%%%%%%%%%%%%%%%%%%%%%%%%%%%%%%%%%%%%%%%%%%%%%%%%%%%%%%%%%%%%%%%%%%%%%%%%%%%%%%%%%%%%%%%%%%%%%%%55
%%%%%%%%%%%%%%%%%%%%%%%%%%%%%%%%%%%%%%%%%%%%%%%%%%%%%%%%%%%%o%%%%%%%%%%%%%%%%%%%%%%%%%%%%%%%%%%%%%%%%%%%%%%%%%%%%%%%%%%%%%%%%%%%%%%%55
In this subsection, we prove the inequality (B)
in the case where $k$ is not $p$-split $p$-rational
assuming that weak GGC does not hold for $p$ and $k$.
%$\mathfrak{D_{p}}$ is a normal subgroup of $\mathrm{Gal}(\widetilde{k}/\mathbb{Q})$.
%under the assumption 
%We split into the two cases:
%\begin{eqnarray*}
%\begin{cases}
%[\mathrm{Gal}(\widetilde{k}/k): \mathfrak{D_{p}}] &<  \# \left( \mathbb{Z}_p/g_0(0)\mathbb{Z}_p \right),\\
%[\mathrm{Gal}(\widetilde{k}/k): \mathfrak{D_{p}}] &=  \# \left( \mathbb{Z}_p/g_0(0)\mathbb{Z}_p \right).
%\end{cases}
%[\mathrm{Gal}(\widetilde{k}/k): \mathfrak{D_{p}}] \leq  [L_{k}\cap \widetilde{k}:k] \leq \# \left( \mathbb{Z}_p/g_0(0)\mathbb{Z}_p \right)
%\end{eqnarray*}
We first give a necessary and sufficient condition for $k$ to be $p$-split $p$-rational.
%%%%%%%%%%%%%%%%%%%%%%%%%%%%%%%%%%%%%%%%%%%%%%%%%%%%%%%%%%%%%%%%%%%%%%%%%%%%%%%%%%%%%%%%%%%%%%%%
%%%%%%%%%%%%%%%%%%%%%%%%%%%%%%%%%%%%%%%%%%%%%%%%%%%%%%%%%%%%%%%%%%%%%%%%%%%%%%%%%%%%%%%%%%%%%%%
\begin{prop}\label{eqrational}
The following two conditions are equivalent:\\
$\mathrm{(i)}$ The decomposition group $\mathfrak{D_{p}}$ is a normal subgroup of $\mathrm{Gal}(\widetilde{k}/\mathbb{Q})$,
$\displaystyle{[\mathrm{Gal}(\widetilde{k}/k): \mathfrak{D_{p}}]}
= \# \left( \mathbb{Z}_p/g_0(0)\mathbb{Z}_p \right)$,
and $\lambda(k_{\infty}^c/k)\geq 2$, where $g_0(S)$ is the same power series defined in Lemma \ref{f(S,T)}.\\
$\mathrm{(ii)}$ The imaginary quadratic field $k$ is $p$-split $p$-rational.
%we have $L=\widetilde{k}$ and $L_{k} \subset \widetilde{k}$.
\end{prop}
%%%%%%%%%%%%%%%%%%%%%%%%%%%%%%%%%%%%%%%%%%%%%%%%%%%%%%%%%%%%%%%%%
\begin{proof}
We suppose that (i) holds.
Let $L$ be the fixed field of $L_{k_{\infty}^c}$ by $TX_{k_{\infty}^c}$.
%maximal unramified abelian pro-$p$ extension field of $k$ in $L_{k_{\infty}^c}$.
We note that
$\displaystyle{\mathrm{Gal}(L/k_{\infty}^c)=\left( X_{k_{\infty}^c} \right)_{\mathrm{Gal}(k_{\infty}^c/k)}}$.
Using the exact sequence (\ref{ex index}) in the proof of Proposition \ref{index},
we have
\begin{eqnarray*}
\#\mathrm{Image}( \mathrm{Gal}(\widetilde{k}/k_{\infty}^c)\rightarrow \left( X_{\widetilde{k}}\right)_{\mathrm{Gal}(\widetilde{k}/k)})
=\#  (X_{\widetilde{k}})_{\mathrm{Gal}(\widetilde{k}/k)}
%&=&
%[\mathrm{Gal}(\widetilde{k}/k): \mathfrak{D_{p}}]\\
 %                     &\leq & \#  (X_{\widetilde{k}})_{\mathrm{Gal}(\widetilde{k}/k_{\infty}^c)}\\
  %                    &=&  \# \left( \mathbb{Z}_p/g_0(0) \mathbb{Z}_p \right).
\end{eqnarray*}
from $\displaystyle{[\mathrm{Gal}(\widetilde{k}/k): \mathfrak{D_{p}}]}
= \# \left( \mathbb{Z}_p/g_0(0)\mathbb{Z}_p \right)$
and from Lemma \ref{Momo}.
%$\displaystyle{[\mathrm{Gal}(\widetilde{k}/k): \mathfrak{D_{p}}] =  \# \left( \mathbb{Z}_p/g_0(0)\mathbb{Z}_p \right)}$.
%By Proposition \ref{index}, we have
%\[
%[\mathrm{Gal}(\widetilde{k}/k): \mathfrak{D_{p}}] =  [L_{k}\cap \widetilde{k}:k] = \# \left( \mathbb{Z}_p/g_0(0)\mathbb{Z}_p \right).
%\]
Hence the homomorphism
$\mathrm{Gal}(\widetilde{k}/k_{\infty}^c)\rightarrow \left( X_{\widetilde{k}}\right)_{\mathrm{Gal}(\widetilde{k}/k)}$
is surjective.
This implies that $L=\widetilde{k}$ from (\ref{ex index}).
Since $L_k k_{\infty}^c/k$ is an abelian extension, $L_k k_{\infty}^c$ is contained in $L$.
Therefore, we obtain $L_k \subset \widetilde{k}$.
Then it is easy to check that $M_p(k)=\widetilde{k}$,
where $M_p(k)$ is the maximal pro-$p$ abelian extension field of $k$ unramified outside all prime ideals lying above $p$.
By Lemma \ref{cond normal}, we have $\widetilde{k}^{\mathfrak{D_p}} \subset L_k$
since $\mathfrak{D_p}$ is a normal subgroup.
Furthermore, we have $\widetilde{k}^{\mathfrak{D_p}} \neq k$
%because
%$\mathrm{char}_{\mathbb{Z}_p[[T]]}(X_{k_{\infty}^c}) \not\subset (T^2)$ holds.
from $\lambda(k_{\infty}^c/k)\geq 2$.
Hence $k$ is $p$-split $p$-rational.
\par
%%%%%%%%%%%%%%%%%%%%%%%%%%%%%%%%%%%%%%%%%%%
Conversely we suppose that $k$ is $p$-split $p$-rational.
By Definition \ref{rational}, we have
\[
L_k \subset \widetilde{k},\quad \widetilde{k}^{\mathfrak{D_{p}}} \neq k ,\quad \widetilde{k}^{\mathfrak{D_{p}}} \subset L_k.
\]
This implies that
$M_p(k) = \widetilde{k}$
and that
$\mathfrak{p}$ splits in $\widetilde{k}$.
%%%%%%%%%%%%%%%%%%%%%%%%%%%%%%%%%%%%%%%%%%%%%%%%%%
Hence the homomorphism
$\mathrm{Gal}(\widetilde{k}/k_{\infty}^c)\rightarrow \left( X_{\widetilde{k}}\right)_{\mathrm{Gal}(\widetilde{k}/k)}$ above
is surjective.
We note that $\lambda(k_{\infty}^c/k)\geq 2$.
By the exact sequence (\ref{ex index}) in the proof of Proposition \ref{index}, we obtain
\[
%\#\mathrm{Image}( \mathrm{Gal}(\widetilde{k}/k_{\infty}^c)\rightarrow \left( X_{\widetilde{k}}\right)_{\mathrm{Gal}(\widetilde{k}/k)})
[\mathrm{Gal}(\widetilde{k}/k): \mathfrak{D_{p}}]
=\#  (X_{\widetilde{k}})_{\mathrm{Gal}(\widetilde{k}/k)}
=\# \left( \mathbb{Z}_p/g_0(0) \mathbb{Z}_p \right).
\]
Since $\widetilde{k}^{\mathfrak{D_{p}}}$ is contained in $L_k$, we have
$\displaystyle{[\mathrm{Gal}(\widetilde{k}/k): \mathfrak{D_{p}}] \leq [L_k :k]}$.
By Lemma \ref{cond normal}, $\mathfrak{D_p}$ is a normal subgroup of $\mathrm{Gal}(\widetilde{k}/\mathbb{Q})$.
%We have $\lambda(k_{\infty}^c/k) \geq 2$ because $\widetilde{k}^{\mathfrak{D_{p}}} \neq k$.
Thus we get the conclusion.
\end{proof}
%%%%%%%%%%%%%%%%%%%%%%%%%%%%%%%%%%%%%%%%%%%%%%%%%%%%%%%%%
%%%%%%%%%%%%%%%%%%%%%%%%%%%%%%%%%%%%%%%%%%%%%%%%%%%%%%%%%%%%%%%%%%%%%%%%%%%%
%Next, we define notation.
For a power series $V(S) = \displaystyle{\sum_{i=0}^{\infty} b_i S^i }\in \mathbb{Z}_p[[S]]$,
we put
\begin{eqnarray*}
\mu(V(S)) &=& \mathrm{sup} \{ ~i~|~V(S) \equiv 0 \mathrm{~mod~}p^i~ ~\},\\
\lambda(V(S)) &=& \mathrm{inf} \{ ~i~|~b_i \not\equiv 0  \mathrm{~mod~}p~\}.
\end{eqnarray*}
Then we call $\mu(V(S)), \lambda(V(S))$ the $\mu$- and $\lambda$-invariant of $V(S)$, respectively.
%We note that $p$ divides $V(S)$ if and only if $\lambda(V(S))=\infty$.
\par
%%%%%%%%%%%%%%%%%%%%%%%%%%%%%%%%%%%%%%%%%%%%%%%%%%%%%%%%%%%%%%%%%%%%%%%%%%%%%%%%%%%%%%%%
%%%%%%%%%%%%%%%%%%%%%%%%%%%%%%%%%%%%%%%%%%%%%%%%%%%%%%%%%%%%%%%%%%%%%%%%%%%%%%%%%%%%%%%%%%%%%%%%%%%%%%%%%%%%%%%%%%%%%%%%%%%%%%%%%%%55
%First we suppose that
%$[\mathrm{Gal}(\widetilde{k}/k): \mathfrak{D_{p}}] <  \# \left( \mathbb{Z}_p/g_0(0)\mathbb{Z}_p \right)$.
%We prove the inequality (B) for this case.
Next, we get a lower bound of the $\lambda$-invariant of $g_0(S)$ in the case where
the $\mu$-invariant of $g_0(S)$ is zero. 
Recall that $p^{n_0} = \displaystyle{[\mathrm{Gal}(\widetilde{k}/k): \mathfrak{D_{p^{\ast}}}]}$.
\begin{prop}\label{normal g}
Assume that
$\mathfrak{D_{p}}$ is a normal subgroup of $\mathrm{Gal}(\widetilde{k}/\mathbb{Q})$.
Assume also that $k$ is not $p$-split $p$-rational
and that $\lambda(k_{\infty}^c/k)\geq 2$.
%$[\mathrm{Gal}(\widetilde{k}/k): \mathfrak{D_{p}}] < \# \left( \mathbb{Z}_p/g_0(0)\mathbb{Z}_p \right)$.
If $\mu(g_0(S))$ is zero, then we have $\lambda(g_0(S)) \geq  p^{n_0}.$ 
\end{prop}
%%%%%%%%%%%%%%%%%%%%%%%%%%%%%%%%%%%%%%%%%%%%%%%%%%%%%%%%%%%%%%%%%%%%%%%%%%%%%%%%%%%%%%%%%%%%%%%%
%%%%%%%%%%%%%%%%%%%%%%%%%%%%%%%%%%%%%%%%%%%%%%%%%%%%%%%%%%%%%%%%%%%%%%%%%%%%%%%%%%%%%%%%%%%%%%%%
\begin{proof}
Since $k_{n_0}^a/k$ is an abelian extension, Leopoldt's conjecture holds for $p$ and $k_{n_0}^a$ (\cite{Bu}).
Hence we have $\mathrm{Gal}(\widetilde{k_{n_0}^a}/k_{n_0}^a) \cong \mathbb{Z}_p^{\oplus{p^{n_0}+1}}$,
where $\widetilde{k_{n_0}^a}$ is the composite of all $\mathbb{Z}_p$-extension fields over $k_{n_0}^a$.
Since $p$ splits in $k_{n_0}^a$, $\widetilde{k_{n_0}^a}/k_{\infty}^a$ is an unramified abelian extension.
Thus we have a surjective homomorphism
\[
X_{k_{\infty}^a} \rightarrow \mathrm{Gal}(\widetilde{k_{n_0}^a}/k_{\infty}^a).
\]
Furthermore, we see that
\[
 \mathrm{Gal}(\widetilde{k_{n_0}^a}/k_{\infty}^a) \cong  \mathbb{Z}_p[[S]]/((1+S)^{p^{n_0}}-1).
\]
Hence we have $\mathrm{char}_{\mathbb{Z}_p[[S]]}(X_{k_{\infty}^a}) \subset ((1+S)^{p^{n_0}}-1)$.
On the other hand, there exists a surjective homomorphism $(\Lambda/f(S,T)\Lambda)^{\oplus r} \rightarrow X_{\widetilde{k}}$,
where $r$ is a positive integer and
$f(S,T)$ is the same power series defined in Lemma \ref{f(S,T)}.
Then we have a surjective homomorphism
$(\mathbb{Z}_p[[S]]/g_0(S) \mathbb{Z}_p[[S]])^{\oplus r} \rightarrow X_{\widetilde{k}}/TX_{\widetilde{k}}$.
This implies that
\[
(Sg_0(S)^r) \subset \mathrm{char}_{ \mathbb{Z}_p[[S]]}(X_{k_{\infty}^a}) \subset ((1+S)^{p^{n_0}}-1).
\]
Since $(1+S)^{p^{n_0}}-1$ is square-free, we obtain $(Sg_0(S)) \subset ((1+S)^{p^{n_0}}-1).$
Then there exists a power series $\widetilde{g_0}(S) \in \mathbb{Z}_p[[S]]$ such that
$Sg_0(S) = \{(1+S)^{p^{n_0}}-1\}\widetilde{g_0}(S)$.
We have
$p^{n_0}=[\mathrm{Gal}(\widetilde{k}/k): \mathfrak{D_{p}}] < \# \left( \mathbb{Z}_p/g_0(0)\mathbb{Z}_p \right)$
by Propositions \ref{index} and \ref{eqrational}.
Hence, we get
$\mathrm{ord}_p(\widetilde{g_0}(0)) > 0$,
where $\mathrm{ord}_p$ is the normalized additive valuation on $\mathbb{Q}_p$
of $p$-adic numbers such that $\mathrm{ord}_p(p)=1$.
%by $\mu(g_0(S))=0$.
Therefore, we obtain $\lambda(g_0(S)) \geq p^{n_0}$ because $\mu(g_0(S))$ is zero.
\end{proof}
%%%%%%%%%%%%%%%%%%%%%%%%%%%%%%%%%%%%%%%%%%%%%%%%%%%%%%%%%%%%%%%%%%%%%
By the proposition above, we can prove the inequality (B).
%%%%%%%%%%%%%%%%%%%%%%%%%%%%%%%%%%%%%%%%%%%%%%%%%%%%%%%%%%%%%%%%%%%%%
\begin{prop}\label{normal rankXM}
%Assume the same conditions as Proposition \ref{normal g}.
Assume the same condition as in Proposition \ref{cycsub}.
Assume also that $\mathfrak{D_{p}}$ is a normal subgroup of $\mathrm{Gal}(\widetilde{k}/\mathbb{Q})$
and that $k$ is not $p$-split $p$-rational.
Then
%we have
%$\mathrm{rank}_{\mathbb{Z}_p}(\mathrm{Gal}(M_{\frak{p}^{\ast}}(N_{\infty})/k)) > p^{n_0}$.
%In particular,
the inequality $\mathrm{(B)}$ holds.
\end{prop}
%%%%%%%%%%%%%%%%%%%%%%%%%%%%%%%%%%%%%%%%%%%%%%%%%%%%%%%%%%%%%%%%%%%%%%%%%%%%%%%
\begin{proof}
We have $\lambda(k_{\infty}^c /k)\geq 2$ since weak GGC does not hold.
By Proposition \ref{cycsub}, there exists a $\Lambda$-submodule $Y$ of $X_{\widetilde{k}}$
such that $Y$ is isomorphic to $\Lambda/(f(S,T))$.
We note that $Y_{\Gamma}$ is a finitely generated $\mathbb{Z}_p$-module
by the proof of Proposition \ref{rank}.
We claim that $\mathrm{rank}_{\mathbb{Z}_p}(Y_{\Gamma}) \geq p^{n_0}$.
By Lemma \ref{Zp}, there exists a unit $u$ such that
\[
\Gamma=\mathrm{Gal}(\widetilde{k}/N_{\infty}) = \overline{ \langle \sigma^{up^{s}} \tau \rangle.}
\]
%We put $T_{s}=(1+S)^{up^s}(1+T)-1.$
Then we have
\begin{eqnarray*}
Y_{\Gamma} &\cong & \Lambda/(f(S,T), (1+S)^{up^s}(1+T)-1)\\
&\cong & \mathbb{Z}_p[[S]] /(f(S, (1+S)^{-up^s}-1)).
\end{eqnarray*}
By the definition of $f(S,T)$,
%and Proposition \ref{g0}, 
we obtain
\begin{eqnarray*}
&~&f(S,  (1+S)^{-up^s}-1)\\
&=&
\{(1+S)^{-up^s}-1\}^{\lambda^{\ast}} + \sum_{i=0}^{\lambda^{\ast}-1} g_{i}(S) \{ (1+S)^{-up^s}-1\}^{i}\\
&\equiv &
%\{(1+S)^{-up^s}-1\}^{\lambda^{\ast}}+
\left\{ \sum_{j=1}^{\infty} \dbinom{-u}{j}S^{jp^s} \right\}^{\lambda^{\ast}}+
\sum_{i=0}^{\lambda^{\ast}-1} g_{i}(S) \left\{ \sum_{j=1}^{\infty} \dbinom{-u}{j}S^{jp^s} \right\}^{i}
\quad \mathrm{~mod~}p\\
&\equiv &
\begin{cases}
g_0(S)  \quad \mathrm{~mod~}(p, S^{p^s}) \quad &\mathrm{~if~} \mu(g_0(S))=0,\\
0      \hspace{6.3mm} \quad \mathrm{~mod~}(p, S^{p^s}) \quad &\mathrm{~if~} \mu(g_0(S))>0.
\end{cases}
%&\cong & 
\end{eqnarray*}
By Proposition \ref{normal g}, we obtain $\lambda(f(S,  (1+S)^{-up^s}-1)) \geq p^{n_0}$.
This implies that
$\mathrm{rank}_{\mathbb{Z}_p}(Y_{\Gamma}) \geq p^{n_0}$.
Thus we obtain 
$\mathrm{rank}_{\mathbb{Z}_p}((X_{\widetilde{k}})_{\Gamma}) \geq p^{n_0}$
by Proposition \ref{rank}.
Using Proposition \ref{index}, we get the conclusion.
\end{proof}~\\
%%%%%%%%%%%%%%%%%%%%%%%%%%%%%%%%%%%%%%%%%%%%%%%%%%%%%%%%%%%%%%%%%%%%%%%%%%
%%%%%%%%%%%%%%%%%%%%%%%%%%%%%%%%%%%%%%%%%%%%%%%%%%%%%%%%%%%%%%%%%%%%%%%%%%%%%%%%%%%%%%%%%%%%%%%%%%%%%%%%%%%%%%%%%%%%%
\subsection{Non-normal case}~
%%%%%%%%%%%%%%%%%%%%%%%%%%%%%%%%%%%%%%%%%%%%%%%%%%%%%%%%%%%%%%%%%%%%%%%%%%%%%%%%%%%%%%%%%
In this subsection, we prove  the inequality (B) in
the case where $\mathfrak{D_{p}}$ is not a normal subgroup of $\mathrm{Gal}(\widetilde{k}/\mathbb{Q})$
under the assumption that weak GGC does not hold for $p$ and $k$.
We first show the following two propositions.
%%%%%%%%%%%%%%%%%%%%%%%%%%%%%%%%%%%%%%%%%%%%%%%%%%%%%%%%%%%%%%%%%%%%%%%%%%%%%%%%%%%%%%%%%
%%%%%%%%%%%%%%%%%%%%%%%%%%%%%%%%%%%%%%%%%%%%%%%%%%%%%%%%%%%%%%%%%%%%%%%%%%%%%%%%%%%%%%%%%
\begin{prop}\label{rankXM}
Assume the same condition as in Proposition \ref{cycsub}.
Assume also that $\mathfrak{D_{p}}$ is not a normal subgroup of $\mathrm{Gal}(\widetilde{k}/\mathbb{Q})$.
%and that $k$ is not a $p$-split $p$-rational.
%Assume that weak {\rm{GGC}} does not hold for $p$ and $k$.
%Assume also that $\mathfrak{D_{p}}$ is not a normal subgroup of $\mathrm{Gal}(\widetilde{k}/\mathbb{Q})$
%and that the characteristic ideal of $X_{k_{\infty}^c}$ has a square-free generator.
Suppose that $\mu(k_{\infty}^a/k)=0$,
then we have $\lambda(k_{\infty}^a/k) \geq [L_{k} \cap \widetilde{k}:k] + 1$.
%Then we have $\mathrm{rank}_{\mathbb{Z}_p}((X_{\widetilde{k}})_{\Gamma}) \geq \mathrm{rank}_{\mathbb{Z}_p}(Y_{\Gamma})$.
%There exist a $\Lambda$-submodule $Y$ of $X_{\widetilde{k}}$ such that $Y \cong \Lambda/f(S,T)$.
\end{prop}
%%%%%%%%%%%%%%%%%%%%%%%%%%%%%%%%%%%%%%%%%%%%%%%%%%%%%%%%%%%%%%%%%%%%%%%%%%%%%%%%%%%%%%%%%
%%%%%%%%%%%%%%%%%%%%%%%%%%%%%%%%%%%%%%%%%%%%%%%%%%%%%%%%%%%%%%%%%%%%%%%%%%%%%%%%%%%%%%%%%
\begin{proof}
We suppose that $\mu(k_{\infty}^a/k)=0$.
By Proposition \ref{split},
$p$ splits completely in $L_k \cap \widetilde{k}$.
Hence we have $\lambda(k_{\infty}^a/k) \geq [L_{k} \cap \widetilde{k}:k]=p^s$.
%using classical Iwasawa theory.
We assume
%that $\mu(k_{\infty}^a/k)=0$ and
that $\lambda(k_{\infty}^a/k) = [L_{k} \cap \widetilde{k}:k]$.
%Since $k_{s}^a/k$ is an abelian extension, Leopoldt's conjecture holds for $p$ and $k_s^a$ (\cite{Bu}).
%Hence we have $\mathrm{Gal}(\widetilde{k_s^a}/k) \cong \mathbb{Z}_p^{s+1}$,
%where $\widetilde{k_s^a}$ the composite of all $\mathbb{Z}_p$-extension fields of $k_{s}^a$.
%Since $p$ splits in $k$, $\widetilde{k_s^a}/k_{\infty}^a$ is an unramified extension.
By the same reason as in the proof of Proposition \ref{normal g},
we have a surjective homomorphism
$
X_{k_{\infty}^a} \rightarrow \mathrm{Gal}(\widetilde{k_s^a}/k_{\infty}^a).
$
Then we have $\mathrm{char}_{\mathbb{Z}_p[[S]]}(X_{k_{\infty}^a})$ $=\mathrm{char}_{\mathbb{Z}_p[[S]]}(\mathrm{Gal}(\widetilde{k_s^a}/k_{\infty}^a))$
because we suppose that $\lambda(k_{\infty}^a/k)$ $= [L_{k} \cap \widetilde{k}:k]$.
Using
$\mathrm{char}_{\mathbb{Z}_p[[S]]}(\mathrm{Gal}(\widetilde{k_s^a}/k_{\infty}^a)) = ((1+S)^{p^s}-1)$,
we get  $\mathrm{char}_{\mathbb{Z}_p[[S]]}(X_{k_{\infty}^a}) =((1+S)^{p^s}-1).$
By Proposition \ref{cycsub}, we have a $\Lambda$-submodule $Y$ of $X_{\widetilde{k}}$
such that
$Y$
%\cong
is isomorphic to $\Lambda/(f(S,T))$
and have a exact sequence
\[
0 \rightarrow Y \rightarrow X_{\widetilde{k}} \rightarrow \mathrm{Coker} \rightarrow 0.
\]
%where we put $\mathrm{Coker}=\mathrm{Coker}( Y \rightarrow X_{\widetilde{k}}).$
The $\Lambda$-modules
$(X_{\widetilde{k}})^{\mathrm{Gal}(\widetilde{k}/k_{\infty}^a)}$ and
$\mathrm{Coker}^{\mathrm{Gal}(\widetilde{k}/k_{\infty}^a)}$
are pseudo-null as $\Lambda$-modules
because
$T$ is coprime to $f(S,T)$.
Hence we get the exact sequence
%Using snake lemma, we get
\begin{eqnarray*}
0
%&\rightarrow &Y^{\mathrm{Gal}(\widetilde{k}/k_{\infty}^a)} \rightarrow (X_{\widetilde{k}})^{\mathrm{Gal}(\widetilde{k}/k_{\infty}^a)}
\rightarrow   \mathrm{Coker}^{\mathrm{Gal}(\widetilde{k}/k_{\infty}^a)}
\rightarrow    Y_{\mathrm{Gal}(\widetilde{k}/k_{\infty}^a)} \rightarrow (X_{\widetilde{k}})_{\mathrm{Gal}(\widetilde{k}/k_{\infty}^a)} \rightarrow 
\mathrm{Coker}_{\mathrm{Gal}(\widetilde{k}/k_{\infty}^a)} \rightarrow 0.
\end{eqnarray*}
%%%%%%%%%%%%%%%%%%%%%%%%%%%%%%%%%%%%%%%%%%%%%%%%%%%%%%%%%%%%5
%%%%%%%%%%%%%%%%%%%%%%%%%%%%%%%%%%%%%%%%%%%%%%%%%%%%%%%%%%%%%%
%%%%%%%%%%%%%%%%%%%%%%%%%%
%By Proposition \ref{cycsub},
%$\mathrm{Coker}^{\mathrm{Gal}(\widetilde{k}/k_{\infty}^a)}$.
Since we suppose that $\mu(k_{\infty}^a/k)=0$, $(X_{\widetilde{k}})_{\mathrm{Gal}(\widetilde{k}/k_{\infty}^a)}$ is a finitely
generated $\mathbb{Z}_p$-module.
By the same method as in the proof of Proposition \ref{rank},
we see that $Y_{\mathrm{Gal}(\widetilde{k}/k_{\infty}^a)} $ is finitely generated as a $\mathbb{Z}_p$-module.
This implies that $\mu(g_0(S))=0$ and that
\[
\mathrm{rank}_{\mathbb{Z}_p}(X_{\widetilde{k}}/T X_{\widetilde{k}})
\geq \mathrm{rank}_{\mathbb{Z}_p}(\mathbb{Z}_p[[S]]/ g_0(S) \mathbb{Z}_p[[S]]).
\]
By Lemma \ref{f(S,T)}, there exists a surjective homomorphism
$\left( \Lambda /f(S,T) \Lambda \right)^{\oplus r} \rightarrow  X_{\widetilde{k}}$,
where $r$ is a positive integer.
Hence
%we get a surjective homomorphism
this homomorphism induces a surjective homomorphism
$\left( \mathbb{Z}_p[[S]] /g_0(S) \mathbb{Z}_p[[S]] \right)^{\oplus r}
\rightarrow   X_{\widetilde{k}}/T X_{\widetilde{k}}$.
Then we have $(g_0(S)^r) \subset \mathrm{char}_{\mathbb{Z}_p[[S]]}(X_{\widetilde{k}}/T X_{\widetilde{k}})$.
Since $(1+S)^{p^s}-1$ is square-free, we obtain
$\mathrm{char}_{\mathbb{Z}_p[[S]]}(X_{\widetilde{k}}/T X_{\widetilde{k}})$ 
$\subset (g_0(S)).$
This implies that
\[
\mathrm{char}_{\mathbb{Z}_p[[S]]}(X_{k_{\infty}^a}) \subset (Sg_0(S)).
\]
Therefore there exists a power series $v(S)$ such that
$(1+S)^{p^s}-1=Sg_0(S)v(S)$.
Hence we have $p^s=g_0(0)v(0)$ and $s \geq \mathrm{ord}_p(g_0(0))$.
This contradicts Proposition \ref{index}.
Thus we get the conclusion.
%we put $\mathrm{Coker}=\mathrm{Coker}( Y \rightarrow X_{\widetilde{k}}).$
\end{proof}
%%%%%%%%%%%%%%%%%%%%%%%%%%%%%%%%%%%%%%%%%%%%%%%%%%%%%%%%%%%%%%%%%%%%%%%%%%%%%%%%%%%%%%%%%
%%%%%%%%%%%%%%%%%%%%%%%%%%%%%%%%%%%%%%%%%%%%%%%%%%%%%%%%%%%%%%%%%%%%%%%%%%%%%%%%%%%%%%%%%
%%%%%%%%%%%%%%%%%%%%%%%%%%%%%%%%%%%%%%%%%%%%%%%%%%%%%%%%%%%%%%%%%%%%%%%%%%%%%%%%%%%%%%%%%
\begin{prop}\label{g0}
Assume the same condition as in Proposition \ref{cycsub}.
Assume also that $\mathfrak{D_{p}}$ is not a normal subgroup of $\mathrm{Gal}(\widetilde{k}/\mathbb{Q})$.
%Assume that weak {\rm{GGC}} does not hold for $p$ and $k$.
%Assume also that $\mathfrak{D_{p}}$ is not a normal subgroup of $\mathrm{Gal}(\widetilde{k}/\mathbb{Q})$
%and that the characteristic ideal of $X_{k_{\infty}^c}$ has a square-free generator.
Suppose that $\mu(g_{0}(S)) =0$, then we have $\lambda(g_{0}(S)) \geq [L_{k} \cap \widetilde{k}:k]$.
%Then we have $\mathrm{rank}_{\mathbb{Z}_p}((X_{\widetilde{k}})_{\Gamma}) \geq \mathrm{rank}_{\mathbb{Z}_p}(Y_{\Gamma})$.
%There exist a $\Lambda$-submodule $Y$ of $X_{\widetilde{k}}$ such that $Y \cong \Lambda/f(S,T)$.
\end{prop}
%%%%%%%%%%%%%%%%%%%%%%%%%%%%%%%%%%%%%%%%%%%%%%%%%%%%%%%%%%%%%%%%%%%%%%%%%%%%%%%%%55
%%%%%%%%%%%%%%%%%%%%%%%%%%%%%%%%%%%%%%%%%%%%%%%%%%%%%%%%%%%%%%%%%%%%%%%%%%%%%%%%%%%%%
\begin{proof}
By the same argument in Proposition \ref{rankXM}, we have
\[
\mathrm{char}_{\mathbb{Z}_p[[S]]}(X_{k_{\infty}^a})  \subset
\mathrm{char}_{\mathbb{Z}_p[[S]]}(\mathrm{Gal}(\widetilde{{k_{s}^a}}/k_{\infty}^a)) = ((1+S)^{p^s}-1)
\]
and have a surjective homomorphism
\[
\left( \mathbb{Z}_p[[S]] /g_0(S) \mathbb{Z}_p[[S]] \right)^{\oplus r}
\rightarrow  X_{\widetilde{k}}/T X_{\widetilde{k}}
\]
for some positive integer $r$.
Thus we get
\[
(Sg_0(S)^r) \subset \mathrm{char}_{\mathbb{Z}_p[[S]]}(X_{k_{\infty}^a}) \subset ((1+S)^{p^s}-1).
\]
Since $(1+S)^{p^s}-1$ is square-free, we have
$(Sg_0(S)) \subset  ((1+S)^{p^s}-1)$.
%has no double roots in an algebraic closure of $\mathbb{Q}_p$.
By the $p$-adic Weierstrass preparation theorem (\cite[Theorem 7.3]{Wa}),
there exist a distinguished polynomial $\widetilde{g_0}(S) \in \mathbb{Z}_p[[S]]$
and a unit $u(S) \in \mathbb{Z}_p[[S]]^{\times}$ such that 
\[
g_0(S) = \frac{(1+S)^{p^s}-1}{S}\widetilde{g_0}(S)u(S).
\]
Thus we obtain $\mathrm{ord}_p(\widetilde{g_0}(0)) >0$ by Proposition \ref{index}.
This implies that $\lambda(g_0(S)) \geq p^s$ if $\mu(g_{0}(S)) =0$.
\end{proof}
%%%%%%%%%%%%%%%%%%%%%%%%%%%%%%%%%%%%%%%%%%%%%%%%%%%%%%%%%%%%%%%%%%%%%%%%%%%%%%%%%%%%%%%%%
By the following proposition, we can prove the inequality $\mathrm{(B)}$.
%%%%%%%%%%%%%%%%%%%%%%%%%%%%%%%%%%%%%%%%%%%%%%%%%%%%%%%%%%%%%%%%%%%%%%%%%%%%%%%%%%%%%%%%%
\begin{prop}\label{fprop}
Assume the same condition as in Proposition \ref{cycsub}.
Assume also that $\mathfrak{D_{p}}$ is not a normal subgroup of $\mathrm{Gal}(\widetilde{k}/\mathbb{Q})$.
%Assume that weak {\rm{GGC}} does not hold for $p$ and $k$.
%Assume also that $\mathfrak{D_{p}}$ is not a normal subgroup of $\mathrm{Gal}(\widetilde{k}/\mathbb{Q})$.
Then the inequality $\mathrm{(B)}$ holds.
%we have $\mathrm{rank}_{\mathbb{Z}_p}((X_{\widetilde{k}})_{\Gamma}) \geq [L_{k} \cap \widetilde{k}:k]$.
%There exist a $\Lambda$-submodule $Y$ of $X_{\widetilde{k}}$ such that $Y \cong \Lambda/f(S,T)$.
\end{prop}
%%%%%%%%%%%%%%%%%%%%%%%%%%%%%%%%%%%%%%%%%%%%%%%%%%%%%%%%%%%%%%%%%%%%%%%%%%%%%
%%%%%%%%%%%%%%%%%%%%%%%%%%%%%%%%%%%%%%%%%%%%%%%%%%%%%%%%%%%%%%%%%%%%%%%%%%%%%%%%%%%%%
\begin{proof}
We use the same method as in the proof of Proposition \ref{normal rankXM}.
By Proposition \ref{cycsub}, there exists a $\Lambda$-submodule $Y$ of $X_{\widetilde{k}}$
such that $Y$
%\cong
is isomorphic to $\Lambda/(f(S,T))$.
We will prove that $\mathrm{rank}_{\mathbb{Z}_p}(Y_{\Gamma}) \geq p^s$.
%We note that $Y_{\Gamma}$ is a finitely generated $\mathbb{Z}_p$-module
%by the proof of Proposition \ref{rank}.
There exists a unit $u$ such that
\[
\mathrm{Gal}(\widetilde{k}/N_{\infty}) = \overline{ \langle \sigma^{up^{s}} \tau \rangle.}
\]
%We put $T_{s}=(1+S)^{up^s}(1+T)-1.$
Then we have
\begin{eqnarray*}
Y_{\Gamma}
\cong  \mathbb{Z}_p[[S]] /(f(S, (1+S)^{-up^s}-1)).
\end{eqnarray*}
%By the definition of $f(S,T)$ and Proposition \ref{g0},
By the proof of Proposition \ref{cycsub},
$Y_{\Gamma}$ is a finitely generated $\mathbb{Z}_p$-module.
Hence we have $\mu(f(S,(1+S)^{-up^s}-1))=0$.
Furthermore, we have
\begin{eqnarray*}
&~&f(S,  (1+S)^{-up^s}-1)
%&=&
%\{(1+S)^{-up^s}-1\}^{\lambda^{\ast}} + \sum_{i=0}^{\lambda^{\ast}-1} g_{i}(S) \{ (1+S)^{-up^s}-1\}^{i}\\
%&\equiv & \{(1+S)^{-up^s}-1\}^{\lambda^{\ast}}+
%\sum_{i=0}^{\lambda^{\ast}-1} g_{i}(S) \left\{ \sum_{j=1}^{\infty} \dbinom{-u}{j}S^{jp^s} \right\}^{i}
%\quad \mathrm{~mod~}p\\
\equiv 
\begin{cases}
g_0(S)  \quad \mathrm{~mod~}(p, S^{p^s}) \quad &\mathrm{~if~} \mu(g_0(S))=0,\\
0       \hspace{6.3mm} \quad \mathrm{~mod~}(p, S^{p^s}) \quad &\mathrm{~if~} \mu(g_0(S))>0.
\end{cases} 
\end{eqnarray*}
Using Proposition \ref{g0}, we have $\mathrm{rank}_{\mathbb{Z}_p}(Y_{\Gamma}) \geq p^s$.
By Proposition \ref{rank}, we obtain  $\mathrm{rank}_{\mathbb{Z}_p}(X_{\widetilde{k}}) \geq p^s$.
This implies the inequality (B) from Proposition \ref{index}.
Thus we get the conclusion.
\end{proof}
Combining the results of Sections \ref{A} and \ref{B},
we have completed the proof of Theorem \ref{main thm}. 
%%%%%%%%%%%%%%%%%%%%%%%%%%%%%%%%%%%%%%%%%%%%%%%%%%%%%%%%%%%%%%%%%%%%%%%%%%%%%%%%%%%%%%%%%
%%%%%%%%%%%%%%%%%%%%%%%%%%%%%%%%%%%%%%%%%%%%%%%%%%%%%%%%%%%%%%%%%%%%%%%%%%%%%%%%%%%%%%%%%
%%%%%%%%%%%%%%%%%%%%%%%%%%%%%%%%%%%%%%%%%%%%%%%%%%%%%%%%%%%%%%%%%%%%%%%%%%%%%%%%%%%%%%%%%
\begin{Rem}\label{frem}
\begin{rm}
Theorem \ref{main thm} holds without the assumption (ii) if $X_{\widetilde{k}}$ is not cyclic as a $\Lambda$-module.
Indeed, we have proved the inequality (A) under the assumption (i) in Theorem \ref{main thm}.
Furthermore, we have verified the inequality (B) in Propositions \ref{normal rankXM} and \ref{fprop}.
In these propositions, we assumed (ii) in the case where $X_{\widetilde{k}}$ is not cyclic as a $\Lambda$-module.
%To prove Proposition \ref{fprop},
%we do not need the assumption that
%the characteristic ideal of $X_{k_{\infty}^c}$ has a square-free generator
%by Propositions \ref{isom str} and \ref{cycsub}.
%in the case where $X_{\widetilde{k}}$ is cyclic as a $\Lambda$-module.
%We verified Proposition \ref{cycsub} using the assumption (ii) of Theorem \ref{main thm}.
%Therefore 
%if  $X_{\widetilde{k}}$ is cyclic.
\end{rm}
%Suppose that $\mathrm{ord}_p(g_0(0)) \neq [L_k \cap \widetilde{k}]$.
%Then $k$ is not $p$-split $p$-rational.
\end{Rem}
%%%%%%%%%%%%%%%%%%%%%%%%%%%%%%%%%%%%%%%%%%%%%%%%%%%%%%%%%%%%%%%%%%%%%%%%%%%%%%%%%%%%%%%%%
%%%%%%%%%%%%%%%%%%%%%%%%%%%%%%%%%%%%%%%%%%%%%%%%%%%%%%%%%%%%%%%%%%%%%%%%%%%%%%%%%%%%%%%%%
%%%%%%%%%%%%%%%%%%%%%%%%%%%%%%%%%%%%%%%%%%%%%%%%%%%%%%%%%%%%%%%%%%%%%%%%%%%%%%%%%%%%%%%%%
\section{Proof of Theorem \ref{main thm 2}}\label{rational case}
In this section, we prove Theorem \ref{main thm 2}.
In the following, we prove some lemmas and propositions needed later.
Let $k$ be a $p$-split $p$-rational field.
By Definition \ref{rational}, we have
$\widetilde{k}^{\mathfrak{D}_{{\mathfrak{p}}}}=\widetilde{k}^{\mathfrak{D}_{{\mathfrak{p}}^{\ast}}}$
and $[\mathrm{Gal}(\widetilde{k}/k): \mathfrak{D}_{{\mathfrak{p}}}] \geq p$.
%We put
%$\nu_{m} (S) =\displaystyle{\frac{(1+S)^{p^{m}}-1}{S}}$
%for a non-negative integer $m$.
We put
$\nu_{m}(S)=\displaystyle{((1+S)^{p^m}-1)/S}$ for a non-negative integer $m$.
As we stated in Section \ref{In}, we see that $X_{\widetilde{k}}$ is $\Lambda$-cyclic
by the following 
%%%%%%%%%%%%%%%%%%%%%%%%%%%%%%%%%%%%%%%%%%%%%%%%%%%%%%%%%%%%%%%%%%%%%%%%%%%%%%%%%%%%%%%%%%%%%
%%%%%%%%%%%%%%%%%%%%%%%%%%%%%%%%%%%%%%%%%%%%%%%%%%%%%%%%%%%%%%%%%%%%%%%%%%%%%%%%%%%%%%%%%%%%%
%%%%%%%%%%%%%%%%%%%%%%%%%%%%%%%%%%%%%%%%%%%%%%%%%%%%%%%%%%%%%%%%%%%%%%%%%%%%%%%%%%%%%%%%%%%%%
\begin{prop}[{\cite[Propositions 3.5 and 3.8]{Mu}}]\label{g00}
%Let $k$ be a $p$-split $p$-rational field.
Assume that $L_k \subset \widetilde{k}$
and that $\lambda(k_{\infty}^c/k) \geq 2$.
Then we have a surjective homomorphism
$$
\Lambda/f(S,T) \Lambda \rightarrow X_{\widetilde{k}}
$$
as a $\Lambda$-module, where $f(S,T)$ is the same power series defined in Lemma \ref{f(S,T)}.
Furthermore,
if $\mathfrak{D_{p}}$ is a normal subgroup of $\mathrm{Gal}(\widetilde{k}/\mathbb{Q})$,
then there exists a power series $u(S) \in \mathbb{Z}_p[[S]]^{\times}$ such that
\[
f(S,0) =
%\frac{(1+S)^{p^{n_0}}-1}{S} u(S),
\nu_{n_0}(S)u(S),
\]
where $n_0$ is the non-negative integer satisfying $[\mathrm{Gal}(\widetilde{k}/k) :\mathfrak{D}_{{\mathfrak{p}}}]=p^{n_0}$.
\end{prop}
%%%%%%%%%%%%%%%%%%%%%%%%%%%%%%%%%%%%%%%%%%%%%%%%%%%%%%%%%%%%%%%%%%%%%%%%%%%%%%%%%%%%%%%%%%%
%%%%%%%%%%%%%%%%%%%%%%%%%%%%%%%%%%%%%%%%%%%%%%%%%%%%%%%%%%%%%%%%%%%%%%%%%%%%%%%%%%%%%%%%%%%
We have $\lambda(N_{\infty}/k)=0$ by genus formula
(see for example \cite[Remark $3.2$]{Mu}).
By Proposition \ref{g00},
we can determine the isomorphism class of $X_{\widetilde{k}}$ as a $\Lambda$-module
provided that weak GGC does not hold for $p$ and $k$.
In fact, by Proposition \ref{isom str}, we have an isomorphism
\[
X_{\widetilde{k}} \cong \Lambda/ f(S,T) \Lambda.
\]
%%%%%%%%%%%%%%%%%%%%%%%%%%%%%%%%%%%%%%%%%%%%%%%%%%%%%%%%%%%%%%%%%%%%%%%%%%%%%%%%%%%%%%%%%%
%%%%%%%%%%%%%%%%%%%%%%%%%%%%%%%%%%%%%%%%%%%%%%%%%%%%%%%%%%%%%%%%%%%%%%%%%%%%%%%%%%%%%%%%%%
%%%%%%%%%%%%%%%%%%%%%%%%%%%%%%%%%%%%%%%%%%%%%%%%%%%%%%%%%%%%%%%%%%%%%%%%%%%%%%%%%%%
%%%%%%%%%%%%%%%%%%%%%%%%%%%%%%%%%%%%%%%%%%%%%%%%%%%%%%%%%%%%%%%%%%%%%%%%%%%%%%%%%%%
Using the isomorphism above,
we can determine the characteristic ideal of $X_{k_{\infty}^a}$.
%%%%%%%%%%%%%%%%%%%%%%%%%%%%%%%%%%%%%%%%%%%%%%%%%%%%%%%%%%%%%%%%%%%%%%%%%%%%%%%%%%%
%%%%%%%%%%%%%%%%%%%%%%%%%%%%%%%%%%%%%%%%%%%%%%%%%%%%%%%%%%%%%%%%%%%%%%%%%%%%%%%%%%5%%
\begin{prop}\label{ant-rational}
Let $k$ be a $p$-split $p$-rational field.
Assume that weak {\rm{GGC}} does not hold for $p$ and $k$. Then we have
%Assume the same conditions as in Proposition \ref{isom str}.
%Then we have
\[
\mathrm{char}_{\mathbb{Z}_p[[S]]}(X_{k_{\infty}^a}) = ((1+S)^{p^{n_0}}-1).
\]
\end{prop}
%%%%%%%%%%%%%%%%%%%%%%%%%%%%%%%%%%%%%%%%%%%%%%%%%%%%%%%%%%%%%%%%%%%%%%%%%%5
\begin{proof}
By Proposition \ref{isom str}, we have
\[
X_{\widetilde{k}}/TX_{\widetilde{k}} \cong \mathbb{Z}_p[[S]]/g_0(S) \mathbb{Z}_p[[S]].
\]
We note that $f(S,0)=g_0(S)$.
Hence we obtain $\mathrm{char}_{\mathbb{Z}_p[[S]]}(X_{k_{\infty}^a}) = ((1+S)^{p^{n_0}}-1)$
by Lemma \ref{Ozaki} and Proposition \ref{g00}.
\end{proof}
%%%%%%%%%%%%%%%%%%%%%%%%%%%%%%%%%%%%%%%%%%%%%%%%%%%%%%%%%%%%%%%%%%%%%%%%%%%%%%
As in Section \ref{B}, we put $p^s = [L_{k} :k]$.
%Hence we have $L_k \cap \widetilde{k} = k_s^a$.
Then there exists a unit $u \in \mathbb{Z}_p^{\times}$ such that
\[
\mathrm{Gal}(\widetilde{k}/N_{\infty}) = \overline{ \langle \sigma^{up^{s}} \tau \rangle}
\]
by Lemma \ref{Zp}.\par
By the same methods as Sections \ref{A} and \ref{B},
we have the following
%%%%%%%%%%%%%%%%%%%%%%%%%%%%%%%%%%%%%%%%%%%%%%%%%%%%%%%%%%%%%%%%%%%%%%%%%%%%%%%%%%%
\begin{prop}\label{Zprank}
%Let $k$ be a $p$-split $p$-rational field.
%Assume that weak {\rm{GGC}} does not hold for $p$ and $k$.
%Then we have
Under the same assumption as in Proposition \ref{ant-rational},
we have
\[
%\mathrm{rank}_{\mathbb{Z}_p}
(X_{\widetilde{k}})_{\mathrm{Gal}(\widetilde{k}/N_{\infty})}
 \cong
%\mathbb{Z}_p^{\oplus p^{n_0}-1}.
\mathbb{Z}_p[[S]]/\nu_{n_0}(S) \mathbb{Z}_p[[S]].
\]
\end{prop}
%%%%%%%%%%%%%%%%%%%%%%%%%%%%%%%%%%%%%%%%%%%%%%%%%%%%%%%%%%%%%%%%%%%%%%%%%%%%%%%%%%%%%%%%%%%
\begin{proof}
By the same reason as in the proof of Proposition \ref{fprop},
we have
\[
(X_{\widetilde{k}})_{\mathrm{Gal}(\widetilde{k}/N_{\infty})} \cong \mathbb{Z}_p[[S]] /f(S,(1+S)^{-up^s}-1) \mathbb{Z}_p[[S]].
\]
We note that $\mu(f(S,(1+S)^{-up^s}-1))=0$.
Using Proposition \ref{rank uper}, we have
$\mathrm{rank}_{\mathbb{Z}_p}((X_{\widetilde{k}})_{\mathrm{Gal}(\widetilde{k}/N_{\infty})}) \leq p^{n_0}-1$.
By the same method as in the proof of Propositions \ref{normal rankXM}, \ref{fprop}, and \ref{g00}, we obtain
$\lambda(f(S,(1+S)^{-up^s}-1)) \geq p^{n_0}-1$.
Thus we obtain $\lambda(f(S,(1+S)^{-up^s}-1)) = p^{n_0}-1$.
In the same way as Proposition \ref{(A)}, we can verify that
$M_{\mathfrak{p}^{\ast}}(N_{\infty})/N_{n_0}$ is an abelian extension.
This implies that $M_{\mathfrak{p}^{\ast}}(N_{\infty})=\widetilde{k_{n_0}^a}$.
Furthermore, we have $\widetilde{k_{n_0}^a} = L_{k_{\infty}^a}$
since $X_{k_{\infty}^a}$ is a free $\mathbb{Z}_p$-module
from the proof of Proposition \ref{ant-rational}.
Therefore we have
\[
(X_{\widetilde{k}})_{\mathrm{Gal}(\widetilde{k}/N_{\infty})} \cong \mathrm{Gal}( L_{k_{\infty}^a}/\widetilde{k})
\cong
\mathbb{Z}_p[[S]] / g_0(S) \mathbb{Z}_p[[S]].
\]
Thus we get the conclusion.
\end{proof}
%%%%%%%%%%%%%%%%%%%%%%%%%%%%%%%%%%%%%%%%%%%%%%%%%%%%%%%%%%%%%%%%%%%%%%%%%%%%%%%%%%%%%%%%%
%%%%%%%%%%%%%%%%%%%%%%%%%%%%%%%%%%%%%%%%%%%%%%%%%%%%%%%%%%%%%%%%%%%%%%%%%%%%%%%%%%%%%%%%%%%%
%We put $\Gamma=\mathrm{Gal}(\widetilde{k}/N_{\infty})$.
We put $\rho = \sigma^{up^s}\tau $.
We fix an isomorphism
\begin{eqnarray}
\mathbb{Z}_p[[\mathrm{Gal}(\widetilde{k}/N_{\infty})]] \cong \mathbb{Z}_p[[T_s]] \quad (\rho \leftrightarrow 1+T_s). \label{rho}
\end{eqnarray}
Then we have $T_s =(1+S)^{up^s}(1+T)-1.$
We note that $\Lambda=\mathbb{Z}_p[[S,T_s]]$ since we have
$T=(1+T_s)(1+S)^{-up^s}-1$.
By the same method as Lemma \ref{f(S,T)}, we can find an annihilator of $X_{\widetilde{k}}$
as follows.
%a $\Lambda$-module.
%%%%%%%%%%%%%%%%%%%%%%%%%%%%%%%%%%%%%%%%%%%%%%%%%%%%%%%%%%%%%%%%%%%%%%%%%%%%%%%%%%%%%%%%%%%%%%%%%%%%%%%%%%%%%%%%%5
%%%%%%%%%%%%%%%%%%%%%%%%%%%%%%%%%%%%%%%%%%%%%%%%%%%%%%%%%%%%%%%%%%%%%%%%%%%%%%%%%%%%%%%%%%%%%%%%%%%%%%%%%%%%%%%%555
\begin{lem}\label{q(S,V)}
%Let $k$ be a $p$-split $p$-rational field.
%Assume that weak {\rm{GGC}} does not hold for $p$ and $k$.
%Suppose that $\lambda^{\ast}\geq 1$, where $\lambda^{\ast}$ is the integer defined above.
Assume the same conditions as in Proposition \ref{ant-rational}.
Then there exist power series $q(S,T_s) \in \mathrm{Ann}_{\Lambda}(X_{\widetilde{k}})$ and
$r_i(T_s) \in \mathbb{Z}_p[[T_s]]$ $(i=0, \dots , p^{n_0}-2)$
such that
$$
q(S,T_s) = S^{p^{n_0}-1} + r_{p^{n_0}-2}(T_s) S^{p^{n_0}-2} + \dots + r_{1}(T_s)S +r_{0}(T_s).
$$
%where $\lambda^{c}=\lambda(k_{\infty}^c / k)$.
\end{lem}
%%%%%%%%%%%%%%%%%%%%%%%%%%%%%%%%%%%%%%%%%%%%%%%%%%%%%%%%%%%%%%%%%%%%%%%%%%%%%%%%%%%%%%
%%%%%%%%%%%%%%%%%%%%%%%%%%%%%%%%%%%%%%%%%%%%%%%%%%%%%%%%%%%%%%%%%%%%%%%%%%%%%%%%%%%%%%%%%
\begin{proof}
Using the isomorphism (\ref{rho}) and Proposition \ref{Zprank}, we have
\[
X_{\widetilde{k}}/T_s X_{\widetilde{k}} \cong \mathbb{Z}_p^{\oplus p^{n_0}-1}.
\]
By Nakayama's lemma, there exist $y_i \in X_{\widetilde{k}} ~~(i=1,2, \dots, p^{n_0}-1)$
such that
$X_{\widetilde{k}}= \langle y_1, \dots, y_{p^{n_0}-1}
\rangle_{\mathbb{Z}_p[[T_s]]}$.
By the same method as Lemma \ref{f(S,T)}, we have relations
\[
Sy_i = \sum_{j=1}^{p^{n_0}-1}h_{ij}(T_s)y_j \quad (i=1,2, \dots, p^{n_0}-1)
\]
for some $h_{ij}(T_s) \in \mathbb{Z}_p[[T_s]]$.
We put $c=p^{n_0}-1$.
From theses relations, we have a matrix
\begin{eqnarray*}
A=
~\left(
\begin{array}{ccccc}
S-h_{11}(T_s) &  -h_{12}(T_s) & \dots & -h_{1 c}(T_s)\\
-h_{21}(T_s) &S-h_{22}(T_s) &  \dots & -h_{2 c}(T_s) \\
  \dots &  \dots  &   \dots      &  \dots \\
-h_{c 1}(T_s)  &  -h_{c 2}(T_s) & \dots & S-h_{c c}(T_s)
\end{array}
\right).
\end{eqnarray*}
We put $q(S,T_s) = \mathrm{det}(A)$.
As in Lemma \ref{f(S,T)}, $q(S,T_s)$ satisfies the desired result.
Thus we get the conclusion.
\end{proof}
%%%%%%%%%%%%%%%%%%%%%%%%%%%%%%%%%%%%%%%%%%%%%%%%%%%%%%%%%%%%%%%%%%%%%%%%%%%%%%%%%%%%%%
%%%%%%%%%%%%%%%%%%%%%%%%%%%%%%%%%%%%%%%%%%%%%%%%%%%%%%%%%%%%%%%%%%%%%%%%%%%%%%%%%%%%%%%%%%%%%%%%%%%%%%%%%%%%%%%%%%%%%%
%%%%%%%%%%%%%%%%%%%%%%%%%%%%%%%%%%%%%%%%%%%%%%%%%%%%%%%%%%%%%%%%%%%%%%%%%%%%%%%%%%%%%%%%%%%%%%%%%%%%%%%%%
%\begin{Rem}\label{uni of f}
%\begin{rm}
We note that the uniqueness of the power series $q(S,T_s)$ is not known.
In this paper, we fix this power series.
%\end{rm}
%\end{Rem}
%%%%%%%%%%%%%%%%%%%%%%%%%%%%%%%%%%%%%%%%%%%%%%%%%%%%%%%%%%%%%%%%%%%%%%%%%%%%%%%%%%%%%%%%%%%%%%%%%%%%%%%%%%%%%%%%%%%%%%
%%%%%%%%%%%%%%%%%%%%%%%%%%%%%%%%%%%%%%%%%%%%%%%%%%%%%%%%%%%%%%%%%%%%%%%%%%%%%%%%%%%%%%%%%%%%%%%%%%%%%%%%%
%%%%%%%%%%%%%%%%%%%%%%%%%%%%%%%%%%%%%%%%%%%%%%%%%%%%%%%%%%%%%%%%%%%%%%%%%%%%%%%%%%%%%%%%%
%%%%%%%%%%%%%%%%%%%%%%%%%%%%%%%%%%%%%%%%%%%%%%%%%%%%%%%%%%%%%%%%%%%%%%%%%%%%%%%%%%%%%%%%%%%%%%%%%%%%%%%%%%%%%%%%%5
%%%%%%%%%%%%%%%%%%%%%%%%%%%%%%%%%%%%%%%%%%%%%%%%%%%%%%%%%%%%%%%%%%%%%%%%%%%%%%%%%%%%%%%%%%%%%%%%%%%%%%%%%%%%%%%%555
\begin{lem}\label{q=f}
%Let $k$ be a $p$-split $p$-rational field.
%Assume that weak {\rm{GGC}} does not hold for $p$ and $k$.
%Then we have
Assume the same conditions as in Proposition \ref{ant-rational}.
Then we have
\[
\mathrm{Ann}_{\Lambda}(X_{\widetilde{k}})=(q(S,T_s)),
\]
where $q(S,T_s)$ is the same power series defined in Lemma \ref{q(S,V)}.
\end{lem}
%%%%%%%%%%%%%%%%%%%%%%%%%%%%%%%%%%%%%%%%%%%%%%%%%%%%%%%%%%%%%%%%%%%%%%%%%%%%%%%%%%%%%%
%%%%%%%%%%%%%%%%%%%%%%%%%%%%%%%%%%%%%%%%%%%%%%%%%%%%%%%%%%%%%%%%%%%%%%%%%%%%%%%%%%%%%%%%%
\begin{proof}
By Proposition \ref{isom str} and Lemma \ref{q(S,V)},
we have $q(S,T_s) \in \mathrm{Ann}_{\Lambda}(X_{\widetilde{k}})=(f(S,T))$.
Then, there exists a power series $Q(S,T_s) \in \Lambda$ such that
\[
q(S,T_s) =f(S,(1+T_s)(1+S)^{-up^s}-1) Q(S,T_s).
\]
Hence we have $q(S,0)=f(S,(1+S)^{-up^s}-1)Q(S,0)$.
By the definition of $q(S,T_s)$, we have  $\mu(q(S,0))=0$ and $\lambda(q(S,0)) \leq p^{n_0}-1$.
By Lemma \ref{q(S,V)} and Proposition \ref{g00}, we have a surjective homomorphism
$
\Lambda/q(S,T_s) \Lambda \rightarrow X_{\widetilde{k}}.
$
This homomorphism induces a surjective homomorphism
$
\mathbb{Z}_p[[S]]/q(S,0) \mathbb{Z}_p[[S]] \rightarrow
\mathbb{Z}_p[[S]]/f(S,(1+S)^{-up^s}-1) \mathbb{Z}_p[[S]].
$
Using Lemma \ref{Zprank},
%This implies that
we obtain
\[
\mathrm{rank}_{\mathbb{Z}_p}(q(S,0)) \geq p^{n_0}-1.
\]
Thus the surjective homomorphism above is an isomorphism.
%From Propositions \ref{normal rankXM} and \ref{g00}, we have $\lambda(f(S,(1+S)^{-up^s}-1))=p^{n_0}-1$ and $\mu((f(S,(1+S)^{-up^s}-1)))=0$.
Hence we obtain $Q(S,0) \in \mathbb{Z}_p[[S]]^{\times}$.
This implies that $Q(S,T)$ is a unit in $\Lambda$.
%$\mathrm{Ann}_{\Lambda}(X_{\widetilde{k}})=(q(S,V))$.
Thus we get the conclusion.
\end{proof}
%%%%%%%%%%%%%%%%%%%%%%%%%%%%%%%%%%%%%%%%%%%%%%%%%%%%%%%%%%%%%%%%%%%%%%%%%%%%%%%%%%%%%%
%%%%%%%%%%%%%%%%%%%%%%%%%%%%%%%%%%%%%%%%%%%%%%%%%%%%%%%%%%%%%%%%%%%%%%%%%%%%%%%%%%%%%%%%%
We put $H=N_{s+1}^{\ast}$ and $H_{\infty}=N_{\infty} N_{s+1}^{\ast}$.
Then $H_{\infty}/H$ is a $\mathbb{Z}_p$-extension of $H$
unramified outside all prime ideals of $H$ lying above $\mathfrak{p}$.
%As in Section \ref{B},
We prove the following
%%%%%%%%%%%%%%%%%%%%%%%%%%%%%%%%%%%%%%%%%%%%%%%%%%%%%%%%%%%%%%%%%%%%%%%%%%%%%%%%%%%
\begin{prop}\label{Zprank2}
Under the same assumption as in Proposition \ref{ant-rational},
%Let $k$ be a $p$-split $p$-rational field.
%Assume that weak {\rm{GGC}}
%does not hold for $p$ and $k$.
we have
\[
\mathrm{rank}_{\mathbb{Z}_p}((X_{\widetilde{k}})_{\mathrm{Gal}(\widetilde{k}/H_{\infty})}) = p(p^{n_0}-1).
\]
\end{prop}
%%%%%%%%%%%%%%%%%%%%%%%%%%%%%%%%%%%%%%%%%%%%%%%%%%%%%%%%%%%%%%%%%%%%%%%%%%%%%%%%%%%%%%%%%%%
\begin{proof}
We note that
$\mathrm{Gal}(\widetilde{k}/ H_{\infty}) = \overline{ \langle \rho^p \rangle}$.
By Lemma \ref{q=f}, we have
\begin{eqnarray*}
(X_{\widetilde{k}})_{\mathrm{Gal}(\widetilde{k}/H_{\infty})}
& \cong & \mathbb{Z}_p[[S,T_s]] /(q(S,T_s), (1+T_s)^p -1)\\
%& \cong & (\mathbb{Z}_p[[V]]/((1+V)^p-1) \mathbb{Z}_p[[V]])[[S]]/(q(S,V))\\
& \cong & (\mathbb{Z}_p[[T_s]]/((1+T_s)^p-1)\mathbb{Z}_p[[T_s]])^{\oplus p^{n_0}-1}\\
& \cong & \mathbb{Z}_p^{\oplus p(p^{n_0}-1)}.
\end{eqnarray*}
Thus we get the conclusion.
\end{proof}
%%%%%%%%%%%%%%%%%%%%%%%%%%%%%%%%%%%%%%%%%%%%%%%%%%%%%%%%%%%%%%%%%%%%%%%%%%%%%%%%%%%%%%%%%
%%%%%%%%%%%%%%%%%%%%%%%%%%%%%%%%%%%%%%%%%%%%%%%%%%%%%%%%%%%%%%%%%%%%%%%%%%%%%%%%%%%%%%%%%%%%
We will prove Theorem \ref{main thm 2}.
%%%%%%%%%%%%%%%%%%%%%%%%%%%%%%%%%%%%%%%%%%%
\begin{proof}[Proof of Theorem \ref{main thm 2}]
We assume that weak {\rm{GGC}} does not hold for $p$ and $k$.
From Proposition \ref{Fuj}, $M_{\mathfrak{p^\ast}}(\widetilde{k})$ coincides with $L_{\widetilde{k}}$.
We note that $\mathfrak{p^{\ast}}$ splits completely in $N_{\infty}^{ D_{\mathfrak{p^{\ast}}}} = N_{n_0}$.
Hence all prime ideals of $H$
lying above $\mathfrak{p}^{\ast}$ do not split in $H_{\infty}$.
%By Lemma \ref{cond normal},
%$p$ splits completely in $\widetilde{k}^{\mathfrak{D}_{{\mathfrak{p}}^{\ast}}}$.
Let
\[
\{ \mathfrak{P}_{\infty,i}^{\ast}~|~1 \leq  i \leq p^{n_0} \}
\]
be the set of prime ideals of
%$N_{\infty} \cdot N_{s+1}^{\ast}$
$H_{\infty}$ lying above $\mathfrak{p}^\ast$.
%%%%%%%%%%%%%%%%%%%%%%%%%%%%%%%%%%%%%%%%%%%%%%%%%%
We denote by $I_{\mathfrak{P}_{\infty,i}^{\ast}}$ the inertia subgroup of
$\mathrm{Gal}(M_{\mathfrak{p^{\ast}}}(H_{\infty}) /H_{\infty})$
for the prime $\mathfrak{P}_{\infty,i}^{\ast}$,
where $M_{\mathfrak{p^{\ast}}}(H_{\infty})$ is the maximal abelian pro-$p$ extension of $H_{\infty}$
unramified outside all prime ideals of $H_{\infty}$ lying above $\mathfrak{p}^{\ast}$.
We have $\mathrm{Gal}(M_{\mathfrak{p^\ast}}(H_{\infty})/\widetilde{k}) \cap I_{\mathfrak{P}_{\infty,i}^{\ast}} =1$ for each $i$
because $M_{\mathfrak{p^{\ast}}}(H_{\infty} ) / \widetilde{k}$ is unramified.
Hence we have $I_{\mathfrak{P}_{\infty,i}^{\ast}} \cong \mathbb{Z}_p$.
Using \cite{Gi, sch, OV},
we obtain $\mu(H_{\infty} / H)=0$
since $H/k$ is an abelian extension.
This implies that
$\mathrm{Gal}(M_{\mathfrak{p^{\ast}}}(H_{\infty} )/ H_{\infty})$ is a finitely generated as a $\mathbb{Z}_p$-module.
Furthermore, since we suppose that
$\lambda(H_{\infty}/ H)=0$, $L_{H_{\infty} }/H_{\infty}$ is a finite extension.
Thus we have
$\mathrm{rank}_{\mathbb{Z}_p}(\mathrm{Gal}(M_{\mathfrak{p^{\ast}}}(H_{\infty} )/ H_{\infty})) \leq  p^{n_0}$.
Therefore we obtain
\[
\mathrm{rank}_{\mathbb{Z}_p}(\mathrm{Gal}(M_{\mathfrak{p^{\ast}}}(H_{\infty})/ \widetilde{k})) \leq p^{n_0}-1.
\]
%By the definition of $n_{0}$,
%By Lemma \ref{cond normal},
%we have $p^{n_0} \leq [L_k \cap \widetilde{k} :k]$.
On the other hand, we have
\begin{eqnarray*}
\mathrm{rank}_{\mathbb{Z}_p}(\mathrm{Gal}(M_{\mathfrak{p^{\ast}}}(H_{\infty} )/ \widetilde{k}))
&=&
\mathrm{rank}_{\mathbb{Z}_p}((X_{\widetilde{k}})_{\mathrm{Gal}(\widetilde{k}/H_{\infty})})\\
&=& 
p(p^{n_0}-1)\\
&>&
p^{n_0}-1
\end{eqnarray*}
by Proposition \ref{Zprank2}.
This is a contradiction.
Thus we have completed the proof.
\end{proof} 
%%%%%%%%%%%%%%%%%%%%%%%%%%%%%%%%%%%%%%%%%%%%%%%%%%%%%%%%%%%%%%%%%%%%%%%%%%%%%%%%%%%%%%%%%
%%%%%%%%%%%%%%%%%%%%%%%%%%%%%%%%%%%%%%%%%%%%%%%%%%%%%%%%%%%%%%%%%%%%%%%%%%%%%%%%%%%%%%%%%%%%
\section{Numerical examples}\label{example}
In this section, we introduce some numerical examples which were computed using PARI/GP.\par
%We put
%$\Lambda=\mathbb{Z}_p[[S]]$ and
%$k=\mathbb{Q}(\sqrt{-d})$, where $d$ is a positive square-free integer.
%For each $n\geq 0 $, we denote by $k_n^{c}$ the intermediate field of the cyclotomic $\mathbb{Z}_p$-extension $k_\infty^{c}$
%such that $k_n^{c}$ is the unique cyclic extension over $k$ of degree $p^n$.
Let $A_{k_{n}^c}$ be the $p$-Sylow subgroup of the ideal class group of $k_{n}^{c}$.
%We put 
Then, by class field theory, we have
$\displaystyle{X_{k_{\infty}^{c}} \cong
\plim[n]A_{k_n^{c}}
%\underleftarrow{\lim} A_{k_n^{c}}
}$, where the inverse limit is
taken with respect to the relative norms. 
%Then 
As in Sections \ref{Preli} and \ref{decomp}, 
$X_{k_{\infty}^{c}}$ is a finitely generated torsion $\mathbb{Z}_p[[T]]$-module
via an fixed isomorphism (\ref{isom cyc}) in Section \ref{decomp}.
%\begin{eqnarray*}
%\mathbb{Z}_p[[ \mathrm{Gal}(K_\infty^{c}/k)]]  \cong \mathbb{Z}_p[[T]] \qquad  (\sigma \leftrightarrow 1+S),
%\label{isom ring}
%\end{eqnarray*}
%where $\sigma$ is a topological generator of $\mathrm{Gal}(K_\infty^{{\rm c}}/K)$.
Let $h(T)$ be the distinguished polynomial which generates
the characteristic ideal of $X_{k_{\infty}^{c}}$.
%$\mathrm{char}_{\mathbb{Z}_p[[T]]}(X_{k_{\infty}^{c}})$.
%and we can apply Theorem to the Iwasawa module $X_{K_\infty^{{\rm c}}}$.\par
We can calculate the polynomial $h(T)$ mod $p^n$ for small $n$ numerically
using PARI/GP.
%Mizusawa's program Iwapoly.ub
%(\cite[Research, Programing, Approximate Computation of Iwasawa Polynomials by UBASIC]{Miz}).
By this method,
we can check whether the characteristic ideal of $X_{k_{\infty}^c}$ has a square-free generator.
\par
We first give a criterion whether weak GGC holds,
using the characteristic ideal of $X_{k_{\infty}^{c}}$:
%Iwasawa module associated to the cyclotomic $\mathbb{Z}_p$-extension:
%%%%%%%%%%%%%%%%%%%%%%%%%%%%%%%%%%%%%%%%%%%%%%%%%%%%%%%%%%%%%%%%%%%%%%%%%%%%%%%%%%%%%%%%%%%
\begin{prop}\label{ex1}
Assume that $L_k \subset \widetilde{k}$, $\lambda(k_{\infty}^c/k)\geq 2$, and that
$\mathrm{ord}_p(g_0(0)) > [L_k :k]$,
where $g_0(S)$ is the same power series defined in Lemma \ref{f(S,T)}.
Assume also that
%$\mathrm{char}_{\mathbb{Z}_p[[\mathrm{Gal}(k_{\infty}^c/k)]]}(X_{k_{\infty}^c})$
the characteristic ideal of $X_{k_{\infty}^c}$
%$\mathrm{char}_{{\mathbb{Z}_p}[[\mathrm{Gal}(k_{\infty}^c/k)]]}(X_{k_{\infty}^c})$
has a square-free generator.
Then weak {\rm{GGC}} holds for $p$ and $k$.
\end{prop}
%%%%%%%%%%%%%%%%%%%%%%%%%%%%%%%%%%%%%%%%%%%%%%%%%%%%%%%%
\begin{proof}
We have $\lambda(N_{\infty}/k)=0$ by $L_k \subset \widetilde{k}$ from genus formula
(see for example \cite[Remark $3.2$]{Mu}).
By Proposition \ref{g00},
we have $\mathrm{ord}_p(g_0(0)) = [\mathrm{Gal}(\widetilde{k}/k) :\mathfrak{D}_{{\mathfrak{p}}}]$.
From $\mathrm{ord}_p(g_0(0)) > [L_k :k]$,
$\mathfrak{D_{p}}$ is not a normal subgroup of $\mathrm{Gal}(\widetilde{k}/\mathbb{Q})$
by Lemma \ref{cond normal}.
Using Proposition \ref{eqrational}, we see that $k$ is not a $p$-split $p$-rational
field.
Furthermore, $X_{\widetilde{k}}$ is cyclic as a $\Lambda$-module by Proposition \ref{g00}.
Therefore weak GGC holds for $p$ and $k$
by Theorem \ref{main thm} and Remark \ref{frem}.
%Theorem \ref{main thm}.
\end{proof}~\par
%%%%%%%%%%%%%%%%%%%%%%%%%%%%%%%%%%%%%%%%%%%%%%%%%%%%%%%%%%%%%%%%%%%%%%%%%%%%%%%%%%%%%%%%%
%%%%%%%%%%%%%%%%%%%%%%%%%%%%%%%%%%%%%%%%%%%%%%%%%%%%%%%%%%%%%%%%%%%%%%%%%%%%%%%%%%%%%%%%%%%
We use the following criteria to check whether $L_{k}$ is contained in $\widetilde{k}$.
%%%%%%%%%%%%%%%%%%%%%%%%%%%%%%%%%%%%%%%%%%%%%%%%%%%%%%%%%%%%%%%%%%%%%%%%%%%%%%%%%%%%%%%%%%
%%%%%%%%%%%%%%%%%%%%%%%%%%%%%%%%%%%%%%%%%%%%%%%%%%%%%%%%%%%%%%%%%%%%%%%%%%%%%%%%%%%%%%%%%%%%%%
\begin{lem}[{\cite[Corollary of Proposition 6.B]{Mi} and \cite[Theorem 1]{Br0}}]\label{brink}
%\\
Let $k=\mathbb{Q}(\sqrt{-d})$ with a square-free positive integer $d$.\\
$\mathrm{(i)}$ 
If $p=3$ and $d \not\equiv 3 \mathrm{~mod~}9$, then
$L_k \subset \widetilde{k}$
if and only if 
the class number of $\mathbb{Q}(\sqrt{3d})$ is not divisible by $3$.\\
$\mathrm{(ii)}$ Assume that $p\geq 5$ and
that $k$ has the same $p$-class number as $k({\zeta_p})$,
then
$L_k \subset \widetilde{k}$,
where $\zeta_p$ is a primitive $p$-th root of unity.
\end{lem}
%%%%%%%%%%%%%%%%%%%%%%%%%%%%%%%%%%%%%%%%%%%%%%%%%%%%%%%%%%%%%%%%%%%%%%%%%%%%%%%%%%%%%%%%%%%%%%%%%%%%%%%%%
%%%%%%%%%%%%%%%%%%%%%%%%%%%%%%%%%%%%%%%%%%%%%%%%%%%%%%%%%%%%%%%%%%%%%%%%%%%%%%%%%%%%%%%%%%%%%%%%%%%%%%%%%
As we stated in Section \ref{In}, we can prove GGC from weak GGC if the following condition is satisfied.
%%%%%%%%%%%%%%%%%%%%%%%%%%%%%%%%%%%%%%%%%%%%%%%%%%%%%%%%%%%%%%%%%%%%%%%%%%%%%%%%%%%%%%%%%%%%%%%%%%%
%%%%%%%%%%%%%%%%%%%%%%%%%%%%%%%%%%%%%%%%%%%%%%%%%%%%%%%%%%%%%%%%%%%%%%%%%%%%%%%%%%%%%%%%%%%%%%%%%%%%%
\begin{prop}[{\cite[Proposition $3.1$]{Fuj10}}]\label{Fu}
Suppose that weak {\rm{GGC}} holds for $p$ and $k$
and that
$\mathrm{char}_{\mathbb{Z}_p[[\mathrm{Gal}(k_{\infty}^c/k)]]}((X_{\widetilde{k}})_{\mathrm{Gal}(\widetilde{k}/k_{\infty}^c)})$
is a prime ideal.
Then {\rm{GGC}} holds for $p$ and $k$.
\end{prop}
%%%%%%%%%%%%%%%%%%%%%%%%%%%%%%%%%%%%%%%%%%%%%%%%%%%%%%%%%%%%%%%%%%%%%%%%%%%%%%%%%%%%%%%%%%%%%%%%%%%%%%%%%%%%%%%%%%%%%%%%%%%%%%%%%%
%%%%%%%%%%%%%%%%%%%%%%%%%%%%%%%%%%%%%%%%%%%%%%%%%%%%%%%%%%%%%%%%%%%%%%%%%%%%%%%%%%%%%%%%%%%%%%%%%%%%%%%%%%%%%%%%%%%%%%%%%%%%%%%%%%
%For the prime number $p$,
Using Propositions \ref{ex1}, \ref{Fu}, and Lemma \ref{brink} above,
we can prove GGC for $(p,k)=(3,\mathbb{Q}(\sqrt{-971}))$, $(3,\mathbb{Q}(\sqrt{-17291}))$,
and $(5,\mathbb{Q}(\sqrt{-2239}))$.~\\
%%%%%%%%%%%%%%%%%%%%%%%%%%%%%%%%%%%%%%%%%%%%%%%%%%%%%%%%%%%%%%%%%%%%%%%%%%%%%%%%%%%%%%%%%%%%%%%%%%%%%%%%%%%%%%%%%%%%%%%%%%%%%%%%%%%%%%%%
%%%%%%%%%%%%%%%%%%%%%%%%%%%%%%%%%%%%%%%%%%%%%%%%%%%%%%%%%%%%%%%%%%%%%%%%%%%%%%%%%%%%%%%%%%%%%%%%%%%%%%%%%%%%%%%%%%%%%%%%%%%%%%%%%%%%%%%%
\begin{Ex}~\label{971}
\begin{rm}
Put $p=3$ and put $k=\mathbb{Q}(\sqrt{-971})$.
Then $3$ splits completely in $k$.
By PARI/GP, we have
$
%\mathrm{Cl}_{k} \otimes \mathbb{Z}_3
A_k
\cong \mathbb{Z}/3\mathbb{Z}$.
%where $\mathrm{Cl}_{k}$ is the ideal class group of $k$.
We can check that $L_k \subset \widetilde{k}$ by Lemma \ref{brink}.
Indeed, the class number of $\mathbb{Q}(\sqrt{2913})$ is $7$. 
%\cite[Corollary of Proposition 6.B]{Mi}.
Hence, we have $\lambda(N_{\infty}/k)=0$.
%(see for example \cite[Remark 3.2 (ii)]{Mu}).
Using PARI/GP,
%Mizusawa's program \cite{Miz},
we have 
\[
h(T)\equiv T^2+64638T~~\mathrm{mod}~3^{11}.
\]
By Hensel's lemma, there exists an integer $\alpha \in \mathbb{Z}_3$ such that
\[
h(T)=T(T+\alpha ),
\]
where $\alpha$ satisfies that $\alpha \equiv 486$ mod $3^6$.
Hence the characteristic ideal of $X_{k_{\infty}^c}$
%$\mathrm{char}_{\mathbb{Z}_p[[\mathrm{Gal}(k_{\infty}^c/k)]]}(X_{k_{\infty}^c})$
has a square-free generator.
Let $f(S,T)$ be the same power series defined in Lemma \ref{f(S,T)}.
By Proposition \ref{char cyc},
we have $\mathrm{ord}_3(f(0,0))=\mathrm{ord}_3(\alpha)=5$.
%%%%%%%%%%%%%%%%%%%%%%%%%%%%%%%%%%%%%%%%%%%%%%%%%%%%%%%%%%%%%%%%%%%%%%%%%%%%%%%%%%%%%%%%%%%%%%%%%%%%%%%%%55
%Since $\mathrm{ord}_3(g_0(0))\neq [\mathrm{Gal}(\widetilde{k}/k):\mathfrak{D}_{\mathfrak{p}}]$,
Hence
$k$ is not $p$-split $p$-rational
and
weak GGC holds by Proposition \ref{ex1}.
%by Theorem \ref{main thm}.
Furthermore, by Proposition \ref{Fu}, GGC holds since we have $\lambda(k_{\infty}^c/k)=2$.
%since $\lambda(k_{\infty}^c/k)=2$.
%Since $\mathrm{ord}_p(\alpha)=5$, we see $\mathfrak{D_p}$ is not a normal subgroup of
%$\mathrm{Gal}(\widetilde{k}/\mathbb{Q$
\end{rm}
\end{Ex}~\par
%%%%%%%%%%%%%%%%%%%%%%%%%%%%%%%%%%%%%%%%%%%%%%%%%%%%%%%%%%%%%%%%%%%%%%%%%%%%%%%%%%%%%%%%%
%%%%%%%%%%%%%%%%%%%%%%%%%%%%%%%%%%%%%%%%%%%%%%%%%%%%%%%%%%%%%%%%%%%%%%%%%%%%%%%%%%%%%%%%%
%%%%%%%%%%%%%%%%%%%%%%%%%%%%%%%%%%%%%%%%%%%%%%%%%%%%%%%%%%%%%%%%%%%%%%%%%%%%%%%%%%%%%%%%%
%%%%%%%%%%%%%%%%%%%%%%%%%%%%%%%%%%%%%%%%%%%%%%%%%%%%%%%%%%%%%%%%%%%%%%%%%%%%%%%%%%%%%%%%%%%%%%%%%%%%%
\begin{Ex}~\label{17291}
\begin{rm}
Put $p=3$ and put $k=\mathbb{Q}(\sqrt{-17291})$.
Then $3$ splits completely in $k$.
By PARI/GP, we have $A_k \cong \mathbb{Z}/3\mathbb{Z}$.
We can check that $L_k \subset \widetilde{k}$ by Lemma \ref{brink}.
Indeed, the class number of $\mathbb{Q}(\sqrt{51873})$ is $1$. 
%\cite[Theorem 1]{Br0}.
Hence we have $\lambda(N_{\infty}/k)=0$.
Furthermore we have 
\[
h(T) \equiv T^4 + 405T^3 + 72T^2 + 522T ~~\mathrm{mod}~3^7.
\]
%We see that $T^3 + 405T^2 + 72T + 522$ mod $3^7$ irreducible.
%By Hensel's lemma, 
This implies that there exists an irreducible polynomial $g(T) \in \mathbb{Z}_p[[T]]$ such that
$h(T)=Tg(T).$
Hence the characteristic ideal of $X_{k_{\infty}^c}$
%$\mathrm{char}_{\mathbb{Z}_p[[\mathrm{Gal}(k_{\infty}^c/k)]]}(X_{k_{\infty}^c})$
has a square-free generator.
Let $f(S,T)$ be the same power series defined in Lemma \ref{f(S,T)}.
By Proposition \ref{char cyc},
we have $\mathrm{ord}_3(f(0,0))=\mathrm{ord}_3(522)=2$.
%By Hensel's lemma, there exists an integer $\alpha \in \mathbb{Z}_p$ such that
%\[
%h(T)=T k(T),
%\]
%where $k(T)$ satisfies that $\alpha \equiv 243$ mod $3^6$.
%Let $f(S,T)$ be the same power series defined in Lemma \ref{f(S,T)}.
Hence
$k$ is not $p$-split $p$-rational
and
weak GGC holds by Proposition \ref{ex1}.
%by Theorem \ref{main thm}.
Furthermore, by Proposition \ref{Fu}, GGC holds 
since $g(T)$ is irreducible.
%%%%%%%%%%%%%%%%%%%%%%%%%%%%%%%%%%%%%%%%%%%%%%%%%%%%%%%%%%%%%%%%%%%%%%%%%%%%%%%%%%%%%%%%%%%%%%%%%%%%%%%%%55
\end{rm}
\end{Ex}~\par
%%%%%%%%%%%%%%%%%%%%%%%%%%%%%%%%%%%%%%%%%%%%%%%%%%%%%%%%%%%%%%%%%%%%%%%%%%%%%%%%%%%%%%%%%
%%%%%%%%%%%%%%%%%%%%%%%%%%%%%%%%%%%%%%%%%%%%%%%%%%%%%%%%%%%%%%%%%%%%%%%%%%%%%%%%%%%%%%%%%
%%%%%%%%%%%%%%%%%%%%%%%%%%%%%%%%%%%%%%%%%%%%%%%%%%%%%%%%%%%%%%%%%%%%%%%%%%%%%%%%%%%%%%%%%
%%%%%%%%%%%%%%%%%%%%%%%%%%%%%%%%%%%%%%%%%%%%%%%%%%%%%%%%%%%%%%%%%%%%%%%%%%%%%%%%%%%%%%%%%%%%%%%%%%%%%
\begin{Ex}~\label{2239}
\begin{rm}
Put $p=5$ and put $k=\mathbb{Q}(\sqrt{-2239})$.
Then $5$ splits completely in $k$.
By PARI/GP, we have $A_k \cong \mathbb{Z}/5\mathbb{Z}$.
We can check that $L_k \subset \widetilde{k}$ by Lemma \ref{brink}.
Indeed, the class number of $k(\zeta_5)$ is $560$. 
%\cite[Theorem 1]{Br0}.
Hence we have $\lambda(N_{\infty}/k)=0$.
Furthermore we have 
\[
h(T)\equiv T^2+3100T~~\mathrm{mod}~5^5.
\]
By Hensel's lemma, there exists an integer $\alpha \in \mathbb{Z}_5$ such that
\[
h(T)=T(T+\alpha ),
\]
where $\alpha$ satisfies that $\alpha \equiv 100$ mod $5^3$.
Hence the characteristic ideal of $X_{k_{\infty}^c}$
%$\mathrm{char}_{\mathbb{Z}_p[[\mathrm{Gal}(k_{\infty}^c/k)]]}(X_{k_{\infty}^c})$
has a square-free generator.
Let $f(S,T)$ be the same power series defined in Lemma \ref{f(S,T)}.
By Proposition \ref{char cyc},
we have $\mathrm{ord}_5(f(0,0))=\mathrm{ord}_5(\alpha)=2$.
%By Hensel's lemma, there exists an integer $\alpha \in \mathbb{Z}_p$ such that
%\[
%h(T)=T k(T),
%\]
%where $k(T)$ satisfies that $\alpha \equiv 243$ mod $3^6$.
%Let $f(S,T)$ be the same power series defined in Lemma \ref{f(S,T)}.
Hence
$k$ is not $p$-split $p$-rational.
From Proposition \ref{ex1},
weak GGC for $p$ and $k$ holds.
%by Theorem \ref{main thm}.
Furthermore, by Proposition \ref{Fu}, GGC holds for $k$ 
since $\lambda(k_{\infty}^c/k)=2$.
%%%%%%%%%%%%%%%%%%%%%%%%%%%%%%%%%%%%%%%%%%%%%%%%%%%%%%%%%%%%%%%%%%%%%%%%%%%%%%%%%%%%%%%%%%%%%%%%%%%%%%%%%55
\end{rm}
\end{Ex}~\par
%%%%%%%%%%%%%%%%%%%%%%%%%%%%%%%%%%%%%%%%%%%%%%%%%%%%%%%%%%%%%%%%%%%%%%%%%%%%%%%%%%%%%%%%%
%%%%%%%%%%%%%%%%%%%%%%%%%%%%%%%%%%%%%%%%%%%%%%%%%%%%%%%%%%%%%%%%%%%%%%%%%%%%%%%%%%%%%%%%%
Let $k_{\infty}/k$ be any $\mathbb{Z}_p$-extension.
For a non-negative integer $n$, 
let $k_n$ be the $n$-th layer of $k_{\infty}/k$ and $\mathcal{O}_{k_n}$ the ring of integers of $k_n$.
We denote by $A_{k_n}$ the $p$-Sylow subgroup of the ideal class group of $k_{n}$.
We define a homomorphism $i_{k_n}:A_{k} \rightarrow A_{k_n}$
by sending the ideal class of $\mathfrak{a}$ to the class of $\mathfrak{a}\mathcal{O}_{k_n}$
for every ideal $\mathfrak{a}$ of $k$.\par
We use the following criteria to check whether $\lambda(N_{\infty}/k)=0$.
%%%%%%%%%%%%%%%%%%%%%%%%%%%%%%%%%%%%%%%%%%%%%%%%%%%%%%%%%%%%%%%%%%%%%%%%%%%%%%%%%%%%%%%%%%%%%%%%%%%%%%%%%%%%%%%%%%%%%%%%%%%%%%%%%%%%%%%%
%%%%%%%%%%%%%%%%%%%%%%%%%%%%%%%%%%%%%%%%%%%%%%%%%%%%%%%%%%%%%%%%%%%%%%%%%%%%%%%%%%%%%%%%%%%%%%%%%%%%%%%%%%%%%%%%%%%%%%%%%%%%%%%%%%%%%%%%
\begin{prop}[{\cite[Proposition 1.C]{Mi}}]\label{Mi-cri}
%\\
The Iwasawa $\lambda$-invariant of $k_{\infty}/k$ is zero
if and only if
$A_{k} = \displaystyle{\bigcup_{n\geq 0} \mathrm{Ker}(i_n)}$, in other words,
every ideal class of $A_{k}$ capitulates in $k_{\infty}$.
%the following two conditions are equivalent:\\
%$\mathrm{(i)}$ \\
%$\mathrm{(ii)}$ We have $M_{\mathfrak{p^\ast}}(\widetilde{k}) \neq L_{\widetilde{k}}$.
\end{prop}
%%%%%%%%%%%%%%%%%%%%%%%%%%%%%%%%%%%%%%%%%%%%%%%%%%%%%%%%%%%%%%%%%%%%%%%%%%%%%%%%%%%%%%%%%%%%%%%%%%%
%%%%%%%%%%%%%%%%%%%%%%%%%%%%%%%%%%%%%%%%%%%%%%%%%%%%%%%%%%%%%%%%%%%%%%%%%%%%%%%%%%%%%%%%%%%%%%%%%%%%%
%%%%%%%%%%%%%%%%%%%%%%%%%%%%%%%%%%%%%%%%%%%%%%%%%%%%%%%%%%%%%%%%%%%%%%%%%%%%%%%%%%%%%%%%%%%%%%%%%%%%%%%%%%%%%%%%%%%%%%%%%%%%%%%%%%%%%%%%
\begin{prop}[{\cite[Theorem 1]{Fuku}}]\label{Fukuda}
Let $k_{\infty}/k$ be any $\mathbb{Z}_p$-extension.
Let $c$ be a non-negative integer such that any prime ideal of $k$ which is ramified
in $k_{\infty}/k$ totally ramified in $k_{\infty}/k_{c}$.
If there exists an integer $n \geq c$ such that $\# A_{k_{n}} = \# A_{k_{n+1}}$,
then $\# A_{k_{m}} = \# A_{k_{n}}$ for all $m \geq n$.
In particular, we see that $\mu(k_{\infty}/k)=\lambda(k_{\infty}/k)=0$.
\end{prop}~\\
%%%%%%%%%%%%%%%%%%%%%%%%%%%%%%%%%%%%%%%%%%%%%%%%%%%%%%%%%%%%%%%%%%%%%%%%%%%%%%%%%%%%%%%%%%%%%%%%%%%%
We can prove GGC for $(p,k)=(3,\mathbb{Q}(\sqrt{-5069}))$ by Theorem \ref{main thm} and propositions above.~\\
%%%%%%%%%%%%%%%%%%%%%%%%%%%%%%%%%%%%%%%%%%%%%%%%%%%%%%%%%%%%%%%%%%%%%%%%%%%%%%%%%%%%%%%%%%%%%%%%%%%%%
\begin{Ex}~\label{5069}
\begin{rm}
Put $p=3$ and put $k=\mathbb{Q}(\sqrt{-5069})$.
Then $3$ splits completely in $k$.
Using PARI/GP, we can check that
$A_k \cong \mathbb{Z}/3\mathbb{Z} \oplus \mathbb{Z}/3\mathbb{Z}$
and that
$\mathrm{Cl}_{k}$ is generated by the classes of the prime ideals lying above $5$ and $19$.
Using Fujii's criteria (\cite[Lemma 4.3]{Fuj10}), we have
$\mathrm{Gal}(\widetilde{k}/L_k \cap \widetilde{k}) \cong \mathbb{Z}/3\mathbb{Z}$.
%Hence $N_1=k_1^{a}$.
%%%%%%%%%%%%%%%%%%%%%%%%%%%%%%%%%%%%%%%%%%%%%%%%%%%%%%%
%By using Mizusawa's program \cite{Miz},
Then we have 
\[
h(T)\equiv T(T +1989) ~~\mathrm{mod}~3^7.
\]
%By Hensel's lemma,
There exists an integer $\alpha \in \mathbb{Z}_3$ such that
$f(T)=T(T+\alpha)$
%where $\alpha$ satisfies
with $\alpha \equiv 45$ mod $3^5$.
Hence the characteristic ideal of $X_{k_{\infty}^c}$ 
%$\mathrm{char}_{\mathbb{Z}_p[[\mathrm{Gal}(k_{\infty}^c/k)]]}(X_{k_{\infty}^c})$
is generated by a square-free generator.\par
%%%%%%%%%%%%%%%%%%%%%%%%%%%%%%%%%%%%%%%%%%%%%%%%%%%%%%%%%%%%%%%%%%%%%%%%%%%%%%%%
Next we prove that $\lambda(N_{\infty}/k)=0$.
Using criteria in \cite[Theorem $2$]{Br}, we obtain,
\begin{eqnarray*}
x^6 - 2x^5 - 33x^4 - 70x^3 + 5462x^2 - 38784x + 83808
\end{eqnarray*}
as a defining polynomial of $k_{1}^{a}$ over $\mathbb{Q}$.
Then we see that $3$ splits completely in $k_1^{a}$.
%%%%%%%%%%%%%%%%%%%%%%%%%%%%%%%%%%%%%%%%%%%%%%%%%%%%%%%%%%%%%%%%%%%%%%%%%%%%%%%%
Furthermore, we obtain
%using PARI/GP, we can calculate the second layer of $N_{\infty}/k$:
\begin{eqnarray*}
&~&x^{18} - 30x^{16} - 604x^{15} - 7705x^{14} - 45030x^{13} + 13536x^{12}+ 2239858x^{11}\\
&+& 34485072x^{10}+305384350x^9 + 2021144762x^8 + 11280115140x^7\\
&+& 53497712456x^6 + 198613209492x^5+ 669494384271x^4 + 1792552517970x^3\\
&+& 3479920548408x^2 + 6719313226248x + 8090872119549
\end{eqnarray*}
as a defining polynomial of $N_{2}$ over $\mathbb{Q}$.
%%%%%%%%%%%%%%%%%%%%%%%%%%%%%%%%%%%%%%%%%%%%%%%%%%%%%%%%%%%%%%%%%%%%%%%%%%%%%%%%%%%%
Then all prime ideal classes above $5$ and $19$ become principal in $N_2$.
Hence $\lambda(N_{\infty}/k)=0$ by Proposition \ref{Mi-cri}.
We can also check this from Proposition \ref{Fukuda}.
Indeed, using PARI/GP, we see that
any prime ideal of $k$ which is ramified
in $N_{\infty}/k$ totally ramified in $N_{\infty}/k_{1}^a$.
%one of the prime ideals of $k_1^{a}$ lying above $3$ ramifies
%in $N_2$.
Furthermore, $k_1^a$ and $N_2$ has the same $p$-class number
since $\mathrm{ord}_3(\# A_{k_1^a})=\mathrm{ord}_3(\# A_{N_2})=3$.
Hence we get $\lambda(N_{\infty}/k)=0$ by Proposition \ref{Fukuda}.
%Let $f(S,T)$ be the same power series defined in Lemma \ref{f(S,T)}.
%By Proposition \ref{char cyc},
%we have $\mathrm{ord}_3(g_0(0))=1$.
%Thus we get
%$\mathrm{ord}_3(g_0(0)) \neq [L_k :k]$.
%Hence $k$ is not $p$-split $p$-rational.
Therefore weak GGC holds by Theorem \ref{main thm}.
%Furthermore it is easy to generalize Proposition \ref{Fu} to the case where $\lambda(k_{\infty}^c/k)=3$.
%Hence, GGC for $3$ and $k$ holds.
Furthermore, by Proposition \ref{Fu}, GGC holds
since $\lambda(k_{\infty}^c/k)=2$.
%Since $\mathrm{ord}_p(\alpha)=5$, we see $\mathfrak{D_p}$ is not a normal subgroup of
%$\mathrm{Gal}(\widetilde{k}/\mathbb{Q$
\end{rm}
\end{Ex}
%%%%%%%%%%%%%%%%%%%%%%%%%%%%%%%%%%%%%%%%%%%%%%%%%%%%%%%%%%%%%%%%%%%%%%%%%%%%%%%%%%%%%%
%%%%%%%%%%%%%%%%%%%%%%%%%%%%%%%%%%%%%%%%%%%%%%%%%%%%%%%%%%%%%%%%%%%%%%%%%%%%%%%%%%%%%%%%%
%%%%%%%%%%%%%%%%%%%%%%%%%%%%%%%%%%%%%%%%%%%%%%%%%%%%%%%%%%%%%%%%%%%%%%%%%%%%%%%%%%%%%%%%%
\section*{Acknowledgements}
The author is sincerely grateful to
%my supervisor,
Professor Masato Kurihara and Professor Yoshitaka Hachimori
for their kind encouragement and valuable advice.
%%%%%%%%%%%%%%%%%%%%%%%%%%%%%%%%%%%%%%%%%%%%%%%%%%%%%%%%%%%%%%%%%%%%%%%%%%%%%%%%%%%%%%%%%
%who encouraged the author and gave him valuable comments on this paper.
%the study of $\mathrm{GGC}$.
The author expresses his thanks to Professor
Satoshi Fujii and Doctor John Minardi for their beautiful studies in
\cite{Fuj10, Fuj} and \cite{Mi}.
%%%%%%%%%%%%%%%%%%%%%%%%%%%%%%%%%%%%%%%%%%%%%%%%%%%%%%%%%%%%%%%%%%%%%%%%%%%%%%%%%%%%%%%%
The author would like to express his deepest appreciation
to Toshihiro Shirakawa
for his valuable advice about calculation
%his valuable discussion.
in Section \ref{example}.
%owe his methods. 
%The author would like to thank Professor Masato Kurihara for his valuable suggestions,
%which led me to get the results in this paper.
%The author is partially supported by JSPS Core-to-core program, Foundation of a Global Research Cooperative Center in Mathematics focused on Number Theory.
%Thank you very much.
%%%%%%%%%%%%%%%%%%%%%%%%%%%%%%%%%%%%%%%%%%%%%%%%%%%%%%%%%%%%%%%%%%%%%%%%%%%%%%%%%%%%%%%%%%%%%%%%%%%%%%%%%%
%%%%%%%%%%%%%%%%%%%%%%%%%%%%%%%%%%%%%%%%%%%%%%%%%%%%%%%%%%%%%%%%%%%%%%%%%%%%%%%%%%%%%%%%%%%%%%%%%%%%%%%%%%%%%%%%%
%%%%%%%%%%%%%%%%%%%%%%%%%%%%%%%%%%%%%%%%%%%%%%%%%%        thebibliography                      %%%%%%%%%%%%%%%%%%%
%%%%%%%%%%%%%%%%%%%%%%%%%%%%%%%%%%%%%%%%%%%%%%%%%%%%%%%%%%%%%%%%%%%%%%%%%%%%%%%%%%%%%%%%%%%%%%%%%%%%%%%%%%%%%%%%%%%%%%%%%%%%%
%%%%%%%%%%%%%%%%%%%%%%%%%%%%%%%%%%%%%%%%%%%%%%%%%%%%%%%%%%%%%%%%%%%%%%%%%%%%%%%%%%%%%%%%%%%%%%%%%%%%%%%%%%%%%%%%

\end{document}